\documentclass[11pt]{amsart}

\usepackage{amsmath, amssymb, amsthm,mathtools,tikz,tikz-cd,graphicx,color,stmaryrd,hyperref,enumerate}

\usepackage[shortlabels]{enumitem}

\usepackage[utf8]{inputenc}

\usepackage[margin=1in]{geometry}
\usepackage{cleveref}

\usepackage[english]{babel}

\hypersetup{
    colorlinks,
    citecolor=blue,
    filecolor=blue,
    linkcolor=blue,
    urlcolor=blue
}

\newcommand{\C}{{\mathbb C}}

\newcommand{\mcD}{\mathcal{D}}
\newcommand{\mcG}{\mathcal{G}}

\newcommand{\bbG}{\mathbb{G}}

\newcommand{\mcH}{\mathcal{H}}

\newcommand{\G}{\mathbb{G}}

\newcommand{\p}{\mathfrak{p}}

\newcommand\ii{{\mathrm{i}}}

\newcommand{\St}{{\mathrm{St}}}

\newcommand{\st}{{\mathrm{st}}}
\newcommand{\unst}{{\mathrm{unst}}}

\newcommand\rG{{\mathrm{G}}}

\newcommand{\Oo}{\mathcal{O}}

\newcommand{\sgn}{\operatorname{sgn}}

\newcommand{\sign}{\operatorname{sgn}}

\newcommand{\Ind}{\operatorname{Ind}}

\newcommand{\diag}{\operatorname{diag}}

\newcommand{\bG}{\mathbf{G}}
\newcommand{\cA}{\mathcal{A}}
\newcommand{\cB}{\mathcal{B}}

\newcommand{\R}{{\mathbb R}}
\newcommand{\rZ}{{\mathrm{Z}}}
\newcommand{\Z}{{\mathbb Z}}
\newcommand{\F}{{\mathbb F}}
\newcommand{\bH}{{\mathbb H}}

\newcommand{\cG}{{\mathcal G}}
\newcommand{\cL}{{\mathcal L}}

\newcommand{\Q}{{\mathbb Q}}
\newcommand{\cO}{{\mathfrak o}}

\newcommand{\cD}{{\mathcal D}}

\newcommand{\ft}{{\mathfrak t}}

\newcommand{\fs}{{\mathfrak{s}}}

\newcommand{\one}{{\mathbf 1}}

\newcommand{\fB}{\mathfrak{B}}

\newcommand{\cE}{\mathcal{E}}

\newcommand{\res}{\mathrm{res}}
\newcommand{\Irr}{\operatorname{Irr}}

\newcommand{\Lie}{\operatorname{Lie}}

\newcommand{\A}{\mathbb{A}}

\newcommand{\princ}{\mathrm{princ}}
\newcommand{\cusp}{\mathrm{cusp}}
\newcommand{\tor}{\mathrm{tor}}

\newcommand{\der}{{\mathrm{der}}}

\newcommand{\red}{{\mathrm{red}}}

\newcommand{\enh}{{\mathrm{e}}}

\newcommand{\LLC}{{\mathrm{LLC}}}

\newcommand{\SL}{{\mathrm{SL}}}

\newcommand{\GL}{{\mathrm{GL}}}
\newcommand{\SO}{{\mathrm{SO}}}
\newcommand{\tS}{{\mathrm{S}}}
\newcommand{\GO}{{\mathrm{GO}}}

\newcommand{\tO}{{\mathrm{O}}}
\newcommand{\tU}{{\mathrm{U}}}

\newcommand{\GU}{{\mathrm{GU}}}
\newcommand{\SU}{{\mathrm{SU}}}

\newcommand{\PGL}{{\mathrm{PGL}}}
\newcommand{\GSp}{{\mathrm{GSp}}}
\newcommand{\GSpin}{{\mathrm{GSpin}}}
\newcommand{\Spin}{{\mathrm{Spin}}}
\newcommand{\PGSpin}{{\mathrm{PGSpin}}}

\newcommand{\Sp}{{\mathrm{Sp}}}

\newcommand{\Cent}{{\mathrm{Z}}}

\DeclareMathOperator{\Gal}{Gal}

\DeclareMathOperator{\cInd}{c-Ind}

\DeclareMathOperator{\tr}{tr}
\DeclareMathOperator{\fdeg}{fdeg}

\DeclareMathOperator{\Nm}{Nm}
\DeclareMathOperator{\Ch}{Ch}

\DeclareMathOperator{\Sc}{Sc}

\numberwithin{equation}{subsection}
\newtheorem{thm}[equation]{Theorem}
\newtheorem{property}[equation]{Property}
\newtheorem{lemma}[equation]{Lemma}

\newtheorem{cor}[equation]{Corollary}

\numberwithin{equation}{subsection}
\newtheorem{prop}[equation]{Proposition}

\theoremstyle{definition}
\numberwithin{equation}{subsection}
\newtheorem{remark}[equation]{Remark}

\newtheorem{numberedparagraph}[equation]{}

\newtheorem{conj}[equation]{Conjecture}

\newtheorem{defn}[equation]{Definition}

\newtheorem{prop-def}{Proposition-Definition}

\makeatletter
\def\@settitle{\begin{center}
    \normalfont
\uppercasenonmath\@title
  \@title
  \end{center}
}

\title[Explicit Local Langlands correspondence for $\GSp_4$ and $\Sp_4$]{\textbf{The explicit Local Langlands Correspondence for $\GSp_4$, $\Sp_4$ and stability}\\\MakeLowercase{\textit{with an application to Modularity Lifting}}}

\author{Kenta Suzuki}
\address{M.I.T., 77 Massachusetts Avenue,
Cambridge, MA, USA}
\email{kjsuzuki@mit.edu}

\author{Yujie Xu}
\address{M.I.T., 77 Massachusetts Avenue,
Cambridge, MA, USA}
\email{yujiexu@mit.edu}

\begin{document}

\maketitle

{\centering\footnotesize {\fontfamily{qbk}\selectfont \textit{Dedicated to Professor George Lusztig, with admiration.}}\par}

\begin{abstract}
We give a purely local proof of the explicit Local Langlands Correspondence for $\GSp_4$ and $\Sp_4$. Moreover, we give a unique characterization in terms of stability of $L$-packets and other properties. 
Finally, in the appendix, we give an application of our explicit local Langlands correspondence to modularity lifting. 
\end{abstract}

\tableofcontents

\section{Introduction}

Let $F$ be a non-archimedean local field and $\bG$ a connected reductive algebraic group over $F$. Let $G^{\vee}$ be the group of $\C$-points of the reductive group whose root datum is the coroot datum of $\bG$. The Local Langlands Conjecture predicts a surjective map\footnote{To avoid overunning the margins, we use abbreviations ``irred.'' for ``irreducible'', ``repres.'' for ``representations'', ``iso.'' for ``isomorphism'', ``cont.'' for ``continuous'' and ``conj.'' for ``conjugacy''.

For simplicity, we only state the conjecture for quasisplit $p$-adic groups in the introduction, which is sufficient for our current paper.} 
\[\begin{Bmatrix}\text{irred. smooth}\\
\text{repres. } \pi\text{ of }\bG(F)\end{Bmatrix}/\text{iso.}\;\longrightarrow
\; \begin{Bmatrix}L\text{-parameters
}
\\ \text{i.e.~cont. homomorphisms}
\\
\varphi_{\pi}\colon W_F\times\SL_2(\C)\to G^{\vee}\rtimes W_F\end{Bmatrix}/\text{$G^\vee$-conj.},\]
where $W_F$ is the Weil group of $F$. The fibers of this map, called \textit{$L$-packets}, are expected to be finite. 
In order to obtain a bijection between the group side and the Galois side, the above Conjecture was later \textit{enhanced} (\'a la Deligne, Vogan, Lusztig etc.).~On the Galois side, one considers \textit{enhanced $L$-parameters}. 

\noindent Many cases of the Local Langlands Conjecture have been established, most notably: 
\begin{itemize}
    \item for $\GL_n(F)$:~\cite{Harris-Taylor, Henniart-LLC-GLn,Scholze-LLC};
    \item for $\SL_n(F)$:~\cite{Hiraga-Saito} for $\mathrm{char}(F)=0$ and~\cite{ABPS-SL} for $\mathrm{char}(F)>0$~(see also \cite{Gelbart-Knapp-SLn-1981,Gelbart-Knapp-SLn-1982});
    \item quasi-split classical groups for $F$ of characteristic zero: \cite{Arthur-classification, Moeglin} etc.
    \item exceptional group $G_2$: \cite{LLC-G2}
\end{itemize}
    For classical groups, the main methods in literature are either (1) to classify representations of these groups in terms of representations of the general linear groups via twisted endoscopy, and to compare the stabilized twisted trace formula on the general linear group side and the stabilized (twisted) trace formula on the classical group side, or (2) to use the theta correspondence. 

    In \cite{aubert-xu-Hecke-algebra}, the second author took a completely different approach 
    to the construction of explicit Local Langlands Correspondences for $p$-adic reductive groups via reduction to LLC for supercuspidal representations of proper Levi subgroups. This strategy was then applied in \cite{LLC-G2} to construct the explicit Local Langlands Correspondence for $p$-adic $G_2$, which is the first known case in literature of Local Langlands Correspondence for exceptional groups. In \cite{G2-stability}, the authors uniquely characterize the Local Langlands Correspondence constructed in \cite{LLC-G2} using an extension of the \textit{atomic stability} property of $L$-packets as formulated by DeBacker, Kaletha etc.~(see for example \cite[Conjecture 2.2]{kaletha-survey-2022}), which is a generalization of the stability property in \cite{DeBacker-Reeder}. 
    To do this, we compute the coefficients of certain local character expansions building on methods in \cite{harish-chandra-local-character,Debacker-Sally-germs,Barbasch-Moy-LCE}.

    In this article, we apply this general strategy pioneered in \cite{LLC-G2, G2-stability} and construct the explicit Local Langlands Correspondence for the symplectic groups $\GSp_4$ and $\Sp_4$ over an arbitrary non-archimedean local field of residual characteristic $\neq 2$, with explicit $L$-packets and explicit matching between the group and Galois sides.

    More precisely, we use a combination of the Langlands-Shahidi method, (extended affine) Hecke algebra techniques, Kazhdan-Lusztig theory and generalized Springer correspondence--in particular, the AMS Conjecture on cuspidal support \cite[Conjecture 7.8]{AMS18}. For \textit{intermediate series}, i.e.~Bernstein series with supercuspidal support ``in between'' a torus and $G$ itself, we use our previous result on Hecke algebra isomorphisms and local Langlands correspondence for Bernstein series obtained in \cite{aubert-xu-Hecke-algebra}. For principal series (i.e.~Bernstein series with supercuspidal support in a torus), we improve on previous works we use 
    \cite{Roche-principal-series, Reeder-isogeny,ABPS-KTheory,Ram} to match the group and Galois sides.

    For supercuspidal representations, we make explicit the theory of \cite{Kal-reg,Kaletha-nonsingular} for the non-singular supercuspidal representations and their $L$-packets. 
    For 
    \textit{singular}\footnote{which we define to be simply the ones that are \textit{not non-singular in the sense of \cite{Kaletha-nonsingular}}} supercuspidal representations, which are 
    not covered in \textit{loc.cit.}
    , we use \cite[Conjecture 7.8]{AMS18} (see Property \ref{property:AMS-conjecture-7.8}) to exhibit them in \textit{mixed} $L$-packets with non-supercuspidal representations. These mixed $L$-packets are drastically different from the supercuspidal $L$-packets of \cite{Kal-reg,Kaletha-nonsingular}. 

Furthermore, our LLC satisfies several expected properties, including the expectation that $\Irr(S_{\varphi})$ parametrizes the internal structure of the $L$-packet $\Pi_{\varphi}(G)$, where $S_{\varphi}$ is the component group of the centralizer of the (image of the) $L$-parameter $\varphi$. Moreover, we explicitly compute the coefficients of local character expansions of Harish-Chandra characters for certain non-supercuspidal representations (see $\S$\ref{stability-section}), which allows us to give a unique characterization of our LLC using \textit{stability} for $L$-packets.

Finally, \textit{explicit} Local Langlands Correspondences (e.g.~explicit Kazhdan--Lusztig triples) have important applications to number theory, such as to the Taylor--Wiles methods and modularity lifting theorems. In Appendix \ref{TW-appendix}, we record such an application, following \cite{BCGP,Thorne-BKvanishing, whitmore2022taylorwiles}.

\subsection{Main results}
We now state our main results. Let $\Irr^{\fs}(G)$ be the Bernstein series attached to the inertial class $\fs=[L,\sigma]$ (for more details, see \cite[(3.3.2)]{LLC-G2}
). 
Let $\Phi_e(G)$ denote the set of $G^\vee$-conjugacy classes of enhanced $L$-parameters for $G$. Let $\Phi_e^{\fs^{\vee}}(G)\subset \Phi_e(G)$ be the Bernstein series on the Galois side, whose cuspidal support lies in $\fs^{\vee}=[L^{\vee},(\varphi_{\sigma},\rho_{\sigma})]$, i.e.~the image under LLC for $L$ of $\fs$ (for more details, see \cite[\S 2.4]{LLC-G2}
). 
For any $\fs=[L,\sigma]_G\in\fB(G)$, the LLC for $L$ given by $\sigma\mapsto(\varphi_\sigma,\rho_\sigma)$ is expected to induce a bijection (see \cite[Conjecture~2]{AMS18} and Conjecture~\ref{conj:matching}):
\begin{equation}\label{Bernstein-block-isom-intro}
    \mathrm{Irr}^{\fs}(G)\xrightarrow{\sim}\Phi_\enh^{\fs^{\vee}}(G). 
\end{equation}
For the group $\GSp_4$ and $\Sp_4$, by \cite[Main Theorem]{aubert-xu-Hecke-algebra}, we have such a bijection \eqref{Bernstein-block-isom-intro} for each Bernstein series $\mathrm{Irr}^{\fs}(G)$ of \textit{intermediate series}. On the other hand, the analogous bijection to \eqref{Bernstein-block-isom-intro} holds for \textit{principal series} Bernstein blocks thanks to \cite{Roche-principal-series,Reeder-isogeny,ABPS-KTheory, AMS18}. 

Let $G=\GSp_4(F)$ or $\Sp_4(F)$, and $p\neq 2$. Combined with the detailed analysis in all of $\mathsection$\ref{sec:group-side} through $\mathsection$\ref{stability-section}, we explicitly construct the Local Langlands Correspondence
\begin{equation} \label{main-thm-bijection}
\begin{split}
   \LLC\colon \mathrm{Irr}(G)&\xrightarrow{1\text{-}1}\Phi_\enh(G)\\
    \pi &\mapsto (\varphi_{\pi},\rho_{\pi}),
\end{split}\end{equation}
and obtain the following result (see Theorem~\ref{main-thm}). 
\begin{thm} 
The explicit Local Langlands Correspondence  ~\eqref{main-thm-bijection} verifies $\Pi_{\varphi_\pi}(G)\xrightarrow{\sim}\Irr(S_{\varphi_\pi})$ for any $\pi\in \Irr(G)$, and satisfies \eqref{Bernstein-block-isom-intro} for any $\fs\in\fB(G)$,
where $\fs^\vee=[L^\vee,(\varphi_\sigma,\rho_\sigma)]_{G^\vee}$, as well as
a list of properties (see \S\ref{section-properties}) that uniquely characterize our correspondence. 
\end{thm}
\noindent In other words,  
\begin{enumerate}
    \item to each explicitly described $\pi\in\Irr(G)$, we attach an explicit $L$-parameter $\varphi_{\pi}$ and determine its enhancement $\rho_{\pi}$ explicitly; 
    \item to each $\varphi\in \Phi(G)$, we describe (the shape of) 
    its $L$-packet $\Pi_{\varphi}(G)$, and give an internal parametrization in terms of $\rho\in\Irr(S_{\varphi})$;
    \item Moreover, for non-supercuspidal representations, we specify the precise parabolic induction that it occurs in. 
\end{enumerate}

\noindent\textit{Acknowledgements.} 
Y.X.~was supported by NSF grant DMS~2202677. K.S.~was partially supported by MIT-UROP. The authors would like to thank Jack Thorne for bringing their attention to the applications of explicit LLC to modularity lifting and for help with references. 
The authors would also like to thank Anne-Marie Aubert, Stephen DeBacker, Dick Gross, Ju-Lee Kim, George Lusztig, Maarten Solleveld, Loren Spice, Jack Thorne and Dmitri Whitmore for helpful conversations related to this project. The authors would like to thank George Lusztig and Wei Zhang for their continued interest and encouragement. 

The authors would like to thank Maarten Solleveld for helpful feedback on a previous version of this paper. The authors would like to thank MIT for providing an intellectually stimulating working environment.

\section{Preliminaries}

Let $F$ be a nonarchimedean local field. Let $J_2:=\begin{pmatrix}&1\\1\end{pmatrix}$ and $\beta:=\begin{pmatrix}&J_2\\-J_2\end{pmatrix}$. Consider the following groups
\begin{align*}
\Sp_4&:=\{g\in\GL_4(F):{^T}g\beta g=\beta\}\\
\GSp_4&:=\{g\in\GL_4(F):{^T}g\beta g=\mu(g)\beta,\text{ for some }\mu(g)\in F^\times\}.
\end{align*}
In particular, there is an exact sequence
\(
1\to \Sp_4(F)\to \GSp_4(F)\xrightarrow{\mu}F^\times\to 1.
\)
The Langlands dual groups are $\GSp_4^\vee=\GSpin_5(\C)$ and $\Sp_4^\vee=\PGSpin_5(\C)\cong \SO_5(\C)$. Here $\GSpin_5:=(\GL_1\times\Spin_5)/\mu_2$ where $\mu_2$ is diagonally embedded as in \cite[Definition~2.3]{asgari}.

\subsection{Root datum}
The following are the data for the root datum for $\Sp_4,\GSp_4$ \cite{Tadic-symplectic-1994,asgari,asgari-shahidi}, of type $C_2$. We also realize everything in terms of the torus $T=\{(a_1,a_2,b_2,b_1):a_1b_1=a_2b_2=\mu\}$.
\begin{itemize}
    \item For $\Sp$, the lattice is $X^*(T):=\Z\{\epsilon_1,\epsilon_2\}$, the roots are $\Delta:=\{\pm\epsilon_1\pm\epsilon_2,\pm2\epsilon_1,\pm2\epsilon_2\}$, and the simple roots are $\{\epsilon_1-\epsilon_{2},2\epsilon_2\}$.
    \item For $\GSp$, the lattice is $X^*(T):=\Z\{\epsilon_0,\epsilon_1,\epsilon_2\}$, the roots are $\Delta:=\{\pm\epsilon_1\pm\epsilon_2\}\cup\{\pm(\epsilon_0-2\epsilon_1),\pm(\epsilon_0-2\epsilon_2),\pm(\epsilon_0-\epsilon_1-\epsilon_2)\}$, and the simple roots are $\{\epsilon_1-\epsilon_2,2\epsilon_2-\epsilon_0\}$.
\end{itemize}
Here, $\epsilon_i(a_1,a_2,b_1,b_2)=a_i$ for $i=1,2$ and $\epsilon_0(a_1,a_2,b_2,b_1)=\mu$.

The root groups are given by:
\begin{align*}
U_{\epsilon_i-\epsilon_j}&=\begin{pmatrix}1+x\one_{ij}\\&1-x\one_{n+1-j,n+1-i}\end{pmatrix}\\
U_{\epsilon_i+\epsilon_j}&=\begin{pmatrix}1&x(\one_{i,n+1-j}+\one_{j,n+1-i})\\&1\end{pmatrix}\\
U_{2\epsilon_i}&=\begin{pmatrix}1&x\one_{i,n+1-i}\\&1\end{pmatrix}\\
U_{-\epsilon_i-\epsilon_j}&=\begin{pmatrix}1\\x(\one_{n+1-i,j}+\one_{n+1-j,i})&1\end{pmatrix}\\
U_{-2\epsilon_i}&=\begin{pmatrix}1&\\x\one_{n+1-i,i}&1\end{pmatrix},
\end{align*}
where $\one_{ij}$ is the matrix with a single one in the $(i,j)$-component.

Letting $\alpha:=\epsilon_1-\epsilon_2$ and $\beta:=2\epsilon_2$ (or $2\epsilon_2-\epsilon_0$, for $\GSp$), and $\delta:=-2\epsilon_1$ (or $\epsilon_0-2\epsilon_1$ for $\GSp$) we obtain:
\begin{center}  \includegraphics[scale=0.05]{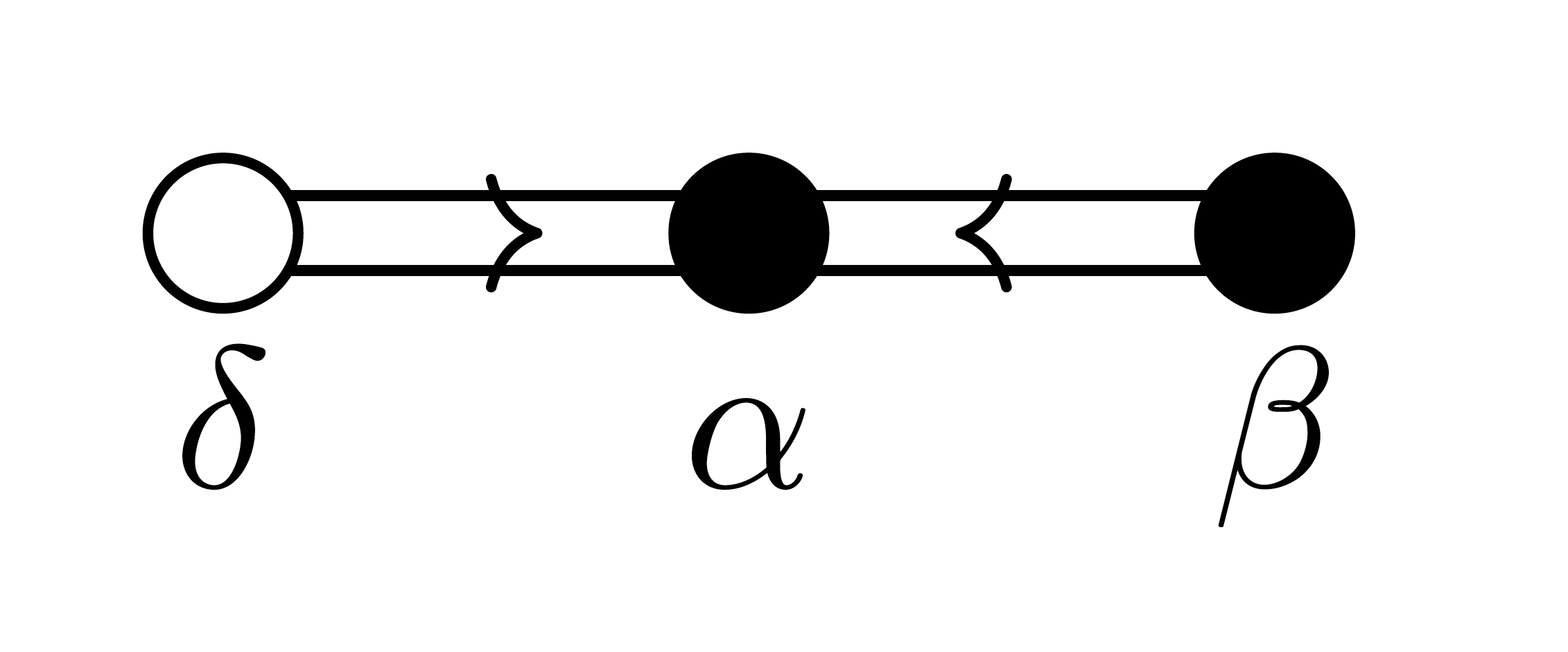}\end{center}
Coroots are given by $\alpha^\vee:=\frac{2(\alpha,-)}{(\alpha,\alpha)}$. For $\Sp_4$ and $\GSp_4$, they are of type $B_2$:
\begin{itemize}
    \item $X_*(T):=\Z\{\epsilon_1^\vee,\epsilon_2^\vee\}$, and the simple coroots are $\{\alpha^\vee:=\epsilon_1^\vee-\epsilon_2^\vee,\beta^\vee:=\epsilon_2^\vee\}$. 
    \item $X_*(T):=\Z\{\epsilon_0^\vee,\epsilon_1^\vee,\epsilon_2^\vee\}$, and the simple coroots are $\{\alpha^\vee:=\epsilon_1^\vee-\epsilon_2^\vee,\beta^\vee:=\epsilon_2^\vee\}$.
\end{itemize}

Here, $\epsilon_0^\vee(t_0)\epsilon_1^\vee(t_1)\epsilon_2^\vee(t_2)=(t_1,t_2,t_0t_2^{-1},t_0t_1^{-1})$.

The Dynkin diagram is:
\begin{center}  \includegraphics[scale=0.05]{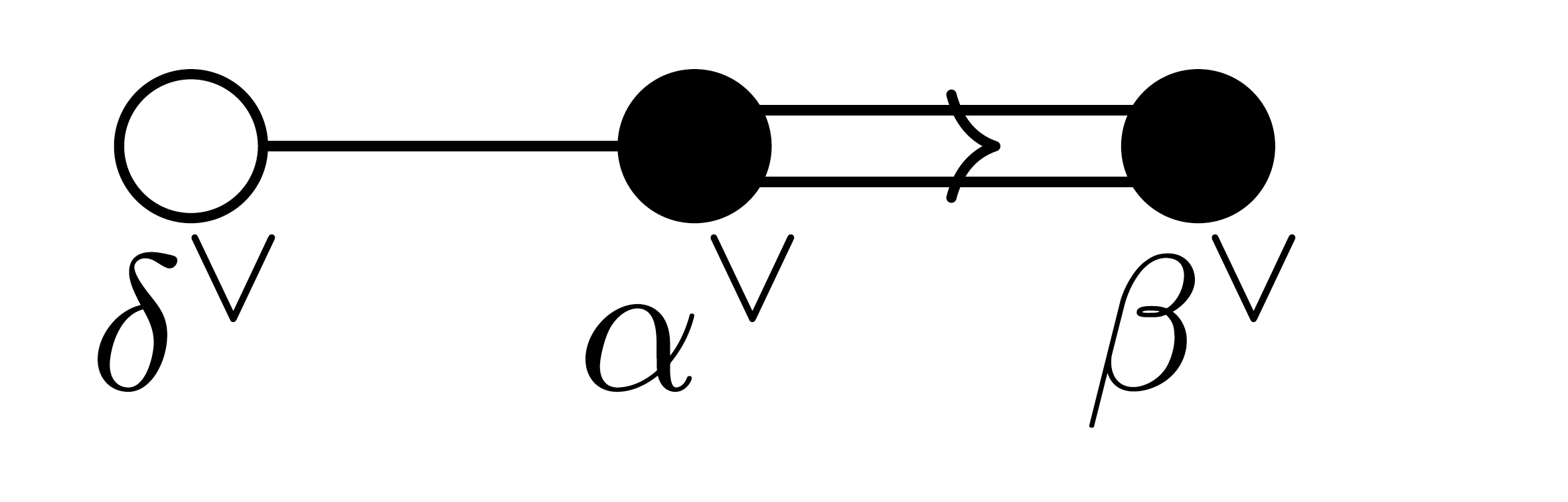}\end{center}
\begin{remark}\label{gsp4-self-dual}
$\GSp_4$ happens to be self-dual, under the following isomorphism:
\begin{align}
    X^*(T)=\Z\{\epsilon_0,\epsilon_1,\epsilon_2\}&\to X_*(T)=\Z\{\epsilon_0^\vee,\epsilon_1^\vee,\epsilon_2^\vee\}\nonumber\\
    \epsilon_0&\mapsto -2\epsilon_0^\vee-\epsilon_1^\vee-\epsilon_2^\vee\label{gsp4-self-dual-iso}\\
    \epsilon_1&\mapsto-\epsilon_0^\vee\nonumber\\
    \epsilon_2&\mapsto-\epsilon_0^\vee-\epsilon_2^\vee,\nonumber
\end{align}
where $\alpha_1\mapsto\alpha_2^\vee$ and $\alpha_2\mapsto\alpha_1^\vee$, and its inverse is given by $\epsilon_0^\vee\mapsto-\epsilon_1,\epsilon_1^\vee\mapsto\epsilon_1+\epsilon_2-\epsilon_0,\epsilon_2^\vee\mapsto\epsilon_1-\epsilon_2$.
\end{remark}

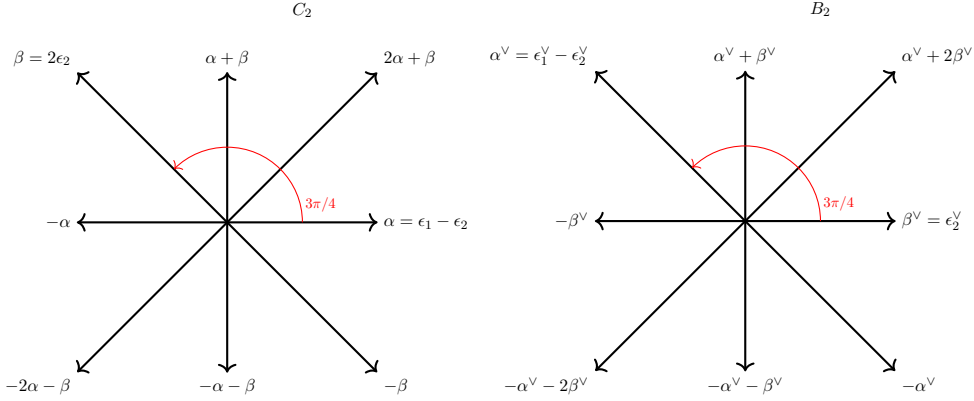
\begin{figure}
  \begin{tikzpicture}
    \foreach\ang in {90,180,270,360}{
     \draw[->,black!80!black,thick] (0,0) -- (\ang:2cm);
    }
    \foreach\ang in {45,135,225,315}{
     \draw[->,black!80!black,thick] (0,0) -- (\ang:2.82cm);
    }
    \draw[red,->](1,0) arc(0:135:1cm)node[pos=0.1,right,scale=0.5]{$3\pi/4$};
    \node[anchor=west,scale=0.6] at (2,0) {$\alpha=\epsilon_1-\epsilon_2$};
    \node[anchor=south east,scale=0.6] at (-2,2) {$\beta=2\epsilon_2$};
    \node[anchor=east,scale=0.6] at (-2,0) {$-\alpha$};
    \node[anchor=south,scale=0.6] at (0,2) {$\alpha+\beta$};
    \node[anchor=south west,scale=0.6] at (2,2) {$2\alpha+\beta$};
    \node[anchor=north east,scale=0.6] at (-2,-2) {$-2\alpha-\beta$};
    \node[anchor=north,scale=0.6] at (0,-2) {$-\alpha-\beta$};
    \node[anchor=north west,scale=0.6] at (2,-2) {$-\beta$};
    \node[anchor=north,scale=0.6] at (1,3) {$C_2$};
  \end{tikzpicture}
  \begin{tikzpicture}
    \foreach\ang in {90,180,270,360}{
     \draw[->,black!80!black,thick] (0,0) -- (\ang:2cm);
    }
    \foreach\ang in {45,135,225,315}{
     \draw[->,black!80!black,thick] (0,0) -- (\ang:2.82cm);
    }
    \draw[red,->](1,0) arc(0:135:1cm)node[pos=0.1,right,scale=0.5]{$3\pi/4$};
    \node[anchor=west,scale=0.6] at (2,0) {$\beta^\vee=\epsilon_2^\vee$};
    \node[anchor=south east,scale=0.6] at (-2,2) {$\alpha^\vee=\epsilon_1^\vee-\epsilon_2^\vee$};
    \node[anchor=east,scale=0.6] at (-2,0) {$-\beta^\vee$};
    \node[anchor=south,scale=0.6] at (0,2) {$\alpha^\vee+\beta^\vee$};
    \node[anchor=south west,scale=0.6] at (2,2) {$\alpha^\vee+2\beta^\vee$};
    \node[anchor=north east,scale=0.6] at (-2,-2) {$-\alpha^\vee-2\beta^\vee$};
    \node[anchor=north,scale=0.6] at (0,-2) {$-\alpha^\vee-\beta^\vee$};
    \node[anchor=north west,scale=0.6] at (2,-2) {$-\alpha^\vee$};
    \node[anchor=north,scale=0.6] at (1,3) {$B_2$};
  \end{tikzpicture}\caption{Root diagram for $B_2=C_2$}\label{c2-root-system}
  \end{figure}
  
\begin{remark}\label{nilpotent-sp4}
By the exceptional isomorphism $B_2=C_2$, we have the following description of nilpotent orbits in $\GSp_4$ and $\Sp_4$ (see \cite[Thm~5.1.2,5.1.3]{collingwood}):
{\begin{center}
\begin{tabular}{ |c|c|c|c|c| } 
 \hline
 &Orbits of $B_2$& Orbits of $C_2$&Roots of $C_2$&Levi subgroup of $\GSp_4$\\ \hline
 regular&$[5]$&$[4]$&$e_{\alpha}+e_\beta$&$\GSp_4$\\
subregular&$[3,1^2]$&$[2^2]$&$e_\beta$&$\GL_2\times\GSp_0$\\
 minimal&$[2^2,1]$&$[2,1^2]$&$e_{\alpha}$&$\GL_1\times\GSp_2$\\
 zero&$[1^5]$&$[1^4]$&$0$&$T$\\
 \hline
\end{tabular}
\end{center}}
\end{remark}

For later use (e.g.~\S \ref{stability-section}), we record the following table \ref{table-Weyl-group-conj-class} for Weyl group conjugacy classes for $\GSp_4$ and $\Sp_4$. 
\begin{table}[ht]
\begin{tabular}{ |c|c| } 
 \hline
 names & cycle types\\ \hline
 $e$& $(1)(1)$\\
 $A_1$& $(1)(\overline{1})$\\
 $\widetilde{A}_1$& $(2)$\\
 $A_1\times A_1$& $(\overline{1})(\overline{1})$ \\
 $C_2$& $(\overline{2})$\\
 \hline
\end{tabular}
\vskip0.2cm
\caption{\label{table-Weyl-group-conj-class} Weyl group conjugacy classes }
\end{table}
We will also need the following picture of a $C_2$-apartment in the building $\mathcal{B}(\GSp_4)$. 
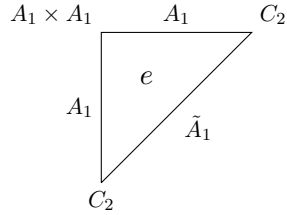
\begin{figure}
\begin{tikzpicture}
    \draw[black]{}(0,-2)--(0,0)--(2,0)--(0,-2);
    \node[anchor=south, scale=0.8] at (1,0) {$A_1$};
    \node[anchor=east, scale=0.8] at (0,-1) {$A_1$};
    \node[anchor=north west,scale=0.8] at (1,-1){$\tilde{A}_1$};
    \node[anchor=north,scale=0.8] at (0,-2){$C_2$};
    \node[anchor=south east, scale=0.8] at (0,0) {$A_1\times A_1$};
    \node[anchor=south west, scale=0.8] at (2,0) {$C_2$};
    \node[] at (0.6,-0.6){$e$};
\end{tikzpicture}
\caption{\label{C2-apartment-figure} The apartment in $\mathcal{B}(\GSp_4)$}
\end{figure}

\subsection{Levi subgroups}\label{Levi-section} 
The Levi subgroups of $\GSp_4$ (resp., $\Sp_4$) are:
\begin{itemize}
    \item $\GSp_4$ (resp., $\Sp_4$)
    \item $\GL_2\times\GSp_0$ (resp., $\GL_2\times\Sp_0$). Explicitly, it is $\GSp_4\cap (\GL_1\times\GL_2\times\GL_1)$.
    \item $\GL_1\times\GSp_2$ (resp., $\GL_1\times\Sp_2$). Explicitly, it is $\GSp_4\cap (\GL_2\times\GL_2)$.
    \item $\GL_1\times\GL_1\times\GSp_0$ (resp., $\GL_1\times\GL_1\times\Sp_0$), the maximal torus.
\end{itemize}
Given representations $\pi$ of $\GL_2$ and characters $\chi_1,\chi_2,\chi_3$, we let $\pi\rtimes\chi_1$, $\chi_1\rtimes\pi$, and $\chi_1\times\chi_2\rtimes\chi_3$ be the (normalized) parabolic induction from $\GL_2\times\GSp_0$, $\GL_2\times\GSp_2$, and $\GL_1\times\GL_1\times\GSp_0$, respectively, using notation from \cite[\S1]{sally-tadic}.

  \begin{remark}\label{gsp4-levi-dual}
  The exceptional isomorphism $\GSp_4^\vee\cong\GSp_4$ of Remark~\ref{gsp4-self-dual} gives the identifications between the dual Levi subgroups:
  \begin{align*}
      \GSp_4^\vee&\cong\GSp_4\\
      (\GL_2\times\GSp_0)^\vee&\cong\GL_1\times\GSp_2\\
      (\GL_1\times\GSp_2)^\vee&\cong\GL_2\times\GSp_0\\
      (\GL_1\times\GL_1\times\GSp_0)^\vee&\cong\GL_1\times\GL_1\times\GSp_0.
  \end{align*}
  \end{remark}
  
  \begin{remark}[LLC for Levis of $\GSp_4(F)$]\label{gsp4-levi-llc-remark}
By Remark~\ref{gsp4-self-dual}, the LLC for the maximal torus $T$ is given as:
\begin{align*}
\hom(W_F,T(\C))&\cong\Irr(T)\\
(\chi_1(w),\chi_2(w),\chi_0\chi_2^{-1}(w),\chi_0\chi_1^{-1}(w))&\mapsto\widehat\chi_0^{-1}\widehat\chi_1\widehat\chi_2\otimes\widehat\chi_1\widehat\chi_2^{-1}\otimes\widehat\chi_1^{-1}.
\end{align*}
Similarly, the LLC for the Levi $\GL_2(F)\times\GSp_0(F)\subset\GSp_4(F)$ is given by:
\begin{align*}
\hom(W_F\times\SL_2(\C),\GL_1(\C)\times\GSp_2(\C))&\cong\Irr(\GL_2(F)\times\GSp_0(F))\\
(\rho\otimes\varphi)&\mapsto(\widehat\rho\otimes\pi_\varphi^\vee)\boxtimes\widehat\rho^{-1},
\end{align*}
where $\pi_\varphi$ is the image of $\varphi$ under the LLC for $\GL_2(F)$. Finally, the LLC for the Levi $\GL_1(F)\times\GSp_2(F)\subset\GSp_4(F)$ is given by:
\begin{align*}
\hom(W_F\times\SL_2(\C),\GL_2(\C)\times\GSp_0(\C))&\cong\Irr(\GL_1(F)\times\GSp_2(F))\\
(\varphi\otimes\rho)&\mapsto(\widehat\rho^{-1}\omega_{\pi_\varphi})\boxtimes\pi_\varphi^\vee,
\end{align*}
where $\omega_{\pi_\varphi}=\widehat{\det(\varphi)}$ is the central character of $\pi_\varphi$.
\end{remark}

\subsection{Parahoric subgroups}\label{parahoric-section}
Types of the reductive quotient of maximal parahoric subgroups are given by deleting a node from the extended Dynkin diagram. We fix a standard choice of parahoric subgroups, with roots as indicated by Figure~\ref{parahorics-diagram}. For $\GSp_4(F)$, the vertices $\beta$ and $\delta$ are in the same orbit in the building:
        \begin{itemize}
        \item Removing $\delta$ (or $\beta$) gives the Dynkin diagram $C_2$, giving the parahoric subgroup $\GSp_4(\cO_F)$ with reductive quotient $\GSp_4(k)$.
        \item Removing $\alpha$ gives the Dynkin diagram $A_1\sqcup A_1$, giving the groups
        \[
        G_\alpha:=\GSp_4(F)\cap\begin{pmatrix}
        \cO&\cO&\cO&\p^{-1}\\
        \p&\cO&\cO&\cO\\
        \p&\cO&\cO&\cO\\
        \p&\p&\p&\cO
        \end{pmatrix}\supset G_{\alpha+}=\GSp_4(F)\cap\begin{pmatrix}
        1+\p&\cO&\cO&\cO\\
        \p&1+\p&\p&\cO\\
        \p&\p&1+\p&\cO\\
        \p^2&\p&\p&1+\p
        \end{pmatrix},
        \]
        with reductive quotient $\GSp_{2,2}(k):=\{(g,h)\in \GSp_2\times \GSp_2:\mu(g)=\mu(h)\}$.
    \end{itemize}
    Similarly, for $\Sp_4(F)$, we have:
    \begin{itemize}
        \item Removing $\delta$ gives the Dynkin diagram $C_2$, giving the parahoric subgroup $\Sp_4(\cO_F)$ with reductive quotient $\Sp_4(k)$.
        \item Removing $\beta$ gives the Dynkin diagram $C_2$, giving the parahoric subgroup \[
        \Sp_4(F)\cap\begin{pmatrix}M_2(\cO)&M_2(\p^{-1})\\M_2(\p)& M_2(\cO)\end{pmatrix}=\begin{pmatrix}\varpi^{-1} I_2\\&I_2\end{pmatrix}\Sp_4(\cO_F)\begin{pmatrix}\varpi I_2\\&I_2\end{pmatrix}\]
        with reductive quotient $\Sp_4(k)$. Here the matrix $\mathrm{diag}(\varpi I_2,I_2)$ is in $\GSp_4(F)$, but not $\Sp_4(F)$.
        \item Removing $\alpha$ gives the Dynkin diagram $A_1\sqcup A_1$, giving the group
        \[
        G_\alpha:=\Sp_4(F)\cap\begin{pmatrix}
        \cO&\cO&\cO&\p^{-1}\\
        \p&\cO&\cO&\cO\\
        \p&\cO&\cO&\cO\\
        \p&\p&\p&\cO
        \end{pmatrix}\supset G_{\alpha+}=\Sp_4(F)\cap\begin{pmatrix}
        1+\p&\cO&\cO&\cO\\
        \p&1+\p&\p&\cO\\
        \p&\p&1+\p&\cO\\
        \p^2&\p&\p&1+\p
        \end{pmatrix},
        \]
        with reductive quotient $\Sp_2(k)\times \Sp_2(k)$. 
    \end{itemize}
    However, note that the isomorphism of $G_\alpha/G_{\alpha+}$ with $\GSp_{2,2}(k)$ (resp., $\Sp_2(k)\times\Sp_2(k)$) above are non-canonical (i.e., depend on a choice of a uniformizer $\varpi$.) To make these isomorphisms more canonical, consider the endoscopic subgroup $H:=\Cent_G(s)$ with $s=\diag(1,-1,-1,1)$ which is isomorphic to $\GSp_{2,2}(F)$ (resp., $\Sp_{2,2}(F)$):
    \begin{align*}
        \GSp_{2,2}(F)&\xrightarrow\sim H\\
        (\begin{pmatrix}a_1&b_1\\c_1&d_1\end{pmatrix},\begin{pmatrix}a_2&b_2\\c_2&d_2\end{pmatrix})&\mapsto\begin{pmatrix}a_2&&&b_2\\&a_1&b_1\\&c_1&d_1\\c_2&&&d_2\end{pmatrix}
    \end{align*}
    Now there is a canonical isomorphism of $G_\alpha/G_{\alpha+}$ with the reductive quotient of the parahoric subgroup
    \[
    H_\alpha:=\{(g,h)\in M_2(\cO)\times \begin{pmatrix}\cO&\p^{-1}\\\p&\cO\end{pmatrix}:\det(g)=\det(h)\in\cO^\times\}.
    \]
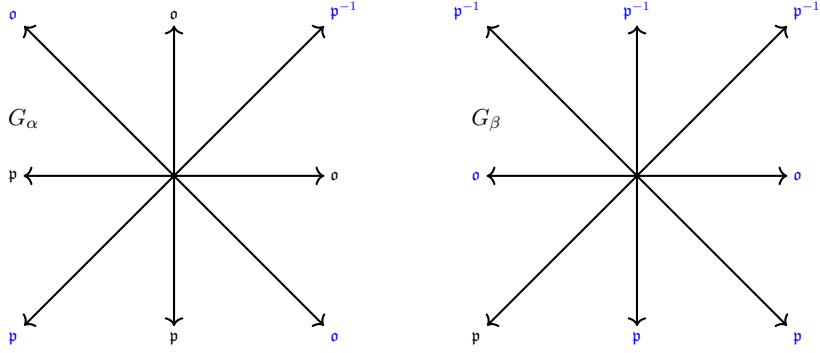
\begin{figure}
  \begin{tikzpicture}
    \foreach\ang in {90,180,270,360}{
     \draw[->,black!80!black,thick] (0,0) -- (\ang:2cm);
    }
    \foreach\ang in {45,135,225,315}{
     \draw[->,black!80!black,thick] (0,0) -- (\ang:2.82cm);
    }
    \node[anchor=west,scale=0.6] at (2,0) {$\cO$};
    \node[blue,anchor=south east,scale=0.6] at (-2,2) {$\cO$};
    \node[anchor=east,scale=0.6] at (-2,0) {$\p$};
    \node[anchor=south,scale=0.6] at (0,2) {$\cO$};
    \node[blue,anchor=south west,scale=0.6] at (2,2) {$\p^{-1}$};
    \node[blue,anchor=north east,scale=0.6] at (-2,-2) {$\p$};
    \node[anchor=north,scale=0.6] at (0,-2) {$\p$};
    \node[blue,anchor=north west,scale=0.6] at (2,-2) {$\cO$};
    \node[anchor=north,scale=0.8] at (-2,1) {$G_\alpha$};
  \end{tikzpicture}\hspace{1cm} \begin{tikzpicture}
    \foreach\ang in {90,180,270,360}{
     \draw[->,black!80!black,thick] (0,0) -- (\ang:2cm);
    }
    \foreach\ang in {45,135,225,315}{
     \draw[->,black!80!black,thick] (0,0) -- (\ang:2.82cm);
    }
    \node[blue,anchor=west,scale=0.6] at (2,0) {$\cO$};
    \node[blue,anchor=south east,scale=0.6] at (-2,2) {$\p^{-1}$};
    \node[blue,anchor=east,scale=0.6] at (-2,0) {$\cO$};
    \node[blue,anchor=south,scale=0.6] at (0,2) {$\p^{-1}$};
    \node[blue,anchor=south west,scale=0.6] at (2,2) {$\p^{-1}$};
    \node[anchor=north east,scale=0.6] at (-2,-2) {$\p$};
    \node[blue,anchor=north,scale=0.6] at (0,-2) {$\p$};
    \node[blue,anchor=north west,scale=0.6] at (2,-2) {$\p$};
    \node[anchor=north,scale=0.8] at (-2,1) {$G_\beta$};
  \end{tikzpicture}
  \caption{Parahoric subgroups $G_\alpha$ and $G_\beta$}\label{parahorics-diagram}
  \end{figure}

\section{The group side}\label{sec:group-side}
\subsection{Supercuspidal representations}
\subsubsection{\textbf{Depth-zero supercuspidal representations of $\Sp_4,\GSp_4$}}\label{d0s}
\begin{numberedparagraph}
First we recall a few general facts on depth-zero supercuspidals. 
Let $\pi$ be an irreducible depth-zero supercuspidal representation of $G$. Then there exists a vertex $x\in\mathcal{B}_{\red}(G,F)$ and an irreducible cuspidal representation $\tau$ of $\mathbb{G}_x(\F_q)$, such that the restriction of $\pi$ to $G_{x,0}$ contains the inflation of $\tau$ (see \cite[\S1-2]{Morris-ENS} or \cite[Proposition~6.6]{Moy-PrasadII}). The normalizer $N_G(G_{x,0})$ of $G_{x,0}$ in $G$ is a totally disconnected group that is compact mod center, which by \cite[proof of (5.2.8)]{Bruhat-Tits-II} coincides with the fixator $G_{[x]}$ of $[x]$ under the action of $G$ on the reduced building of $\mathbf{G}$. Then $\pi$ is compactly induced from a representation of $N_G(G_{x,0})$, i.e.
\begin{equation} \label{eqn:depth zero supercuspidal}
\pi=\cInd_{G_{[x]}}^G(\boldsymbol{\tau}).
\end{equation}
Many properties of the representation $\pi$ is already visible from the representation $\boldsymbol{\tau}$ of $G_{[x]}$:
\begin{lemma}\label{depth-0-fdeg}\cite[Prop~3.2.4]{LLC-G2}
The formal degree of the depth-zero representation $\pi=\cInd_{G_{x,0}}^G\mathbf\tau$ is
\[
\fdeg(\pi)=\frac{q^{rk(G)/2}\dim(\tau_{unip})}{|(Z_{\G_{x,0}^\vee}(s))(\F_q)|_{p'}},
\]
where $|\cdot|_{p'}$ denotes the coprime-to-$p$ order.
\end{lemma}

The following construction gives a special class of supercuspidals, i.e.~depth-zero regular supercuspidal representations of $G$ is as in \cite[Lem~3.4.12]{Kal-reg}: 
\begin{defn}\label{regular-supercuspidal-definition}
For $S\subset G$ a maximally unramified elliptic maximal torus and $\theta\colon S(F)\to\C^\times$ a regular character of depth zero, let $\pi_{(S,\theta)}:=\cInd_{S(F)G_{x,0}}^{G(F)}(\theta\otimes\pm R_{S'}^{\overline\theta})$.
\end{defn}
One can generalize the above construction and consider a larger class of supercuspidals called ``non-singular'' supercuspidals, which are the largest class of supercuspidals living in purely supercuspidal $L$-packets (see for example \cite{LLC-G2} for more exposition).     
\end{numberedparagraph}

\begin{numberedparagraph}
More concretely, depth-zero irreducible supercuspidal representations of $G$ are parametrized by irreducible cuspidal representations of reductive quotients $\G_x$ of maximal parahorics, which can be inflated to $G_{x,0}$, and (non-uniquely) extended to $G_{[x]}$. Recall from the classical Deligne-Lusztig theory \cite[\S10]{Deligne-Lusztig} and \cite[(8.4.4)]{Lusztig-characters-Princeton-book}, we have bijections
\begin{equation}\label{deligne-lusztig-bijection}
    \Irr(\G_x)\xrightarrow{\sim}\bigsqcup_{(s)}\mathcal E(\G_x(\F_q),s)\xrightarrow{\sim}\bigsqcup_{(s)}\mathcal E(\Cent_{\G_x^\vee}(s),1),
\end{equation}
where $(s)$ runs through the conjugacy classes of semisimple elements of $\G_x^\vee$. Moreover, the bijections preserve cuspidality. We hope to see when $\bH^\vee=\Cent_{\G_x^\vee}(s)$ has a unipotent cuspidal representation. We will repeatedly use the following result:
\begin{lemma}[{\cite[Thm~3.22]{Lus78}},{\cite[8.11]{Lus77}}]\label{lus}\ 
\begin{itemize}
    \item $\SO_{2n+1}(\F_q)$ has a unique unipotent cuspidal representation exactly when $n=s^2+s$ for some integer $s\ge 1$, of dimension
    \[
    \frac{|\SO_{2n+1}(\F_q)|_{p'}q^{{2n\choose2}+{2n-2\choose2}+\cdots}}{2^n(q+1)^{2n}(q^2+1)^{2n-1}\cdots(q^{2n}+1)}.
    \]
    \item $\SO_{2n}(\F_q)$ has a unique unipotent cuspidal representation exactly when $n=4s^2$ for some $s\ge1$. The non-split form $\SO^-_{2n}(\F_q)$ has a unique unipotent cuspidal representation exactly when $n=(2s+1)^2$ for some $s\ge1$.
    \item $\GL_n$ has no unipotent cuspidal representations for any $n\ge1$.
\end{itemize}
\end{lemma}
\end{numberedparagraph}

\begin{numberedparagraph}For us, by \S\ref{parahoric-section} the reductive quotients $\G_x$ are $\Sp_4(k)$ or $\Sp_2(k)\times\Sp_2(k)$ for $G=\Sp_4(F)$ and either $\GSp_4(k)$ or $\GSp_{2,2}(k):=\{(g,h)\in \GSp_2(k)\times\GSp_2(k):\mu(g)=\mu(h)\}$ for $G=\GSp_4(F)$. Using \eqref{deligne-lusztig-bijection} we classify the cuspidal representations of these groups:

\begin{lemma}\label{gsp2,2-finite-cuspidal}
Every cuspidal representations of $\GSp_{2,2}(\F_q)$ (defined in \S\ref{parahoric-section}) is given by, for $s=(g,h)\in\GL_2(\F_q)\times\GL_2(\F_q)/\F_q^\times$ where $g$ has eigenvalues $\lambda_1,\lambda_1^q$ and $h$ has eigenvalues $\lambda_2,\lambda_2^q$ where $\lambda_1,\lambda_2\in\F_{q^2}\backslash\F_q$:
\begin{itemize}
    \item if $\lambda_1^{q-1}\ne-1$ or $\lambda_2^{q-1}\ne-1$, then
    \[
    \mathcal E(\GSp_{2,2}(\F_q),s)\cong\mathcal E(R_{\F_{q^2}/\F_q}\G_m\times R_{\F_{q^2}/\F_q}\G_m/\F_q^\times,1)=\{1\}.
    \]
    Denote such a cuspidal representation as $\overline\rho_{(\alpha,\beta)}$.
    \item if $\alpha^{q-1}=\beta^{q-1}=-1$, then
    \[
    \mathcal E(\GSp_{2,2}(\F_q),s)\cong\mathcal E(R_{\F_{q^2}/\F_q}\G_m\times R_{\F_{q^2}/\F_q}\G_m/\F_q^\times\rtimes\mu_2,1)=\{1,\sgn\}.
    \]
    Denote such cuspidal representations as $\overline\rho^+_{(\alpha,\beta)}$ and $\overline\rho^-_{(\alpha,\beta)}$.
\end{itemize}
\begin{remark}\label{rho+-characterization}
The cuspidal representations $\overline\rho^+_{(\alpha,\beta)}$ are characterized as the common irreducible constituent of $\Ind_{\SL_2\times\SL_2}^{\GL_{2,2}}(R_T^\alpha\boxtimes R_T^\beta)$ and the Gelfand-Graev representation $\Gamma_{\mathcal O}^{\GL_{2,2}}$ where $\mathcal O$ is the orbit of $(\begin{pmatrix}1&1\\&1\end{pmatrix},\begin{pmatrix}1&1\\&1\end{pmatrix})$.
The restriction of $\overline\rho^+_{(\alpha,\beta)}$ to $\SL_2(\F_q)\times\SL_2(\F_q)$ is $R_+'(\theta_0)\boxtimes R_+'(\theta_0)+R_-'(\theta_0)\boxtimes R_-'(\theta_0)$, in \cite[pg~55]{bonnafe-sl2}'s notation.
\end{remark}

\end{lemma}

\begin{lemma}\label{gsp4-finite-cuspidal}
The following are the cuspidal representations of $\GSp_4(\F_q)$:
\begin{itemize}
    \item The $q-1$ twists of the unique unipotent cuspidal, i.e., in $\mathcal E(\GSp_4,s)$ where $s\in\Cent(\GSpin_5)$.
    \item $R_T^\theta$ where $T$ is an anisotropic maximal torus and $\theta$ is a regular character.
\end{itemize}
\end{lemma}

\begin{lemma}\label{sp4-finite-cuspidal}
The following are the cuspidal representations of $\Sp_4(\F_q)$:
\begin{itemize}
    \item The unique unipotent cuspidal.
    \item For any $\alpha\in\mu_{q+1}\backslash\{\pm1\}$ then for any $s\in\SO_5(\F_q)$ with eigenvalues $1,-1,-1,\alpha^{\pm1}$,
    \[
    \mathcal E(\Sp_4,s)\cong\mathcal E(\tO_2(\F_q)\times\tU_1(\F_q),1)=\{1,\sgn\}.
    \]
There are $(q-1)/2$ such conjugacy classes, giving rise to $q-1$ representations.
\item  For $\alpha\ne\beta^{\pm1}\in\mu_{q+1}\backslash\{\pm1\}$ and $s$ with eigenvalues $1,\alpha^{\pm1},\beta^{\pm1}$,
\[
\mathcal E(\Sp_4,s)\cong\mathcal E(T,1)=\{1\}.
\]
where $T=R_{\F_{q^2}/\F_q}\G_m\times R_{\F_{q^2}/\F_q}\G_m$ is an isotropic maximal torus.
\item For $\alpha\in\mu_{q^2+1}\backslash\{\pm1\}$ and $s$ with eigenvalues $1,\alpha,\alpha^q,\alpha^{q^2},\alpha^{q^3}$,
\[
\mathcal E(\Sp_4,s)\cong\mathcal E(T,1)=\{1\},
\]
where $T=\{t\in\F_{q^4}:\mathrm{Nm}_{\F_{q^4}/\F_{q^2}}t=1\}$ is an anisotropic maximal torus.
\end{itemize}
\end{lemma}
\end{numberedparagraph}
As a consequence of Lemma~\ref{gsp2,2-finite-cuspidal}, Lemma~\ref{gsp4-finite-cuspidal}, and Lemma~\ref{sp4-finite-cuspidal}, we obtain the following classifications of depth-zero supercuspidals of $\GSp_4$ and $\Sp_4$.

\begin{numberedparagraph}
Firstly, we have the following classification of depth-zero supercuspidals of $\GSp_4(F)$. 
\begin{prop}\label{depth-zero-sc-GSp4}
The depth-zero supercuspidal representations $\pi$ of $G=\GSp_4(F)$ are:
\begin{enumerate}
    \item $\pi=\pi_{(S,\theta)}$ for some maximally unramified elliptic maximal torus $S$ and a regular character $\theta$ of depth zero. These are regular supercuspidals. 
    \item\label{depth-zero-sc-GSp4(1)} $\pi_{\beta}(\theta_{10}\otimes\chi):=\cInd_{G_\delta\Cent}^G(\theta_{10}\otimes\chi)$ where $\theta_{10}$ is inflated from the unique unipotent cuspidal $\tilde\theta_{10}$ of $\GSp_4(\F_q)$ and $\chi$ is a character of $\Cent$ such that $\chi(\Cent_{\GSp_4(\cO_F)})=1$. This is $F$-singular.
    \item\label{depth-zero-sc-GSp4(2)} $\pi_\alpha(\eta_2;\chi):=\cInd_{G_\alpha\Cent}^{\GSp_4}(\omega_\cusp^{\eta_2}\otimes\chi)$ which is a $k_F$-singular hence $F$-singular supercuspidal, where: 
    \begin{itemize}
        \item $\eta_2$ is a ramified quadratic character and $\varpi\in F$ is a uniformizer such that $\eta_2(\varpi)=1$
        \item $\omega_\cusp^{\eta_2}:=(\overline{\rho}_{(\lambda,\lambda)}^+)^{(I_2,\diag(\varpi,1))}$ where $\lambda^{q-1}=-1$, and $\overline{\rho}_{(\lambda,\lambda)}^+$ is the representation of $\GSp_{2,2}(\F_q)$ defined in Lemma~\ref{gsp2,2-finite-cuspidal}, which is viewed as a representation of $G_\alpha/G_{\alpha+}$ by conjugating by $(I_2,\diag(\varpi,1))$.
        \item $\chi$ is an unramified character of $\Cent$.
    \end{itemize}
    \item\label{depth-zero-sc-GSp4(3)} Induced representations $\pi_{(S,\theta\boxtimes\theta\otimes\chi)}$ where $S=\{(x,y)\in E^\times\times E^\times:\Nm_{E/F}x=\Nm_{E/F}y\}$ and $\theta$ is a character of $E^\times$ giving rise to a character $\theta\boxtimes\theta$ of $S$, and $\chi$ is a character of $F^\times$ viewed as a character of $S$ via $\Nm_{E/F}$. This is a $F$-singular but $k_F$-nonsingular representation.

\end{enumerate}
\end{prop}

\begin{remark}
    By Remark~\ref{rho+-characterization}, the representation $\tilde\rho(\eta_2)$ is characterized as the common irreducible constituent of the cuspidal $R_T^\theta$ with $\theta^2=1$ and the Gelfand-Graev representation corresponding to the nilpotent orbit $(\begin{pmatrix}1&1\\&1\end{pmatrix},\begin{pmatrix}1&\varpi\\&1\end{pmatrix})$ of $H_\alpha$.
\end{remark}
By Lemma \ref{depth-0-fdeg}, the formal degree of the singular depth-zero 
supercuspidal $\pi_{\beta}(\theta_{10}\otimes \chi)$ is
\begin{equation}
    \fdeg(\pi_{\beta}(\theta_{10}\otimes \chi))=\frac{q^{11/2}q(q-1)^2}{2(q-1)(q^2-1)(q^4-1)}=q^{1/2}\frac{q^6}{2(q+1)(q^4-1)},
\end{equation}
since $\dim(\tilde{\theta}_{10})=\frac{q(q-1)^2}{2}$, $\dim(\GSp_4(\F_q))=11$ and $|\GSp_4(\F_q)|=(q-1)q^4(q^2-1)(q^4-1)$ by \cite[p.75]{Carter-book}. Note that the normalization of volumes given by \cite{DeBacker-Reeder} guarantees that there is a factor of $q^{1/2}$ in the formal degree formula for $\GSp_4$. 

To compute the formal degree of $\pi_{\alpha}(\eta_2;\chi)$: since $\dim R_T^{\epsilon}=(q-1)^2$, we have $\dim(\omega_{\cusp}^{\eta_2})=\frac{1}{2}(q-1)^2$. Note that $|\GSp_{2,2}(\F_q)|=(q-1)q^2(q^2-1)^2$ and $\dim\GSp_{2,2}(\F_q)=7$. Therefore, we have 
\begin{equation}\label{sc-formal-degree-equation}
    \fdeg\left(\pi_{\alpha}(\eta_2;\chi)\right)=\frac{\frac{1}{2}(q-1)^2}{(q-1)q^2(q^2-1)^2q^{-7/2}}=q^{1/2}\frac{q}{2(q+1)(q^2-1)}.
\end{equation}

The formal degree of $\pi_{(S,\theta\boxtimes\theta\otimes\chi)}$ is similar:
\begin{equation}
    \fdeg\left(\pi_{(S,\theta\boxtimes\theta\otimes\chi)}\right)=\frac{(q-1)^2}{(q-1)q^2(q^2-1)^2q^{-7/2}}=q^{1/2}\frac{q}{(q+1)(q^2-1)}.
\end{equation}

The remaining cases of formal degrees can be easily computed as they are non-singular supercuspidals. 

\end{numberedparagraph}

\begin{numberedparagraph}
The case of $\Sp_4(F)$ is given as follows. 
\begin{prop}\label{depth-zero-sc-Sp4}
The depth-zero supercuspidal representations of $G=\Sp_4(F)$ are given by:

\begin{enumerate}
\item $\pi=\pi_{(S,\theta)}$ for some maximally unramified elliptic maximal torus $S$ and a regular character $\theta$ of depth zero. These are regular supercuspidals. 
    \item\label{depth-zero-sc-Sp4(1)} Induced representations $\cInd_{G_\beta}^G\rho$ and $\cInd_{G_\gamma}^G\rho$, where $\rho$ is one of the following representations of $\Sp_4(\F_q)$, inflated via $G_\beta,G_\gamma\to\Sp_4(\F_q)$:
    \begin{enumerate}
        \item\label{depth-zero-sc-Sp4(1)-theta10} The unique cuspidal unipotent $\theta_{10}$ of $\Sp_4(\F_q)$, which gives rise to $F$-singular representations $\pi_{\beta}(\theta_{10}):=\cInd_{G_{\beta}}^G\inf \theta_{10}$ and $\pi_{\gamma}(\theta_{10}):=\cInd_{G_{\gamma}}^G\inf \theta_{10}$ coming from $G_{\beta}$ and $G_{\gamma}$. These are $k_F$-singular hence $F$-singular supercuspidals.
        \item\label{depth-zero-sc-Sp4(1)-mu2} Corresponding to the characters $1,\sgn$ of $\G_s=\tO_2(\F_q)\times\tU_2(\F_q)$ under \eqref{deligne-lusztig-bijection}; this gives rise to $k_F$-nonsingular, and hence $F$-nonsingular. This gives a total of $q-1$ nonsingular representations; 
    \end{enumerate}
    \item\label{depth-zero-sc-Sp4(2)} Induced representations $\pi_\alpha^\pm(\eta_2):=\cInd_{G_\alpha}^{\Sp_4}(R_{\pm}'(\theta_0)\times R_{\pm}'(\theta_0)^{\diag(\varpi,1)})$ where $R_\pm'(\theta_0)$ are representations of $\SL_2(\F_q)$ defined in \cite[\S5.2]{bonnafe-sl2}. This is $k_F$-singular and hence $F$-singular. 
    \item Induced representations $\pi_\alpha(\theta):=\cInd_{G_\alpha}^{\Sp_4}(R_T^\theta\boxtimes(R_T^\theta)^{\diag(\varpi,1)})$ where $\theta$ is a regular character of an anisotropic torus $T$ of $\SL_2(\F_q)$. This is $F$-singular but $k_F$-nonsingular.
\end{enumerate}
\end{prop}
By Lemma \ref{depth-0-fdeg}, the formal degree of the singular supercuspidals $\pi_{\beta}(\theta_{10})$ and $\pi_{\gamma}(\theta_{10})$ is 
\begin{equation}
    \fdeg(\pi_{\beta}(\theta_{10}))=\fdeg(\pi_{\gamma}(\theta_{10}))=\frac{\frac{1}{2}q(q-1)^2}{q^4(q^2-1)(q^4-1)q^{-10/2}}=\frac{q^2}{2(q+1)^2(q^2+1)}
\end{equation}
since $\dim(\theta_{10})=\frac{1}{2}q(q-1)^2$ by \cite[Theorem~8.2]{Lus77}, $\dim|\Sp_4(\F_q)|=10$ and $|\Sp_4(\F_q)|=q^4(q^2-1)(q^4-1)$. 
Note that $\pi_{\beta}(\theta_{10})$ and $\pi_{\gamma}(\theta_{10})$ live in the same $L$-packet $\Pi_{\varphi(\eta)}$, mixed with two principal series representations as in \S\ref{stability-section} $L$-packet \eqref{size-4-Sp4-packet-theta-10}. 

To compute the formal degree of $\pi_{\alpha}^{\pm}(\eta_2)$: since $\dim(R_{\pm}'(\theta_0)\times R_{\pm}'(\theta_0)^{\diag(\varpi,1)})=\frac{1}{4}(q-1)^2$, $\dim(\SL_2\times\SL_2)=6$ and $|\SL_2(\F_q)\times\SL_2(\F_q)|=q^2(q^2-1)^2$, we have 
\begin{equation}
    \fdeg(\pi_{\alpha}^{\pm}(\eta_2))=\frac{\frac{1}{4}(q-1)^2}{q^2(q^2-1)^2q^{-6/2}}=\frac{q}{4(q+1)^2}.
\end{equation}
These representations live in stable mixed $L$-packets as in Corollary \ref{2x4-llc}. 

Similarly, the formal degree of $\pi_\alpha(\theta)$ is
\begin{equation}
    \fdeg(\pi_{\alpha}(\theta))=\frac{(q-1)^2}{q^2(q^2-1)^2q^{-6/2}}=\frac{q}{(q+1)^2}.
\end{equation}
This representation lives in the mixed $L$-packet in \eqref{pi-alpha-eta-mixed-packet}.
\end{numberedparagraph}

\subsubsection{\textbf{Positive-depth supercuspidal representations of $\Sp_4,\GSp_4$}} 

\begin{numberedparagraph}\label{type-datum-sec}\textbf{Type datum.}~
Recall Yu's classification of arbitrary-depth supercuspidals in terms of type datum \cite{Yu} (which was later generalized in \cite{Kim-Yu-types} to include non-supercuspidal types). 

\begin{defn} \label{defn:cuspidal datum}
A \textit{cuspidal $G$-datum} is a tuple $\cD:=(\vec G,y,\vec r, \pi^0,\vec\phi)$ consisting of 
\begin{enumerate}
\item a tamely ramified Levi sequence $\vec G=(G^0\subset G^1\subset\cdots\subset G^d=G)$ of twisted $E$-Levi subgroups of $ G$, such that $\rZ_{\bG^0}/\rZ_{\bG}$ is anisotropic;
\item a point $y$ in $\cB(G^0,F)\cap\cA(T,E)$, whose projection to the reduced building of $G^0$ is a vertex, where $T$ is a maximal torus of $G^0$ (hence of $G^i$) that splits over $E$;
\item a sequence $\vec r=(r_0,r_1,\ldots,r_d)$ of real numbers such that $0< r_0<r_1<\cdots< r_{d-1}\le r_d$ if $d>0$, and $0\le r_0$ if $d=0$;
\item an irreducible depth-zero supercuspidal representation $\rho^0$ of $K^0=G^0_{[y]}$ whose restriction to $G^0_{y,0+}$ is trivial and such that the compact induction $\cInd_{K^0}^{G^0}\rho^0$ is irreducible supercuspidal;
\item
 a sequence $\vec\phi=(\phi_0,\phi_1,\ldots,\phi_d)$ of characters, where $\phi_i$ is a character of $G^i$ which is trivial on $G^i_{y,r_i+}$ and nontrivial on $G^i_{y,r_i}$ for $0\le  i\le d-1$, such that
 \begin{itemize}
     \item $\phi_d$ is trivial on $G^d_{y,r_d+}$ and nontrivial on $G^d_{y,r_d}$ if $r_{d-1}< r_d$, and $\phi_d=1$ if $r_{d-1}=r_d$ (with $r_{-1}$ defined to be $0$). 
     \item Moreover, $\phi_i$ is $G^{i+1}$-generic of depth $r_i$ relative to $y$ in the sense of \cite[\S9]{Yu} for $0\le i\le d-1$.
 \end{itemize}
\end{enumerate}
\end{defn}
\end{numberedparagraph}

Formal degrees of arbitrary-depth tame supercuspidal representations in the sense of \cite{Yu} can be computed as in \cite[Theorem A]{Schwein}. 
Let $\mathbf{G}$ be a semisimple $F$-group, and let $\mathcal{D}$ be a cuspidal $\mathbf{G}$-datum with associated supercuspidal representation $\pi$. Let $R_i$ denote the absolute root system of $\mathbf{G}^i$, for the twisted Levi sequence $(\mathbf{G}^i)_{0\leq i\leq d}$. Let $\exp_q(t):=q^t$.
\begin{prop}
The formal degree of $\pi$ is given by
\begin{equation}
    \fdeg(\pi)=\frac{\dim\rho}{[G^0_{[y]}:G^0_{y,0+}]}\exp_q\left(\frac{1}{2}\dim G+\frac{1}{2}\dim \bbG^0_{y,0}+\frac{1}{2}\sum\limits_{i=0}^{d-1}r_i(|R_{i+1}|-|R_i|)\right).
\end{equation}
\end{prop}

\begin{remark} \label{rem:formal degrees}
The Formal Degree Conjecture of \cite{Hiraga-Ichino-Ikeda}, which describes the formal degree $\fdeg(\pi)$ of any irreducible smooth representation $\pi$ of $G$ in terms of adjoint gamma factor, has been proved for regular supercuspidal representations in \cite[Theorem~B]{Schwein}, for non-singular supercuspidal representations in \cite[Theorem~9.2]{Ohara-fdegr}, and for unipotent supercuspidal representations in \cite[Theorem~3]{FOS}.  
\end{remark}

\begin{numberedparagraph}\label{twistes-Levi-classification} \textbf{Twisted Levi Sequences.}~
We first classify twisted Levi subgroups in $\Sp_4$ and $\GSp_4$.
\begin{prop}[{\cite[page~23]{walds}}]\label{sp-tori}
Conjugacy classes of maximal tori in $\Sp_{2n}(F)$ are given by the data of:
\begin{itemize}
    \item finite extensions $F_1^\#,\dots,F_r^\#/F$;
    \item $2$-dimensional \'{e}tale $F_i^\#$-algebras $F_i$; and
\end{itemize}
such that $n=\sum_{i=1}^r[F_i:F]$. Then, $W:=\bigoplus_{i=1}^rF_i$ is a $2n$-dimensional vector space over $F$ with a symplectic form
\begin{equation}
q(\sum_{i=1}^rw_i,\sum_{i=1}^rw_i'):=\sum_{i= 1}^r\frac1{[F_i:F]}\tr_{F_i/F}(c_iw_i\overline w_i'),
\end{equation}
where elements $c_i\in F_i^\times$ are such that $\overline{c}_i=-c_i$, where $\overline{\cdot}$ denotes the unique nontrivial automorphism of $F_i/F_i^\#$. Then there is a torus (whose conjugacy class depends only on the $c_i$'s modulo $N_{F_i/F_i^\#}\G_m$)
\begin{equation}
T_{F_1/F_1^\#,\dots,F_r/F_r^\#}^{(1)}:=\prod_{i=1}^rR_{F_i^\#/F}R_{F_i/F_i^\#}^{(1)}\G_m
\end{equation}
acting component-wise on $W$. Similarly, conjugacy classes of $\GSp_{2n}(F)$ are given by the same data, giving rise to the torus
\begin{equation}
T_{F_1/F_1^\#,\dots,F_r/F_r^\#}:=\{(x_i)\in\prod_{i=1}^rR_{F_1/F}\G_m:\Nm_{F_1/F_1^\#}(x_1)=\cdots=\Nm_{F_r/F_r^\#}(x_r)\in F^\times\}.
\end{equation}
\end{prop}
For $\Sp_4(F)$, the anistropic maximal tori are thus of the following form: 
\begin{itemize}
    \item $T^{(1)}_{F_1/F,F_2/F}(c_1,c_2)=R^{(1)}_{F_1/F}\G_m\times R^{(1)}_{F_2/F}\G_m$, with $F_1,F_2/F$ quadratic extensions, where $c_i\in F^\times/N_{F_i/F}(F_1^\times)$;
    \item $T^{(1)}_{F_1^\#\oplus F_1^\#/F_1^\#}=\{(x,y)\in R_{F_1^\#/F}\G_m\times R_{F_1^\#/F}\G_m:xy=1\}$ with $F_1^\#/F$ a quadratic extension; and
    \item $T^{(1)}_{F_1/F_1^\#}(c)=R_{F_1^\#/F}R_{F_1/F_1^\#}^{(1)}\G_m$, with $F_1/F_1^\#/F$ a tower of quadratic extensions, where $c\in(F_1^\#)^\times/N_{F_1/F_1^\#}(F_1^\times)$.
\end{itemize}
Twisted Levi subgroups are obtained as centralizers of coroots into these tori. 

For the torus\[T^{(1)}_{F_1/F,F_2/F}(c_1,c_2)=R^{(1)}_{F_1/F}\G_m\times R^{(1)}_{F_2/F}\G_m\subset\SL_2(F)\times\SL_2(F)\subset\Sp_4(F),\] its subtorus $R^{(1)}_{F_1/F}\G_m\times 1$ (resp., $1\times R^{(1)}_{F_2/F}\G_m$) has centralizer $R^{(1)}_{F_1/F}\G_m\times\SL_2(F)$ (resp., $\SL_2(F)\times R^{(1)}_{F_2/F}\G_m$).

When $F_1=F_2$ the torus $T^{(1)}_{F_1/F,F_2/F}(c_1,c_2)$ also has the diagonal sub-torus $\Delta(R^{(1)}_{F_1/F}\G_m)$, which has centralizer $\mathrm U_{F_1/F}(c_1,c_2)$, the unitary group of the hermitian space $E\oplus E$ with hermitian form $h(w_1\oplus w_2,w_1'\oplus w_2')=\frac12(c_1w_1\overline w_1'+c_2w_2\overline w_2')$.  

The torus $T^{(1)}_{F_1^\#\oplus F_1^\#/F_1^\#}$ has the sub-torus $\{(x,y)\in F^\times\times F^\times:xy=1\}$, which has centralizer $\GL_2(F)\times\Sp_0(F)$.

The torus $T^{(1)}_{F_1/F_1^\#}$ has no nontrivial $F$-rational sub-tori.

Similarly, for $\GSp_4(F)$, the maximal tori which are anisotropic modulo center are thus of the following form:
\begin{itemize}
    \item $T_{F_1/F,F_2/F}(c_1,c_2)=\{(x,y)\in R_{F_1/F}\G_m\times R_{F_2/F}\G_m:\Nm_{F_1/F}x=\Nm_{F_2/F}y\}$ for quadratic field extensions $F_1,F_2/F$;
    \item $T_{F_1^\#\oplus F_1^\#/F_1^\#}=\{(x,y)\in R_{F_1^\#/F}\G_m\times R_{F_1^\#/F}\G_m:xy\in F^\times\}$ for a quadratic extension $F_1^\#/F$; and
    \item $T_{F_1/F_1^\#}(c):=\{x\in R_{F_1/F}\G_m:\Nm_{F_1/F_1^\#}x\in F^\times\}$, where $F_1/F_1^\#$ is a quadratic field extension.
\end{itemize}

For the torus $T_{F_1/F,F_2/F}\subset\{(x,y)\in\GL_2(F)\times\GL_2(F):\det(x)=\det(y)\}\subset\GSp_4(F)$, 
its subtorus $\{(x,y)\in R_{F_1/F}\G_m\times F^\times:\Nm_{F_1/F}x=y^2\}$ has centralizer 
\[\{(x,y)\in R_{F_1/F}\G_m\times\GL_2(F):\Nm_{F_1/F}x=\det(y)\}.\] The base change to $F_1$ gives the Levi subgroup $F_1^\times\times\GL_2(F_1)$. 

When $F_1=F_2$ it also has the diagonal sub-torus $\Delta(R_{F_1/F}\G_m)$, which has centralizer $\GU_{F_1/F}(2)$, whose base change to $F_1$ gives the Levi subgroup $F_1\times\GL_2(F_1)$.

Finally, the tori $T_{F_1^\#\oplus F_1^\#/F_1^\#}$ and $T_{F_1/F_1^\#}$ have no interesting sub-tori.
\end{numberedparagraph}

\begin{numberedparagraph}\textbf{Explicit type data for $\GSp_4(F)$ and $\Sp_4(F)$.}~

For $G=\Sp_4$ the type datum are:
\begin{enumerate}[label=($\text{pos-depth}_{{\arabic*}}$)]
    \item $\vec G=(T_{F_1/F_1^\#}^{(1)})$ for a tower of quadratic extensions $F_1/F_1^\#/F$. Here $G^0$ is abelian, so $\dim\rho^0=1$. The corresponding representation is nonsingular.
    \item $\vec G=(T_{F_1^\#\oplus F_1^\#/F_1^\#}^{(1)})$ for a quadratic extension $F_1^\#/F$. Here $G^0$ is abelian, so $\dim\rho^0=1$. The corresponding representation is nonsingular.
    \item $\vec G=(T^{(1)}_{F_1/F,F_2/F})$, with $F_1,F_2/F$ quadratic extensions. $G^0$ is abelian, so $\dim\rho^0=1$. The corresponding representation is nonsingular.
    \item $\vec G=(R_{F_1/F}^{(1)}\G_m\times\SL_2(F)\subset G)$ for a quadratic extension $F_1/F$. The positive-depth representation is nonsingular, since $R_{F_1/F}^{(1)}\G_m\times\SL_2(F)$ does not have any singular supercuspidal representations.
    \item $\vec G=(\mathrm U_{F_1/F}(c_1,c_2)\subset G)$ for a quadratic extension $F_1/F$.
    Here, the character $\phi_0$ is trivial, since $G^0$ does not have any interesting characters.

The unitary group $G^1=\mathrm U_{F_1/F}(c_1,c_2)$ is quasi-split if and only if the discriminant $-c_1c_2\in Nm_{F_1/F}(F_1^\times)$. Thus, $G^1$ has singular supercuspidals if and only if $G^0$ is quasi-split, which happens if $-c_1c_2\in Nm_{F_1/F}(F_1^\times)$.

    \item $\vec G=(T_{F_1/F,F_2/F}^{(1)}\subset R_{F_1/F}^{(1)}\G_m\times \SL_2(F)\subset G)$ for quadratic extensions $F_1,F_2/F$. Here, $G^0$ is abelian so $\dim\rho^0=1$. The corresponding representation is nonsingular.
    \item $\vec G=(T_{F_1/F,F_1/F}^{(1)}\subset \mathrm U_{F_1/F}(c_1,c_2)\subset G)$ for a quadratic extension $F_1/F$. Here, $G^0$ is abelian so $\dim\rho^0=1$. Moreover, $G^1$ has no interesting characters, so $\phi_1=1$. The corresponding representation is nonsingular.
    \item $\vec G=(T_{F_1^\#\oplus F_1^\#/F_1^\#}^{(1)}\subset\GL_2(F)\times\Sp_0(F)\subset G)$, for a quadratic extension $F_1^\#/F_1$. Here, $G^0$ is abelian so $\dim\rho^0=1$ and the representation is nonsingular.
\end{enumerate}

The possibilities for $G=\GSp_4$ are:
\begin{enumerate}[label=($\text{pos-depth}_{{\arabic*}}$)]
    \item $\vec G=(G^0=T_{F_1/F_1^\#}\subset G)$ for a tower of quadratic extensions $F_1/F_1^\#/F$. Since $G^0$ is abelian, $\dim\rho^0=1$. The corresponding representation is nonsingular.
    \item $\vec G=(G^0=T_{F_1^\#\oplus F_1^\#/F_1^\#}\subset G)$ for a quadratic extension $F_1^\#/F$. Since $G^0$ is abelian, $\dim\rho^0=1$. The corresponding representation is nonsingular.
    \item $\vec G=(G^0=T_{F_1/F,F_2/F}\subset G)$, with $F_1,F_2/F$ quadratic extensions. Since $G^0$ is abelian, $\dim\rho^0=1$. The corresponding representation is nonsingular.
    \item $\vec G=(G^0=\{(x,y)\in R_{F_1/F}\G_m\times\GL_2(F):\Nm_{F_1/F}x=\det(y)\}\subset G^1=G)$ for a quadratic extension $F_1/F$.  The positive-depth representation is nonsingular, since $R_{F_1/F}\G_m\times\GL_2(F)$ does not have any singular supercuspidal representations.
    \item $\vec G=(G^0=GU_{F_1/F}(c_1,c_2)\subset G^1=G)$ for a quadratic extension $F_1/F$. 
    The unitary group $G^0=\mathrm{GU}_{F_1/F}(c_1,c_2)$ is quasi-split if and only if the discriminant $-c_1c_2\in Nm_{F_1/F}(F_1^\times)$. Thus, $G^0$ has singular supercuspidals if and only if $G^0$ is quasi-split, which happens if $-c_1c_2\in Nm_{F_1/F}(F_1^\times)$.
    \item $\vec G= (G^0=T_{F_1/F,F_2/F}\subset G^1=\{(x,y)\in R_{F_1/F}\G_m\times\GL_2(F):\Nm_{E/F}x=\det(y)\}\subset G^2=G)$ for quadratic extensions $F_1,F_2/F$. The corresponding representation is nonsingular.
    \item $\vec G=(G^0=T_{F_1/F,F_1/F}\subset GU_{F_1/F}(2)\subset G^1=G)$ for a quadratic extension $F_1/F$. The corresponding representation is nonsingular.
    \item $\vec G=(T_{F_1^\#\oplus F_1^\#/F_1^\#}^{(1)}\subset\GL_2(F)\times\GSp_0(F)\subset G)$, for a quadratic extension $F_1^\#/F_1$. Here, $G^0$ is abelian so $\dim\rho^0=1$ and the representation is nonsingular.
\end{enumerate}
Note that the trivial representation of the compact unitary group $\SU(2)$ is a singular supercuspidal, which is only visible on the level of Vogan packets; it mixes with the Steinberg of $\SL_2(F)$.    
\end{numberedparagraph}

\subsection{Reducibility of induced representations}\label{section-matching-nonsc-GSp4}

\begin{prop}[{\cite[Prop~6.1]{sha91}}]

(a) Let $G=\GSp_4(F)$ for $F$ a non-archimedean field. Let $\alpha$ and $\beta$ be the short and long simple roots of $G$, respectively. Let $\mathbf{P}=\mathbf{M}\mathbf{N}$ be the maximal parabolic subgroup such that $\mathbf{M}$ is generated by $\alpha$ and $\mathbf{M}\cong \GL_2\times\GL_1$. Fix an irreducible unitary supercuspidal representation $\sigma=\sigma_1\otimes\chi$ of $M=\mathbf{M}(F)$, where $\sigma_1$ is a supercuspidal unitary representation of $\GL_2(F)$ with central character $\omega$ and $\chi$ is a unitary character of $F^*$. Then $I(\sigma)$ is always irreducible. The representation $I(\sigma_1\nu^s\otimes\chi)$ is reducible if and only if $\omega=1$ and $s=\pm \frac{1}{2}$, where $\nu$ denotes $\nu=|\det()|$ for $\GL_2(F)$. The representation $I(\sigma_1\nu^{1/2}\otimes\chi)$ has a unique generic special subrepresentation and a unique irreducible preunitary non-tempered non-generic quotient. For $0<s<1/2$, all the representations $I(\sigma_1\nu^s\otimes\chi)$ are in the complementary series and $s=1/2$ is their end point. 

(b) Let $G=\Sp_4(F)$, the representation $I(\sigma)$ is reducible if and only if $\sigma\cong\widetilde{\sigma}$ (thus $\omega^2=1$) and $\omega\neq 1$. Suppose $\omega=1$ so that $I(\sigma)$ is irreducible. Then $I(\sigma\nu^s)$ is reducible if and only if $s=\pm 1/2$. The representation $I(\sigma\nu^s)$ has a unique generic special subrepresentation and a unique irreducible preunitary non-tempered non-generic quotient. For $0<s<1/2$, all the representations $I(\sigma\nu^s)$ are in the complementary series and $s=1/2$ is their end point. 

(c) The Plancherel measure $\mu(s\widetilde{\alpha},\sigma)$ is given by the formula 
\begin{equation}
    \mu(s\widetilde{\alpha})=\begin{cases}\gamma(G/P)^2q^{n(\sigma_1)}\frac{(1-\omega(\varpi)q^{-2s})(1-\omega(\varpi)^{-1}q^{2s})}{(1-\omega(\varpi)q^{-1-2s})(1-\omega(\varpi)^{-1}q^{-1+2s})}&\text{if $\omega$ is unramified}\\ \gamma(G/P)^2q^{n(\sigma_1)+n(\omega)}&\text{otherwise}\end{cases}
\end{equation}
Here $n(\sigma_1)$ and $n(\omega)$ are the conductors of $\sigma_1$ and $\omega$, respectively. 
\end{prop}

For a character $\chi$ of $F^\times$, let $e(\chi):=\log_q|\chi(\varpi)|$ be the unique real number such that $\chi=\nu^{e(\chi)}\chi_0$ where $\chi_0$ is a unitary character.

\begin{lemma}[{\cite[Lem~3.2]{sally-tadic}}]\label{reducibility-lemma}
Let $\chi_1$, $\chi_2$, and $\theta$ be characters of $F^\times$. Then $\chi_1\times\chi_2\rtimes\theta$ is reducible if and only if $\chi_1=\nu^{\pm1}$, $\chi_2=\nu^{\pm1}$, or $\chi_1=\nu^{\pm1}\chi_2^{\pm1}$.
\end{lemma}

We thus have the following theorem:

\begin{thm}\label{gsp4-induced}
A representation of $\GSp_4(F)$ parabolically induced from a Levi $L\subset G$ is not irreducible exactly in the following cases:
\begin{enumerate}
    \item \label{gsp4-110-levi}When $L=T$, the representation $\chi_1\times\chi_2\rtimes\theta$ is reducible when either:
    \begin{enumerate}
        \item if $\chi_1\times\chi_2\rtimes\theta$ is regular, i.e., $\chi_1\ne1,\chi_2\ne1,\chi_1\ne\chi_2^{\pm1}$:
        \begin{enumerate}
            \item\label{gsp4-1ai} $\chi_1=\nu\chi_2$ where $\chi_2^2\ne\nu^{-2},\nu^{-1},1$ and $\chi_2\ne \nu^{-2},\nu$. Then $\nu\chi_2\times\chi_2\rtimes\theta$ has length $2$ and in the Grothendieck ring
            \[
            \nu\chi_2\times\chi_2\rtimes\theta=\nu^{1/2}\chi_21_{\GL_2}\rtimes\theta+\nu^{1/2}\chi_2\St_{\GL_2}\rtimes\theta.
            \]
            The Langlands classification is
            \begin{align*}
                \nu^{1/2}\chi_2\St_{\GL_2}\rtimes\theta&=\begin{cases}J(\nu^{1/2}\chi_2\St_{\GL_2};\theta)&e(\chi_2)>-\frac12\\
                J(\nu^{1/2}\chi_2\St_{\GL_2}\rtimes\theta)&e(\chi_2)=-\frac12\\
                J(\nu^{-1/2}\chi_2^{-1}\St_{\GL_2};\nu\chi_2^2\theta)&e(\chi_2)<-\frac12
                \end{cases}\\
                \nu^{1/2}\chi_21_{\GL_2}\rtimes\theta&=\begin{cases}J(\nu\chi_2,\chi_2;\theta)&e(\chi_2)>0\\
                J(\nu\chi_2,\chi_2\rtimes\theta)&e(\chi_2)=0\\
                J(\nu\chi_2,\chi_2^{-1};\chi_2\theta)&0>e(\chi_2)\ge-\frac12\\
                J(\chi_2^{-1},\nu\chi_2;\nu\chi_2\theta)&-\frac12>e(\chi_2)>-1\\
                J(\chi_2^{-1};\nu^{-1}\chi_2^{-1}\rtimes\nu\chi_2^2\theta)&e(\chi_2)=-1\\
                J(\chi_2^{-1},\nu^{-1}\chi_2^{-1};\nu\chi_2^2\theta)&e(\chi_2)<-1
                \end{cases}
                \end{align*}
            \item\label{gsp4-1aii} $\chi_2=\nu$ and $\chi_1\ne1,\nu^{\pm1},\nu^{\pm2}$. Then $\chi_1\times\nu\rtimes\theta$ has length $2$ and in the Grothendieck ring
            \[
            \chi_1\times\nu\rtimes\theta=\chi_1\rtimes\nu^{1/2}\theta\St_{\GSp_2}+\chi_1\rtimes\nu^{1/2}\theta1_{\GSp_2}.
            \]
            Then,
            \begin{align*} \chi_1\rtimes\nu^{1/2}\theta\St_{\GSp_2}&=\begin{cases}
            J(\chi_1;\nu^{1/2}\theta\St_{\GSp_2})&e(\chi_1)>0\\
            J(\chi_1\rtimes\nu^{1/2}\theta\St_{\GSp_2})&e(\chi_1)=0\\
            J(\chi_1^{-1};\nu^{1/2}\chi_1\theta\St_{\GSp_2})&e(\chi_1)<0
                \end{cases}\\
            \chi_1\rtimes\nu^{1/2}\theta1_{\GSp_2}&=\begin{cases}
            J(\chi_1,\nu;\theta)&e(\chi_1)>0\\
            J(\nu;\chi_1\rtimes\theta)&e(\chi_1)=0\\
            J(\chi_1^{-1},\nu;\chi_1\theta)&e(\chi_1)<0
            \end{cases}
            \end{align*}
            \item\label{gsp4-1aiii} $\chi_1=\nu^2$ and $\chi_2=\nu$. Then $\nu^2\times\nu\rtimes\theta$ has length $4$, consisting of:
            \[
            \nu^{3/2}\theta\St_{\GSp_4},\nu^{3/2}\theta1_{\GSp_4},J(\nu^2;\nu^{1/2}\theta\St_{\GSp_2}), J(\nu^{3/2}\St_{\GL_2};\theta)
            \]
            \item\label{gsp4-1aiv}$\chi_1=\nu\chi_2$ and $\chi_2$ of order $2$. Then $\nu\chi_2\times\chi_2\rtimes\theta$ has length $4$, with a unique essentially square-integrable subquotient denoted by $\delta([\chi_2,\nu\chi_2],\theta)$, as well as
            \[
           J(\nu^{1/2}\chi_2\St_{\GL_2};\theta),J(\nu^{1/2}\chi_2\St_{\GL_2};\chi_2\theta),J(\nu\chi_2;\chi_2\rtimes\theta).
            \]
        \end{enumerate}
        \item if $\chi_1\times\chi_2\rtimes\theta$ is not regular:
        \begin{enumerate}
            \item\label{gsp4-1bi} $\chi_1=\nu,\chi_2=1$ then $\nu\times1\rtimes\theta$ has length $4$ consisting of essentially tempered representations $\tau(S,\theta)$ and $\tau(T,\theta)$ such that $1\rtimes\nu^{1/2}\theta\St=\tau(S,\theta)+\tau(T,\theta)$, as well as
            \(
            J(\nu;1_{F^\times}\rtimes\theta)\) and \(J(\nu^{1/2}\St_{\GL_2};\theta)
            \), where $1\rtimes \theta1_{\GSp_2}=J(\nu;1\rtimes\theta)+J(\nu^{1/2}\St;\theta)$
            \item\label{gsp4-1bii} $\chi_1=\chi_2=\nu$ then $\nu\times\nu\rtimes\theta$ has length $2$ consisting of
            \begin{align*}\nu\rtimes\nu^{1/2}\theta1_{\GSp_2}&=J(\nu;\nu^{1/2}\theta\St_{\GSp_2})\\
            \nu\rtimes\nu^{1/2}\theta\St_{\GSp_2}&=J(\nu,\nu;\theta).
            \end{align*}
            \item\label{gsp4-1biii} $\chi_1=\nu\chi_2$ and $\chi_1^2=\nu$, then $\nu\chi_2\times\chi_2\rtimes\theta$ has length $2$ consisting of 
            \[\nu^{1/2}\chi_21_{\GL_2}\rtimes\theta,\nu^{1/2}\chi_2\St_{\GL_2}\rtimes\theta.
            \]
            Here, $\nu^{1/2}\chi_2\St_{\GL_2}\rtimes\theta$ is tempered and $\nu^{1/2}\chi_21_{\GL_2}\rtimes\theta=J(\nu\chi_2,\nu\chi_2;\chi_2\theta)$.
        \end{enumerate}
    \end{enumerate}
    \item\label{gsp4-20-levi} When $L=\GL_2\times\GSp_0$, the representation $\nu^\beta\rho\rtimes\chi$, where $\beta\in\R$, $\rho$ is a unitary supercuspidal of $\GL_2$, and $\chi\colon F^\times\to\C^\times$ is reducible if and only if $\beta=\pm1/2$ and $\rho=\rho^\vee$ and $\omega_\rho=1$.
    
    Moreover, $\nu^{1/2}\rho\rtimes\chi$ has a unique generic special sub-representation and a unique irreducible preunitary nontempered non-generic quotient.
    
    \item\label{gsp4-12-levi} When $L=\GL_1\times\GSp_2$, the representation $\chi\rtimes\rho$, where $\chi\colon F^\times\to\C^\times$ and $\rho$ is a supercuspidal representation of $\GSp_2$, is reducible in the following cases:
    \begin{enumerate}
        \item\label{gsp4-3a} $\chi=1_{F^\times}$, in which case $1_{F^\times}\rtimes\rho$ splits into a sum of two tempered irreducible sub-representations which are not equivalent.
        \item\label{gsp4-3b} $\chi=\nu^{\pm1}\xi_o$ where $\xi_o\colon F^\times\to\C^\times$ is a character of order two such that $\xi_o\rho\cong\rho$. Then $\nu\xi_o\rtimes\rho$ has a unique irreducible sub-representation which is square-integrable.
    \end{enumerate}
\end{enumerate}
\end{thm}
\begin{proof}
Case~\eqref{gsp4-110-levi} is from \cite[\S3]{sally-tadic} and Cases~\eqref{gsp4-20-levi}~and~\eqref{gsp4-12-levi} are from \cite[\S4]{sally-tadic}.

More precisely, Case~\ref{gsp4-1ai} is \cite[Lemma~3.3]{sally-tadic}, Case~\ref{gsp4-1aii} is \cite[Lemma~3.4]{sally-tadic}, Case~\ref{gsp4-1aiii} is \cite[Lemma~3.5]{sally-tadic}, Case~\ref{gsp4-1aiv} is \cite[Lemma~3.6]{sally-tadic}, Case~\ref{gsp4-1bi} is \cite[Lemma~3.8]{sally-tadic}, Case~\ref{gsp4-1bii} is \cite[Lemma~3.9]{sally-tadic}, Case~\ref{gsp4-1biii} is \cite[Lemma~3.7]{sally-tadic}.
\end{proof}

Let $\xi$ have order $2$ and write $\xi\rtimes1=T_\xi^1+T_\xi^2$ as a sum of irreducible representations of $\Sp_2$. Moreover, for any supercuspidal representation $\sigma$ of $\SL_2(F)$, let
\[
F_\sigma^\times:=\{a\in F^\times:\sigma^{\diag(a,1)}\cong\sigma\},
\]
which is really a subgroup of the finite group $F^\times/(F^\times)^2$.

The analogue of Theorem~\ref{gsp4-induced} for $\Sp_4$ is:
{
\begin{thm}\label{sp4-induced} 
\begin{enumerate}
    \item\label{sp4-110-levi} When $L=T$, the representation $\chi_1\times\chi_2\rtimes1$ is reducible when:
    \begin{enumerate}
        \item\label{sp4-1a} The representations coming from irreducibles of $\GSp_4$, i.e., $\chi_1\ne\nu^{\pm1}$, $\chi_2\ne\nu^{\pm1}$, and $\chi_1\ne\nu^{\pm1}\chi_2$. Then $\chi_1\times\chi_2\rtimes1$ is reducible exactly when $\chi_1$ or $\chi_2$ has order $2$. We may suppose without loss that $\chi_2$ has order $2$.

        \begin{enumerate}
        \item\label{sp4-1ai} If $\chi_1=\chi_2$ or $\chi_1$ is not of order $2$ then $\chi_1\rtimes T_{\chi_2}^1$ and $\chi_1\rtimes T_{\chi_2}^2$ are irreducible
        \item\label{sp4-1aii} If $\chi_1=\chi_2$ then both $\chi_1\rtimes T_{\chi_1}^1$ and $\chi_1\rtimes T_{\chi_1}^2$ have length two.
        \end{enumerate}
        
        \item\label{sp4-1b} if $\chi_1\times\chi_2\rtimes1$ is regular, i.e., $\chi_1\ne1,\chi_2\ne1,\chi_1\ne\chi_2^{\pm1}$:
        \begin{enumerate}
            \item\label{sp4-1bi} $\chi_1=\nu\chi_2$ where $\chi_2^2\ne\nu^{-2},\nu^{-1},1$ and $\chi_2\ne \nu^{-2},\nu$. Then
            \[
            \nu\chi_2\times\chi_2\rtimes1=\nu^{1/2}\chi_21_{\GL_2}\rtimes1+\nu^{1/2}\chi_2\St_{\GL_2}\rtimes1            \]
            has length two.
            \item\label{sp4-1bii} $\chi_2=\nu$ and $\chi_1\ne1,\nu^{\pm1},\nu^{\pm2}$. Then
            \[
            \chi_1\times\nu\rtimes1=\chi_1\rtimes\nu^{1/2}\St_{\Sp_2}+\chi_1\rtimes\nu^{1/2}1_{\Sp_2}
            \]
            has length two.
            \item\label{sp4-1biii} $\chi_1=\nu^2$ and $\chi_2=\nu$. Then $\nu^2\times\nu\rtimes1$ has length $4$, consisting of:
            \[
            \nu^{3/2}\St_{\Sp_4},\nu^{3/2}1_{\Sp_4},J(\nu^2;\nu^{1/2}\St_{\Sp_2}), J(\nu^{3/2}\St_{\GL_2};1)
            \]
            \item\label{sp4-1biv} $\chi_1=\nu\chi_2$ and $\chi_2$ of order $2$. Then \[\nu\chi_2\times\chi_2\rtimes1=\nu^{1/2}\chi_21_{\GL_2}\rtimes1+\nu^{1/2}\chi_2\St_{\GL_2}\rtimes1\]where $\nu^{1/2}\chi_21_{\GL_2}\rtimes1$ and $\nu^{1/2}\chi_2\St_{\GL_2}\rtimes1$ each have length three.
        \end{enumerate}
        Otherwise, there are no extra reducibilities.
        \item if $\chi_1\times\chi_2\rtimes1$ is not regular:
        \begin{enumerate}
            \item\label{sp4-1ci} $\chi_1=\nu,\chi_2=1$ then $\nu\times1\rtimes1$ has length $4$ consisting of essentially tempered representations $\tau$ and $\tau'$, as well as
            \[
            J(\nu;1_{F^\times}\rtimes1_{\Sp_2}),J(\nu^{1/2}\St_{\GL_2};1)
            \]
            \item\label{sp4-1cii} $\chi_1=\chi_2=\nu$ then
            \[\nu\times\nu\rtimes1=\nu\rtimes\nu^{1/2}1_{\Sp_2}+\nu\rtimes\nu^{1/2}\St_{\Sp_2}.
            \]
            in the Grothendieck ring, where both $\nu\rtimes\nu^{1/2}1_{\Sp_2}$ and $\nu\rtimes\nu^{1/2}\St_{\Sp_2}$ are irreducible.
            \item\label{sp4-1ciii} $\chi_1=\nu\chi_2$ and $\chi_1^2=\nu$, then $\nu\chi_2\times\chi_2\rtimes1$ has length $2$ consisting of 
            \[\nu^{1/2}\chi_21_{\GL_2}\rtimes1,\nu^{1/2}\chi_2\St_{\GL_2}\rtimes1.
            \]
            Otherwise, there are no extra reducibilities.
        \end{enumerate}
    \end{enumerate}
    \item\label{sp4-2} When $L=\GL_2\times\Sp_0=\GL_2$, the representation $\nu^\beta\rho\rtimes1$, where $\beta\in\R$ and $\rho$ is a unitary supercuspidal of $\GL_2$ is reducible if and only if $\rho$ is self-dual and:
    \begin{enumerate}
        \item $\beta=\pm1/2$ and $\omega_\rho=1_{F^\times}$; or
        \item $\beta=0$ and $\omega_\rho\ne1_{F^\times}$,.
    \end{enumerate}
    \item\label{sp4-3} When $L=\GL_1\times\Sp_2$, the representation $\nu^\beta\chi\rtimes\rho$, where $\chi$ is a unitary character and $\beta\in\R$ and $\rho$ is a supercuspidal representation of $\Sp_2$, is reducible in the following cases:
    \begin{enumerate}
        \item\label{sp4-3a} $\chi=1_{F^\times}$ and $\beta=0$,
        \item\label{sp4-3b} $\chi$ has order two and nontrivial on $F_\sigma^\times$ and $\beta=0$.
        \item\label{sp4-3c} $\chi$ has order two and trivial on $F_\sigma^\times$ and $\beta=\pm1$.
    \end{enumerate}
\end{enumerate}
\end{thm}
\begin{proof}
See \cite[Section~5]{sally-tadic}.
\end{proof}}

\section{The Galois side}\label{sec:Galois-side}

We are concerned with $L$-parameters of $G=\Sp_4,\GSp_4$, i.e, homomorphisms $\varphi\colon W_F\times\SL_2(\C)\to G^\vee$ such that $\varphi(w)$ is semisimple for any $w\in W_F$, and the restriction $\varphi|_{\SL_2(\C)}$ is a morphism of complex algebraic groups.

\begin{lemma}\label{no-levi-lemma}
If $\cG_\varphi^\circ$ is abelian, then members of the $L$-packet for $\varphi$ are representations with support $\Cent_{G^\vee}(\Cent_{\cG_\varphi}^\circ)^\vee$.
\end{lemma}
\begin{proof}

Let $\rho\in\Irr(S_\varphi)$. Since $\cG_\varphi^\circ$ is abelian, the cuspidal support $\cL^\varphi$ of $(u_\varphi,\rho)$, which is a quasi-Levi of $\cG_\varphi$ in the sense of \cite[pg~5]{AMS18}, must be $\Cent_{\cG_\varphi}(\cG_\varphi^\circ)$. Thus the cuspidal support of $(\varphi,\rho)$ must be $\Cent_{G^\vee}(\Cent_{\cL^\varphi}^\circ)=\Cent_{G^\vee}(\cG_\varphi^\circ)$. By Property~\ref{property:AMS-conjecture-7.8} the members of the $L$-packet of $\varphi$ has support $\Cent_{G^\vee}(\Cent_{\cG_\varphi}^\circ)^\vee$.
\end{proof}

Let $G=\Sp_4(F)$ and $\varphi\colon W_F\times\SL_2\to G^\vee=\SO_5(\C)$ be an $L$-parameter. Consider $\varphi|_{W_F}$ as a $5$-dimensional representation of $W_F$ with an invariant symmetric inner product.

We use the following notation from \S\ref{section-properties}:
\[
    \cG_\varphi=\Cent_{\SO_5(\C)}(\varphi(W_F))\hspace{0.2cm}\text{ and }\hspace{0.2cm}
    S_\varphi=\pi_0(\Cent_{\SO_5(\C)}(\varphi(W_F'))).
\]
The cuspidal support map $\Sc\colon \Phi_e(G)\to\bigsqcup_{L\in\cL(G)}\Phi_{e,\cusp}(L)/W_G(L)$ is defined via the Springer correspondence\footnote{There exist in literature different ways to normalize the Springer correspondences, see for example \cite{Collingwood-McGovern}; for constructing LLC, the normalization used sends the regular nilpotent orbit to the sign representation of $W$.} for $\cG_\varphi$, so we conduct case-work on the shape of the $L$-parameter $\varphi$.

There are the following cases, depending on how the $W_F$-representation $U$ decomposes (parameterized by partitions of $5$). 
\begin{enumerate}
    \item
    $U$ it is irreducible, so $\cG_\varphi=1$ and $S_\varphi=1$. This is a supercuspidal singleton packet.
    \item
    $U=V\oplus\chi$ where $\dim V=4$ with a symmetric form $V\otimes V\to\C$ and $\chi^2=1$. Here $\cG_\varphi=\mu_2$ and $S_\varphi=\mu_2$. Here, $\cL_\varphi=\mu_2$ so $\Cent_{G^\vee}(\Cent_{\cL_\varphi}^\circ)=G^\vee$. Thus this is a purely supercuspidal packet of size $2$.
    \item
    $U=V_1\oplus V_2$ where $\dim V_1=3$ and $\dim V_2=2$, both self-dual with invariant symmetric forms. Here $\cG_\varphi=\mu_2$ and $S_\varphi=\mu_2$. Again, this is a purely supercuspidal packet of size $2$.
    \item
    $U=V\oplus \chi_1\oplus\chi_2$ where $\dim V=3$ and $V$ is self-dual with an invariant symmetric form. Either:
    \begin{enumerate}
        \item\label{311-case-a} $\chi_1=\chi_2$ so $\chi_1^2=\chi_2^2=1$ since $\chi_1\oplus\chi_2$ must be self-dual. Now $\cG_\varphi=\tS(\mu_2\times \tO_2(\C))\cong\tO_2(\C)$, since an automorphism of $U$ must act by scalars on $V$ and by an orthogonal transformation on $\chi_1\oplus\chi_2$. Since $\cG_\varphi$ has no unipotents, $\varphi|_{\SL_2(\C)}$ is trivial and $S_\varphi=\mu_2$.
        
        Here $\cL_\varphi=1\times\SO_2(\C)$ so the cuspidal support is $\Cent_{G^\vee}(\Cent^\circ_{\cL_\varphi})=\Cent_{G^\vee}(1\times\SO_2(\C))=\GL_1(\C)\times\SO_3(\C)$. Since supercuspidal $L$-parameters of $\SO_3(\C)\cong\PGL_2(\C)$ have trivial unipotent, by Property~\ref{infinitesimal-prop} (and the observation that $\varphi|_{\SL}$), we have \[\varphi=\lambda_\varphi=\iota_{\GL_1\times\SO_3}\circ \lambda_{\varphi_v}=\iota_{\GL_1\times\SO_3}\circ\varphi_v.\]
        Thus the packet consists of sub-quotients of the parabolic induction $\widehat\chi_1\rtimes\pi_V$ where $\pi_V$ is the representation of $\Sp_2(F)$ corresponding to $V$ under the LLC for $\Sp_2(F)\cong\SL_2(F)$ (this is well-defined, since $V$ corresponds to a singleton packet). 
        \item $\chi_1\ne\chi_2$ and $\chi_1^2=\chi_2^2=1$ then $\cG_\varphi=\mu_2^2$, so $\varphi|_{\SL_2(\C)}$ is trivial and $S_\varphi=\mu_2^2$. By Lemma~\ref{no-levi-lemma}, this is a purely supercuspidal packet of size $4$.
        \item $\chi_1\ne\chi_2$ and $\chi_1=\chi_2^{-1}$ then $\chi_1\oplus\chi_2$ carries the symmetric form $\langle(a_1,b_1),(a_2,b_2)\rangle:=a_1b_2+a_2b_1$ so $\cG_\varphi=\C^\times$ and $\varphi|_{\SL_2(\C)}=1$ and $S_\varphi=1$. Again by Lemma~\ref{no-levi-lemma} the support of the unique member of the $L$-packet is $\Cent_{\SO_5}(\cG_\varphi^\circ)^\vee=F^\times\times\Sp_2(F)$. By the same argument as in case~\ref{311-case-a}, the member of the $L$-packet is $\chi_1\rtimes\pi_V$.
        
    \end{enumerate}
    \item
    $U=V_1\oplus V_2\oplus\chi$ where $\dim V_1=\dim V_2=2$, and $\chi^2=1$. Either:
    \begin{enumerate}
        \item $V_1\cong V_2$ and $V_1$ has an invariant symmetric form so $\cG_\varphi=\C^\times$ and $S_\varphi=1$. By Lemma~\ref{no-levi-lemma}, this is a purely supercuspidal singleton packet.
        \item $V_1\cong V_2$ and $V_1$ has an invariant symplectic form $\omega$ then $V_1\oplus V_1$ carries the symmetric form $\langle v_1\oplus v_2,w_1\oplus w_2\rangle:=\omega(v_1,w_2)-\omega(v_2,w_1)$. Then $\chi=1$ and $\cG_\varphi=\Sp_2(\C)$. The Springer correspondence for $\Sp_2\cong\SL_2$ is shown on Table~\ref{table-sl2-springer}. Thus the Levi subgroup $\cL_\varphi\subset\cG_\varphi$ is either $T$ or $\Sp_2(\C)$ and $\Cent_{G^\vee}(\Cent_{\cL_\varphi}^\circ)$ is either $\GL_2(\C)\times\SO_1(\C)$ or $G^\vee$, correspondingly. Thus:
        {\begin{center}
\begin{table}[ht]        \begin{tabular}{ |c|c| } 
 \hline
 Unipotent pairs& Representations of $W=\mu_2$\\ \hline
 $([1^2],1)$&$1$\\
  $([2],1)$&$\sgn$\\
 $([2],-1)$&cusp\\
 \hline
\end{tabular}\label{table-sl2-springer}\caption{The Springer correspondence for $\SL_2$ \cite[\S10.3]{Lusztig-IC}}\end{table}\end{center}}

\begin{itemize}
    \item When $\varphi|_{\SL_2}=1$ then $S_\varphi=1$ so the $L$-packet is $\{\pi_V\rtimes1\}$. Here, since $V$ is an $L$-parameter into $\SL_2$, we have $\omega_\rho=1$, so by Theorem~\ref{sp4-induced} the representation $\pi_V\rtimes 1$ is irreducible.
    \item When $\varphi|_{\SL_2}$ is nontrivial, then $S_\varphi=\mu_2$ so the $L$-packet has size $2$. This packet is determined in Section~\ref{section-mixed-packets-GSp4}.
    
    Concretely, the second $L$-parameter can be considered the $W_F\times\SL_2(\C)$-representation $U=M_2(\C)\oplus\C$ where $W_F$ acts on $M_2(\C)$ by left multiplication via the representation $V_1$, and $\SL_2(\C)$ acts on $M_2(\C)$ by right multiplication.
\end{itemize}

\label{sp4-case5b}
        \item $V_1\ncong V_2$ and both have an invariant symmetric form, then $\chi\cong\det(V_1)\otimes\det(V_2)$. Here $\cG_\varphi=\mu_2^2$ and $S_\varphi=\mu_2^2$. By Lemma~\ref{no-levi-lemma} this is a purely supercuspidal packet of size four.
        \item $V_1\ncong V_2$ and $V_1\cong V_2^\vee$ then $\cG_\varphi=\C^\times$ and $S_\varphi=1$. By Lemma~\ref{no-levi-lemma} the member of the singleton $L$-packet is $\pi_V\rtimes1$, supported in $\GL_2(F)\times\Sp_0(F)$. The representation $\pi_V\rtimes 1$ is irreducible by Theorem~\ref{sp4-induced}\eqref{sp4-2}, since $\pi_V$ is not self-dual.
    \end{enumerate}
    \item
    $U=V\oplus\chi_1\oplus\chi_2\oplus\chi_3$ where $\dim V=2$ with $V$ self-dual with an invariant symmetric form $V\otimes V\to\C$. Either:
    \begin{enumerate}
        \item $\chi_1=\chi_2=\chi_3$ with $\chi_1^2=1$ then $\cG_\varphi= \SO_3(\C)\times\mu_2$, and $\chi_1=\det(V)$. The Springer correspondence for $\SO_3(\C)\cong\PGL_2(\C)$ is given in Table~\ref{Springer-SO3}, where all local systems are supported in the torus.{\begin{center}\begin{table}\label{Springer-SO3}\begin{tabular}{ |c|c| }
 \hline
 Unipotent pairs& Representations of $W=\mu_2$\\ \hline
 $([1^2],1)$&$1$\\
 $([2],1)$&$\sgn$\\
 \hline\end{tabular}\caption{Springer Correspondence for $\SO_3(\C)$}\end{table}\end{center}}Thus $\cL_\varphi=\mu_2\times\C^\times\subset\cG_\varphi$. Now $\Cent_{G^\vee}(\Cent_{\cL_\varphi}^\circ)=\C^\times\times\SO_3(\C)$ and the members of the $L$-packet are supported in $\GL_1(F)\times\Sp_2(F)$. Explicitly, the restriction $\varphi|_{\SL_2(\C)}$ is either:
        \begin{enumerate}
            \item trivial, so $S_\varphi=\mu_2$. The $W_F$-representation $V\oplus\chi_1$ can be viewed as an $L$-parameter $W_F\to\SO_3(\C)$, which then corresponds to representations $\pi_1,\pi_2$ of $\Sp_2(F)$ under LLC for $\Sp_2(F)$ (the packet has size $2$). The $L$-packet is $\{\widehat\chi_1\rtimes\pi_1,\widehat\chi_1\rtimes\pi_2\}$, which are irreducible by Theorem~\ref{sp4-induced}\eqref{sp4-3}.
            \item\label{sp4-galois-6aii} nontrivial. Then $S_\varphi=\mu_2$, and by Property~\ref{infinitesimal-prop} the $L$-packet is $\{\nu\widehat\chi_1\rtimes\pi_1,\nu\widehat\chi_1\rtimes\pi_2\}$, which are irreducible by Theorem~\ref{sp4-induced}\eqref{sp4-3}.
        \end{enumerate}

        \item $\chi_1=\chi_2\ne\chi_3$ then $\chi_1^2=\chi_3^2=1$ and $\chi_3=\det(V)$ and $\cG_\varphi=\mu_2\times\tS(\tO_2(\C)\times\mu_2)$ with $S_\varphi=\mu_2\times\mu_2$. By Lemma~\ref{no-levi-lemma} the members of the size four $L$-packet are supported in $\GL_1(F)\times\Sp_2(F)$. By the LLC for $\Sp_2(F)$ the $W_F$-representation $V\oplus\chi_3$ viewed as an $L$-parameter $W_F\to\SO_3(\C)$ gives an $L$-packet $\{\pi_1,\pi_2\}$. Now, each of the representations $\chi_1\rtimes\pi_1$ and $\chi_1\rtimes\pi_2$ have length two by Theorem~\ref{sp4-induced}\eqref{sp4-3}, so they decompose into, say $\tau_{11}+\tau_{12}$ and $\tau_{21}+\tau_{22}$, respectively. Then the $L$-packet for $\varphi$ is $\{\tau_{11},\tau_{12},\tau_{21},\tau_{22}\}$.
        \item $\chi_1\ne\chi_2\ne\chi_3$ and $\chi_1^2=\chi_2^2=\chi_3^2=1$ then $\cG_\varphi=\mu_2\times\tS(\mu_2\times\mu_2\times\mu_2)$ and $S_\varphi\cong\mu_2^3$. This is a purely supercuspidal packet by Lemma~\ref{no-levi-lemma}.
        \item $\chi_1\ne\chi_2\ne\chi_3$ and $\chi_1^2=1$ and $\chi_2=\chi_3^{-1}$ but $\chi_2^2\ne1$. Here $\cG_\varphi=\mu_2\times\C^\times$ and $S_\varphi\cong \mu_2$. The members of the $L$-packet are supported in $\GL_1(F)\times\Sp_2(F)$. Letting $\{\pi_1,\pi_2\}$ be the $L$-packet under the LLC for $\Sp_2(F)$ corresponding to the $W_F$-representation $V\oplus\chi_1$ viewed as a $L$-parameter $W_F\to\SO_3(\C)$, the $L$-packet for $\varphi$ is $\{\chi_2\rtimes\pi_1,\chi_2\rtimes\pi_2\}$, which is irreducible by Theorem~\ref{sp4-induced}\eqref{sp4-3}.
    \end{enumerate}
    \item\label{sp4-galois-case7}
    $U=1\oplus \chi_1\oplus\chi_1^{-1}\oplus\chi_2\oplus\chi_2^{-1}$.
    \begin{enumerate}
        \item $\chi_1=\chi_2=1$ then $\cG_\varphi=\SO_5(\C)$. The Springer correspondence of $\cG_\varphi$ is \cite[\S10.6]{Lusztig-IC}:
        {\begin{center}
\begin{tabular}{ |c|c| } 
 \hline
 Unipotent pairs& Representations of $W=\mu_2^2\rtimes S_2$\\ \hline
 $([5],1)$&$(\emptyset,[1^2])$\\
 $([3,1^2],1)$&$([1],[1])$\\
 $([3,1^2],-1)$&$(\emptyset,[2])$\\
 $([2^2,1],1)$&$([1^2],\emptyset)$\\
 $([1^5],1)$&$([2],\emptyset)$\\
 \hline
\end{tabular}
\end{center}}
where we identify representations of the semidirect product $(\Z/2)^2\rtimes S_2$ via Lemma~\ref{mackey-induced} (see also, \cite[Theorem~10.1.2]{Collingwood-McGovern}). All of the representations are principal series.

 By Property~\ref{infinitesimal-prop}, the cuspidal support of the $L$-parameter is
\[
   \varphi_v(w)=\lambda_{\varphi_v}(w)=\lambda_{\varphi}(w)=\varphi(\diag(\|w\|^{1/2},\|w\|^{-1/2})).\]
By Remark~\ref{nilpotent-sp4} all nilpotent orbits in $G^\vee$ are induced from some regular nilpotent orbit in a Levi subgroup $L^\vee\subset G^\vee$. Thus $\varphi(\|w\|^{1/2},\|w\|^{-1/2})$ is dual to the modulus character $\delta_{B_L\backslash L}$. Thus, by Remark~\ref{gsp4-levi-llc-remark}, we have $\varphi_v=\widehat\chi_1^{-1}\delta_{B_L\backslash L}$. Thus the $L$-packet contains an irreducible subquotient of $\ii_P^G(\St_L)$.

\begin{enumerate}
    \item If $\varphi|_{\SL_2}$ is $[4]$, then the $L$-packet member is a subquotient of $\ii^G_B(\delta_{B\backslash G})$, which is square-integrable modulo center, by Property~\ref{property:L-packets}. Thus the $L$-packet is $\{\St_{\GSp_4}\}$.
    \item If $\varphi|_{\SL_2}$ is $[2^2]$ then $S_\varphi=\mu_2$, then the $L$-packet members are irreducible constituents of $1\rtimes\St_{\GL_2}$. This is case~\ref{sp4-1ci} and the $L$-packet is $\{\tau(S,\nu^{-1/2}\widehat\chi_1^{-1}),\tau(T,\nu^{-1/2}\widehat\chi_1^{-1})\}$.
    \item If $\varphi|_{\SL_2}$ is $[2,1^2]$. The $L$-packet members is $\St_{\GL_2}\rtimes1$, which is case~\ref{sp4-1ciii}.
    \item If $\varphi|_{\SL_2}$ is trivial, then the $L$-packet is $\{1\times1\rtimes1\}$, where $1\times1\rtimes1$ is irreducible by Theorem~\ref{sp4-induced}\eqref{sp4-1a}.
\end{enumerate}

\item $\chi_1=\chi_2\ne1$ then $\chi_1$ has order $2$ and $\cG_\varphi=\tS(\tO_4(\C)\times\mu_2)\cong \tO_4(\C)$. 

The Springer correspondence for $\tO_4$ is (see \cite[\S10.1, p. 166]{collingwood}):
{\begin{center}\begin{tabular}{ |c|c| } 
 \hline
 Unipotent pairs& Representations of $W=\mu_2^2\rtimes\mu_2$\\ \hline
 $(00,1)$&$(1\otimes1,1)=1_W$\\
 $(00,-1)$&$(1\otimes1,\sgn)$\\
 $(0e,1)=(e0,1)$&$(1\otimes\sgn,1)$\\
 $(ee,(1,1))$&$(\sgn\otimes\sgn,1)=\sgn_W$\\
 $(ee,(1,-1))$&$(\sgn\otimes\sgn,\sgn)$\\
 $(ee,(-1,1))$&cusp\\
 $(ee,(-1,-1))$&cusp\\
 \hline
\end{tabular}\end{center}}
Here on the right $0$ and $e$ denote the unipotent classes of $\SL_2$, which induce unipotent classes on $\SO_4=(\SL_2\times\SL_2)/\mu_2$, and on the left are representations of the Weyl group $W=\mu_2^2\rtimes\mu_2$ parameterized via Lemma~\ref{mackey-induced}. 

Thus $\cL_\varphi\subset\cG_\varphi^\circ=\SO_4(\C)$ is either the maximal torus or $\SO_4(\C)$. When $\cL_\varphi=\SO_4(\C)$, we have $\Cent_{G^{\vee}}(\Cent_{\cL_{\varphi}}^{\circ})=G^{\vee}$, which corresponds to a supercuspidal member in the $L$-packet for $\varphi$. When $\cL_\varphi$ is a maximal torus, we have $\Cent_{G^{\vee}}(\Cent_{\cL_{\varphi}}^{\circ})$ is also a torus, which gives rise to a principal series representation in the $L$-packet for $\varphi$. Moreover, since the $L$-parameter is bounded, by Property~\ref{property:L-packets} the representations are tempered. Either $\varphi|_{\SL_2}$ is:
\begin{enumerate}
    \item trivial. Here, $S_\varphi\cong\mu_2$. The $L$-packet consists of irreducible constituents of $\chi_1\times\chi_1\rtimes1$. This is case~\ref{sp4-1ai} and the $L$-packet is $\{\widehat\chi_1\rtimes T^1_{\widehat\chi_1},\widehat\chi_1\rtimes T^2_{\widehat\chi_1}\}$. 
    \label{sp4-galois-7bi}
    \item the embedding into the first copy of $\SL_2(\C)$. Here, $S_\varphi=1$ and the $L$-packet consists of an irreducible constituent of $\nu^{1/2}\chi_1\times\nu^{-1/2}\chi_1\rtimes1$, which is Theorem~\ref{sp4-induced}\eqref{sp4-1b}. Since the member is tempered, the packet is $\{\chi_1\St_{\GL_2}\rtimes1\}$.
    \label{sp4-galois-7bii} 
    \item the diagonal embedding of $\SL_2(\C)$. Here, $S_\varphi\cong\mu_2^2$. Concretely, the $L$-parameter $\varphi$ may be viewed as the $W_F\times\SL_2(\C)$-representation $U=M_2(\C)\oplus\C$ where $W_F$ acts on $M_2(\C)$ by $\chi_1$ and $\SL_2(\C)$ acts on $M_2(\C)$ by conjugation. The symmetric form is the trace pairing on $M_2(\C)$. \label{sp4-galois-7biii}
\end{enumerate}

Thus in case~\ref{sp4-galois-7biii} the members of the size four $L$-packet consists of two supercuspidals and two principal series. The $L$-packet is determined in Section~\ref{section-mixed-packets-GSp4}. 
\label{sp4-case7b}

\item $\chi_2=1$ and $\chi_1$ is of order $2$. We have $\cG_\varphi=\tS(\tO_3\times \tO_2)\cong\SO_3\times \tO_2$. Since both the Springer correspondence for $\SO_3$ and $\tO_2$ do not have any nontrivial cuspidal supports (by Table~\ref{Springer-SO3}), the members of $L$-packets are principal series. Moreover, again the $L$-packet is bounded, so by Property~\ref{property:L-packets} the representations are tempered.
\begin{enumerate}
    \item if $\varphi|_{\SL_2}=1$, then $S_\varphi=\mu_2$ and the packet consists of irreducible constituents of $\chi_1\times1\rtimes1$. This is case~\ref{sp4-1ai}, so the $L$-packet is $\{1\rtimes T_{\chi_1}^1,1\rtimes T_{\chi_1}^2\}$.
    \item if $\varphi|_{\SL_2}$ is non-trivial, then $S_\varphi=\mu_2$ and the packet consists of irreducible constituents of $\chi_1\times\nu^{1/2}\rtimes1$. This is case~\ref{sp4-1ai} and the $L$-packet is $\{\nu^{1/2}\rtimes T_{\chi_1}^1,\nu^{1/2}\rtimes T_{\chi_1}^2\}$.
\end{enumerate}

        \item $\chi_2=1$ and $\chi_1^2\ne1$. Here $\cG_\varphi=\SO_3(\C)\times\SO_2(\C)$. By Table~\ref{Springer-SO3} the unipotent pairs are all supported in the torus, so the $L$-packets are singletons consisting of a principal series. The restriction $\varphi|_{\SL_2(\C)}$ is either:
        \begin{enumerate}
            \item trivial, then the packet is $\{\chi_1\times1\rtimes1\}$, where $\chi_1\times1\rtimes 1$ is irreducible by Theorem~\ref{sp4-induced}\eqref{sp4-1a}. 
            \item nontrivial, then the packet is $\{\chi_1\times\nu^{1/2}\rtimes1\}$, where $\chi_1\times\nu^{1/2}\rtimes1$ is irreducible by Theorem~\ref{sp4-induced}\eqref{sp4-1a}. 
        \end{enumerate}
        
        \label{case7c}
        \item $\chi_1\ne\chi_2$ are distinct order $2$ characters. Here $\cG_\varphi=\tS(\tO_2(\C)\times \tO_2(\C)\times\mu_2)\cong\tO_2(\C)^2$. Here $S_\varphi=\mu_2^2$ and by Lemma~\ref{no-levi-lemma} the $L$-packet members are principal series. The $L$-packet consists of the irreducible constituents of $\chi_1\times\chi_2\rtimes1$, which has length $4$ by Theorem~\ref{sp4-induced}\eqref{sp4-1aii}.

        \label{sp4-galois-7e}
        \item $\chi_1=\chi_2^{-1}\ne1$ and $\chi_1^2\ne1$. Here $\cG_\varphi=\GL_2(\C)$ and $S_\varphi=1$. Here $\cL_\varphi\subset\GL_2(\C)$ is the maximal torus, so the $L$-packet consists of principal series representations.
        \begin{enumerate}
            \item if $\varphi|_{\SL_2}$ is trivial, then the $L$-packet is $\{\chi_1\times\chi_1\rtimes1\}$, where irreducibility is by Theorem~\ref{sp4-induced}\eqref{sp4-1a}.
            \item if $\varphi|_{\SL_2}$ is nontrivial, then the member is a irreducible constituent of $\nu^{1/2}\chi_1\times\nu^{-1/2}\chi_1\rtimes1$. If $\chi_1\ne\nu^{\pm3/2}$ and $\chi_1^2\ne\nu^{\pm1}$ then the $L$-packet is $\{\chi_1\St_{\GL_2}\rtimes1\}$.
            
            Otherwise, if $\chi_1=\nu^{\pm3/2}$ then $\nu^{\pm3/2}\St_{\GL_2}\rtimes1$ has length two, since we are in case~\ref{sp4-1biii}. By Property~\ref{langlands-class} the $L$-packet is $\{J(\nu^{\pm3/2}\St_{\GL_2};1)\}$.
            
            If $\chi_1=\nu^{\pm1/2}$ then $\nu^{\pm1/2}\St_{\GL_2}\rtimes1$ has length two, since we are in case~\ref{sp4-1ci}.
            
            If $\chi_1=\nu^{\pm1/2}\xi_1$ for some order $2$ character $\xi_1$ then $\nu^{\pm1/2}\xi_1\St_{\GL_2}\rtimes1$ has length three, and the $L$-packet is $\{J(\nu^{1/2}\xi_1\St_{\GL_2},1)\}$.
        \end{enumerate}
        
        \label{sp4-galois-7f}
        \label{sp4-galois-7i}
        \item If $\chi_1^{\pm1}$ and $\chi_2^{\pm1}$ are all distinct, then $\cG_\varphi=\C^\times\times\C^\times$ and $\varphi|_{\SL_2(\C)}=1$ and $S_\varphi=1$. By Lemma~\ref{no-levi-lemma} the $L$-packet is a singleton $\{\chi_1\times\chi_2\rtimes1\}$, which is reducible by Theorem~\ref{sp4-induced}\eqref{sp4-1a}.
        \label{sp4-galois-7j}

    \end{enumerate}
\end{enumerate}
In particular, the only mixed packets occur in cases~\ref{sp4-case5b}~and~\ref{sp4-galois-7biii}.

We also use the following well-known fact:
\begin{lemma}[{Mackey's little groups method, \cite[\S8.2]{serre-finite-groups}}]\label{mackey-induced}
Let $G=A\rtimes H$ be a finite group, where $A$ is abelian. Then, there is a bijection
\[\Irr(G)\cong\{\chi\in H\backslash A^*,\rho\in\Irr(H^\chi)\},\]
where $H\backslash A^*$ denotes the set of $H$-orbits in $A^*=\hom(A,\C^\times)$ and $H^\chi$ is the stabilizer of $\chi$. A pair $(\chi,\rho)$ corresponds to the irreducible $G$-representation $\Ind_{A\rtimes H^\chi}^G(\widetilde\chi\otimes\rho)$, where $\widetilde\chi(ah):=\chi(a)$ for $a\in A$ and $h\in H^\chi$.
\end{lemma}

Now let $G=\GSp_4(F)$ and $\varphi\colon W_F\times\SL_2(\C)\to G^\vee\cong\GSp_4(\C)$ an $L$-parameter. Now $\varphi|_{W_F}$ can be considered a $4$-dimensional $W_F$-representation $U$ with a invariant symplectic form $\omega\colon U\otimes U\to\xi$, where $\xi$ is the similitude character. Now $U$ decomposes into irreducible representations according to partitions $[4]$, $[2^2]$, $[2,1^2]$, or $[1^4]$ (the partition $[3,1]$ is impossible since the attached bilinear form is necessarily symmetric). Then, $\cG_\varphi$ is the group of $W_F$-representation endomorphisms $g\colon U\to U$ such that the following diagram commutes for some constant $c\in\C^\times$ (the similitude):
\[ \begin{tikzcd}
U\otimes U \arrow{r}{\omega} \arrow[swap]{d}{g\otimes g} & \xi \arrow{d}{c} \\
U\otimes U \arrow{r}{\omega}& \xi.
\end{tikzcd}
\]
Thus there are the following cases: 
\begin{enumerate}
    \item $U$ is irreducible with $U\cong \xi U^\vee$ and the unique pairing $U\otimes U\to\xi$ is anti-symmetric. Here $\cG_\varphi=\C^\times$ and $S_\varphi=1$ so the packet is a singleton supercuspidal.   \item $U=V_1\oplus V_2$ where $V_1$ and $V_2$ are irreducible of dimension $2$. Either:
    \begin{enumerate}
        \item\label{gsp4-galois-2a} $V_1\cong V_2$, with an invariant anti-symmetric form $\omega\colon V_1\otimes V_1\to\xi$. Here $\xi=\det(V_1)$. Then $U$ carries the symplectic form $\omega'(v_1\oplus w_1,v_2\oplus w_2)=\omega(v_1,w_2)+\omega(w_1,v_2)$. Thus, $\cG_\varphi=\GO_2(\C)\cong(\C^\times)^2\rtimes\mu_2$, embedded as
        \(
        \begin{pmatrix}aI_2&bI_2\\cI_2&dI_2\end{pmatrix}\in\GSp_4(\C)
        \) and $S_\varphi=\mu_2$. By Remark~\ref{no-levi-lemma}, the $L$-parameter is supported in $\GL_2(\C)\times\GSp_0(\C)$, so the representations are supported in $\GL_1(F)\times\GSp_2(F)$. The cuspidal support of $\varphi$ is $V_1$ and $\xi$ viewed as an $L$-parameter $W_F\to\GL_2(\C)\times\GSp_0(\C)$. By Remark~\ref{gsp4-levi-llc-remark} to the representation $\widehat\xi^{-1}\det(\pi_{V_1})\boxtimes\pi_{V_1}^\vee=1\boxtimes\pi_{V_1}^\vee$ of $\GL_1(F)\times\GSp_2(F)$, which is the cuspidal support of $\varphi$. Here, $\pi_{V_1}$ is the representation of $\GSp_2(F)$ corresponding to $V_1$ under LLC for $\GSp_2(F)$. Thus the members of the $L$-packet are the two irreducible constituents of $1\rtimes\pi_{V_1}^\vee$ (this is case~\ref{gsp4-3a}).
        
        \item $V_1\cong V_2$, with an invariant symmetric form $\langle-,-\rangle\colon V_1\otimes V_1\to\xi$. Here, $\xi=\det(V_1)$. Then $\omega(v_1\oplus w_1,v_2\oplus w_2)=\langle v_1,w_2\rangle-\langle v_2,w_1\rangle$.
        
        Thus, $\cG_\varphi=\GL_2(\C)$ embedded as $\diag(g,J{^T}g^{-1}J^{-1})\in\GSp_4(\C)$ and $S_\varphi=1$. Letting $T\subset\cG_\varphi$ be a maximal torus the (trivially) enhanced $L$-parameters are supported in $\Cent_{G^\vee}(T)=\GL_1\C\times\GSp_2\C$, so the members of packets are supported in $\GL_2F\times\GSp_0F$.
        
        \begin{enumerate}
            \item If $\varphi(\SL_2)=1$ then the cuspidal support of $\varphi$ is $\xi$ and $V$ viewed as a $L$-parameter $W_F\to\GL_1\C\times\GSp_2\C$. By Remark~\ref{gsp4-levi-llc-remark}, the member of the $L$-packet is an irreducible constituent of $(\widehat\xi\otimes\pi_{V_1}^\vee)\rtimes\widehat\xi^{-1}$.
            
            We are in case~\ref{gsp4-20-levi}. Since $V_1\cong\xi V_1^\vee$ we have $\pi_{V_1}\cong\widehat\xi\otimes\pi_{V_1}^\vee$. Thus if $\xi=\nu^\beta\xi'$ for a unitary character $\xi'$ and $\beta\in\R$ then $\pi_{V_1}\rtimes\widehat\xi^{-1}$ is irreducible as long as $\beta\ne\pm1$. In this case the $L$-packet is $\{\pi_{V_1}\rtimes\widehat\xi^{-1}\}$.
            
            Otherwise since the $L$-parameter $\varphi$ is not (essentially) bounded the singleton $L$-packet consists of the unique essentially tempered subquotient of $\pi_{V_1}\rtimes\widehat\xi^{-1}$.

            \item If $\varphi|_{\SL_2}$ is nontrivial then the cuspidal support of $\varphi$ is $\nu\xi$ and $\nu^{1/2}V$ viewed as a $L$-parameter $W_F\to\GL_1\times\GSp_2(\C)$. By Remark~\ref{gsp4-levi-llc-remark}, the member of the $L$-packet is an irreducible constituent of $(\nu^{1/2}\widehat\xi\otimes\pi_{V_1}^\vee)\rtimes\nu^{-1}\widehat\xi^{-1}\cong\nu^{1/2}\pi_{V_1}\rtimes\nu^{-1}\widehat\xi^{-1}$. Letting $\xi=\nu^\beta\xi'$ as above, if $\beta\notin\{0,-2\}$ then the singleton $L$-packet consists of the unique essentially tempered subquotient of $\nu^{1/2}\pi_{V_1}\rtimes\nu^{-1}\widehat\xi^{-1}$, by Property~\ref{property:L-packets}.
        \end{enumerate}
        
        \item $V_1\ncong V_2$ then $V_1\cong\xi\otimes V_2^\vee$ and so $\cG_\varphi=\C^\times\times\C^\times$ and $S_\varphi=1$. Here, $\xi=\det(V_1)$. By Lemma~\ref{no-levi-lemma} the $L$-parameter is supported in $\GL_2(\C)\times\GSp_0(\C)$, given by $(V_1,\xi)$ viewed as an $L$-parameter $W_F\to\GL_2(\C)\times\GSp_0(\C)$. Thus by Remark~\ref{gsp4-levi-llc-remark} the $L$-packet member is an irreducible constituent of $1\rtimes\pi_{V_1}^\vee$, where $\pi_{V_1}$ is the supercuspidal representation of $\GL_2(F)$ corresponding to $V_1$ under the LLC for $\GL_2(F)$.
    \end{enumerate}
    \item $U=V\oplus\chi_1\oplus\chi_2$ where $V$ is irreducible of dimension $2$ and $\chi_1,\chi_2$ are characters of $W_F$. There is an anti-symmetric pairing $\omega\colon V\otimes V\to\xi$, where $\xi=\det(V)$. Moreover, $\chi_1\chi_2=\xi$ and there is an anti-symmetric pairing $\omega'$ on $\chi_1\oplus\chi_2$ given by $\omega'(a_1\oplus b_1,a_2\oplus b_2)=a_1b_2-a_2b_1$. Either:
    \begin{enumerate}
        \item $\chi_1=\chi_2$, then $\cG_\varphi=\{(z,g)\in\C^\times\times\GL_2(\C):z^2=\det(g)\}\cong\C^\times\times\SL_2(\C)$. By Table~\ref{table-sl2-springer} there are two cases:
        \begin{enumerate}
            \item $\varphi|_{\SL_2}=1$, in which case the unipotent pair is supported in $\C^\times\times T$. Then $S_\varphi=1$ and the $L$-parameter is supported in $\GL_1(\C)\times\GSp_2(\C)$. The support is $V$ and $\chi_1$ viewed as an $L$-parameter $W_F\to\GL_1(\C)\times\GSp_2(\C)$. Thus by Remark~\ref{gsp4-levi-llc-remark}, the packet is $\{(\widehat\chi_1\otimes\pi_V^\vee)\rtimes\widehat\chi_1^{-1}\}$. Here, $(\widehat\chi_1\otimes\pi_V^\vee)\rtimes\widehat\chi_1^{-1}$ is irreducible by Theorem~\ref{gsp4-induced}, since $\det(\chi_1\otimes V^\vee)=1$ implies the representation $\widehat\chi_1^{-1}$ is unitary.  
            \item\label{gsp4-galois-3aii} $\varphi|_{\SL_2}$ is regular unipotent, in which case the unipotent pair is supported in either $\C^\times\times T$ or $\C^\times\times\SL_2(\C)$. Thus the $L$-packet is of size $2$, with an intermediate series supported in $\GL_2(F)\times\GSp_0(F)$ and a supercuspidal representation. This packet is determined in Section~\ref{section-mixed-packets-GSp4}.
        \end{enumerate}
        \item $\chi_1\ne\chi_2$ and $\chi_1\chi_2=\xi$ then $\cG_\varphi=\{(z,g)\in\C^\times\times T:z^2=\det(g)\}\cong\C^\times\times\C^\times$, embedded as $\begin{pmatrix}a\\&z\\&&z\\&&&b\end{pmatrix}\in\GSp_4(\C)$ where $ab=z^2$. Here $S_\varphi=1$ and the enhanced $L$-parameter is supported in $\GL_1(\C)\times\GSp_2(\C)$, given by $\chi_1$ and $V$ viewed as an $L$-parameter $W_F\to\GL_1(\C)\times\GSp_2(\C)$. Thus the $L$-packet member is an irreducible constituent of $(\widehat\chi_1\otimes\pi_V^\vee)\rtimes\widehat\chi_1^{-1}$.
        
        We are in case~\ref{gsp4-20-levi} of Theorem~\ref{gsp4-induced}. Let $\beta=e(\chi_1\chi_2^{-1}):=\log_q(\chi_1\chi_2^{-1}(\varpi))$.
        Then $(\widehat\chi_1\otimes\pi_V^\vee)\rtimes\widehat\chi_1^{-1}$ is irreducible unless $\beta\in\{\pm1\}$. If $\beta\in\{\pm1\}$ then the $L$-packet member is the unique essentially non-tempered subquotient of $(\widehat\chi_1\otimes\pi_V^\vee)\rtimes\widehat\chi_1^{-1}$, since the $L$-parameter $\varphi$ is not bounded.
        
    \end{enumerate}
    
    \item\label{gsp4-galois-4} $U=\chi_1\oplus\chi_2\oplus\chi_3\oplus\chi_4$ where $\chi_i$ are characters of $W_F$. Either:
    \begin{enumerate}
        \item\label{gsp4-galois-4a} $\chi_1=\chi_2=\chi_3=\chi_4$ and $\chi_1^2=\xi$, then $\cG_\varphi=G^\vee$. The Springer correspondence of $G^\vee=\GSp_4(\C)$ is (by the classification in Remark~\ref{nilpotent-sp4}): 
        {\begin{center}
\begin{tabular}{ |c|c| } 
 \hline
 Unipotent pairs& Representations of $W=\mu_2^2\rtimes S_2$\\ \hline
 $([4],1)$&$(\emptyset,[1^2])$\\
 $([2^2],1)$&$([1],[1])$\\
 $([2^2],-1)$&$(\emptyset,[2])$\\
 $([2,1^2],1)$&$([1^2],\emptyset)$\\
 $([1^4],1)$&$([2],\emptyset)$\\
 \hline
\end{tabular}
\end{center}}
Here, again the representations of $W$ are parametrized by Lemma~\ref{mackey-induced} (see also, \cite[Theorem~10.1.2]{Collingwood-McGovern}\footnote{Note that our normalization of the Springer correspondence differs with \cite{Collingwood-McGovern} by a $\sgn$-twist.}). Since all the unipotent pairs are supported in the torus, all representations here are principal series. By Property~\ref{infinitesimal-prop}, the cuspidal support of the $L$-parameter is
\[
   \varphi_v(w)=\lambda_{\varphi_v}(w)=\lambda_{\varphi}(w)=\chi_1(w)\varphi(\diag(\|w\|^{1/2},\|w\|^{-1/2})).\]
By Remark~\ref{nilpotent-sp4} all nilpotent orbits in $G^\vee$ are induced from some regular nilpotent orbit in a Levi subgroup $L^\vee\subset G^\vee$. Thus $\varphi(\|w\|^{1/2},\|w\|^{-1/2})$ is dual to the modulus character $\delta_{B_L\backslash L}$. Thus, by Remark~\ref{gsp4-levi-llc-remark}, we have $\varphi_v=\widehat\chi_1^{-1}\delta_{B_L\backslash L}$. Thus the $L$-packet contains an irreducible subquotient of $\widehat\chi_1^{-1}\ii_P^G(\St_L)$.

\begin{enumerate}
    \item If $\varphi|_{\SL_2}$ is $[4]$, then the $L$-packet member is a subquotient of $\widehat\chi_1^{-1}\ii^G_B(\delta_{B\backslash G})$, which is square-integrable modulo center, by Property~\ref{property:L-packets}. Thus the $L$-packet is $\{\widehat\chi_1^{-1}\St_{\GSp_4}\}$.
    \item If $\varphi|_{\SL_2}$ is $[2^2]$ then $S_\varphi=\mu_2$, then the $L$-packet members are irreducible constiuents of $1\rtimes \widehat\chi_1^{-1}\St_{\GL_2}$. This is case~\ref{gsp4-1bi} and the $L$-packet is $\{\tau(S,\nu^{-1/2}\widehat\chi_1^{-1}),\tau(T,\nu^{-1/2}\widehat\chi_1^{-1})\}$.
    \item If $\varphi|_{\SL_2}$ is $[2,1^2]$. The $L$-packet members is $\St_{\GL_2}\rtimes\widehat\chi_1^{-1}$, which is case~\ref{gsp4-1ai}. 
    \item If $\varphi|_{\SL_2}$ is trivial, then the $L$-packet is $\{1\times1\rtimes\widehat\chi_1^{-1}\}$, where $1\times1\rtimes\widehat\chi_1^{-1}$ is irreducible by Lemma~\ref{reducibility-lemma}.
\end{enumerate}
        \item\label{gsp4-galois-4b} $\chi_1=\chi_2\ne\chi_3=\chi_4$ and $\chi_1^2=\chi_3^2=\xi$, then $\cG_\varphi=\{(g,h)\in\GSp_2\times\GSp_2:\mu(g)=\mu(h)\}$. Thus $W_F\to T^\vee\subset\GSp_4(\C)$ is given by $(\chi_1,\chi_3,\chi_3,\chi_1)$. The Springer correspondence for $\cG_\varphi$ is:
        
        {\begin{center}
\begin{tabular}{ |c|c| } 
 \hline
 Unipotent pairs& Representations of $W=\mu_2^2$\\ \hline
 $(00,1)$&$1\otimes1$\\
 $(0e,1)$&$1\otimes\sgn$\\
 $(e0,1)$&$\sgn\otimes1$\\
 $(ee,1)$&$\sgn\otimes\sgn$\\
 $(ee,-1)$&cuspidal\\
 \hline
\end{tabular}
\end{center}}
In all cases the image $\varphi(W_F)$ is compact modulo center, so by Property~\ref{property:L-packets} the representations in the $L$-packets are essentially tempered. Either:
\begin{enumerate}
    \item If $\varphi(\SL_2(\C))=1$, then $S_\varphi=1$. The $L$-parameter is supported in $\chi_1\otimes\chi_3\otimes\xi$, so by Remark~\ref{gsp4-levi-llc-remark}, the member is an irreducible constituent of $\widehat\chi_1^{-1}\widehat\chi_3\times\widehat\chi_1^{-1}\widehat\chi_3\rtimes\widehat\chi_1^{-1}$. By \cite[Lem~3.2]{sally-tadic} this is irreducible.
    \item If $\varphi|_{\SL_2(\C)}$ is the embedding to the first factor of $\cG_\varphi$, then $S_\varphi=1$ and the $L$-parameter is supported in $\chi_3\otimes \nu^{1/2}\chi_1\otimes\xi$. Thus by Remark~\ref{gsp4-levi-llc-remark} the member is an irreducible constituent of $\nu^{1/2}\widehat\chi_1\widehat\chi_3^{-1}\times\nu^{-1/2}\widehat\chi_1\widehat\chi_3^{-1}\rtimes\widehat\chi_1^{-1}$. This is case~\ref{gsp4-1biii}, so the $L$-packet is $\{\widehat\chi_1\St_{\GL_2}\rtimes\widehat\chi_1^{-1}\}$.
    \item If $\varphi|_{\SL_2(\C)}$ is the embedding to the second factor of $\cG_\varphi$, swap the role of $\chi_1$ and $\chi_3$ and we are in the case above.
    \item\label{gsp4-galois-4b-iv} If $\varphi|_{\SL_2(\C)}$ is regular we have $S_\varphi=\mu_2$, and the corresponding unipotent pairs have support in either $T^\vee$ or $\cG_\varphi$. Thus the packet is of size $2$ consisting of a principal series and a supercuspidal. The $L$-packet is determined in Section~\ref{section-mixed-packets-GSp4}.
\end{enumerate}

        \item\label{gsp4-galois-4c} $\chi_1=\chi_2\ne\chi_3=\chi_4$ and $\chi_1\chi_3=\xi$, then $\cG_\varphi$ is the Levi $\GL_2(\C)\times\GSp_0(\C)$. Here $S_\varphi=1$ and the $L$-packet members are principal series, since the unipotent pairs are supported in $T^\vee$. Moreover, since the $L$-parameter factors through the Levi $\GL_2(\C)\times\GSp_0(\C)$, the $L$-packet is not discrete, and hence by Property~\ref{property:L-packets} the members are not square-integrable modulo center.
        \begin{enumerate}
            \item if $\varphi(\SL_2)=1$, then the $L$-parameter has support $\chi_1\otimes\chi_1\otimes\xi$. Thus $L$-packet is $\{\widehat\chi_3^{-1}\widehat\chi_1\times1\rtimes\widehat\chi_1^{-1}\}$, where irreducibility is by Lemma~\ref{reducibility-lemma}.
            \item if $\varphi|_{\SL_2}$ is nontrivial, then the $L$-parameter has support $\nu^{1/2}\chi_1\otimes\nu^{-1/2}\chi_1\otimes\xi$, so the $L$-packet member is an irreducible constituent of $\widehat\chi_3^{-1}\widehat\chi_1\times\nu\rtimes\nu^{-1/2}\widehat\chi_1^{-1}$. If $\widehat\chi_3^{-1}\widehat\chi_1\notin\{1,\nu^{\pm1},\nu^{\pm2}\}$ then this is case~\ref{gsp4-1aii} and the $L$-packet is $\{\widehat\chi_3^{-1}\widehat\chi_1\rtimes\widehat\chi_1^{-1}\St_{\GSp_2}\}$ by Property~\ref{langlands-class}.
            
            If $\widehat\chi_3^{-1}\widehat\chi_1=\nu^{\pm1}$ then we are in case~\ref{gsp4-1bii} and the $L$-packet must be $\{\nu^{3/2}\St_{\GL_2};\nu^{-1/2}\widehat\chi_1^{-1}\}$. 
            
            Otherwise, $\widehat\chi_3^{-1}\widehat\chi_1=\nu^{\pm2}$ and we are in case~\ref{gsp4-1aiii}. By Property~\ref{langlands-class} the $L$-packet is $\{J(\nu^2;\widehat\chi_1^{-1}\St_{\GSp_2})\}$. 
        \end{enumerate}
        
        \item\label{gsp4-galois-4d} $\chi_1=\chi_2\ne\chi_3\ne\chi_4$ and $\chi_1^2=\chi_3\chi_4=\xi$ then the $L$-parameter $\varphi|_{W_F}\colon W_F\to T^\vee\hookrightarrow G^\vee$ is given by $(\chi_3,\chi_1I_2,\chi_4)$. Here $\cG_\varphi$ is the Levi $\GL_1(\C)\times\GSp_2(\C)$, so by Property~\ref{property:L-packets} the $L$-packet members are not square-integrable modulo center. Here $S_\varphi=1$ and the $L$-packet members are principal series.
        \begin{enumerate}
            \item if $\varphi(\SL_2)=1$ then the $L$-parameter has support $\chi_1\otimes\chi_3\otimes\xi$. The $L$-packet consists of a subquotient of $\widehat\chi_1^{-1}\widehat\chi_3\times\widehat\chi_1\widehat\chi_3^{-1}\rtimes\widehat\chi_1^{-1}$. There are several cases:
            \begin{itemize}
                \item If $(\widehat\chi_1^{-1}\widehat\chi_3)^2=\nu^{\pm1}$, then we are in case~\ref{gsp4-1ai} and the $L$-packet is $\{\nu^{\mp1/2}\widehat\chi_1^{-1}\widehat\chi_31_{\GL_2}\rtimes\widehat\chi_1^{-1}\}$.
                \item If $\widehat\chi_1^{-1}\widehat\chi_3=\nu^{\pm1}$ then we are in case~\ref{gsp4-1bii} and the $L$-packet is $\{\nu\rtimes\nu^{-1/2}\widehat\chi_1^{-1}1_{\GSp_2}\}$.
                \item Otherwise by Lemma~\ref{reducibility-lemma} the packet is $\{\widehat\chi_1^{-1}\widehat\chi_3\times\widehat\chi_1\widehat\chi_3^{-1}\rtimes\widehat\chi_1^{-1}\}$.
            \end{itemize}
        \end{enumerate}

        \item\label{gsp4-galois-4e} $\chi_1\ne\chi_2\ne\chi_3\ne\chi_4$ with $\chi_1\chi_4=\chi_2\chi_3$ then $\cG_\varphi$ is the maximal torus. Thus $S_\varphi=1$ and the $L$-parameter is supported in $\chi_1\otimes\chi_2\otimes\xi$. The $L$-packet member is an irreducible subquotient of $\widehat\chi_1\widehat\chi_3^{-1}\times\widehat\chi_1\widehat\chi_2^{-1}\rtimes\widehat\chi_1^{-1}$.
        
        If $\widehat\chi_i\widehat\chi_j^{-1}$ is not of the form $\nu^{\pm1}$ for any $i\ne j$ then this is irreducible by \cite[Lem~3.2]{sally-tadic}. Otherwise:
        \begin{itemize}
            \item if $\widehat\chi_1\widehat\chi_2^{-1}=\nu$ and $\widehat\chi_1\widehat\chi_3^{-1}\notin\{1,\nu^{\pm1},\nu^{\pm2}\}$ then we are in case~\ref{gsp4-1aii} and the $L$-packet is $\{\widehat\chi_1\widehat\chi_3^{-1}\rtimes\nu^{1/2}\widehat\chi_1^{-1}1_{\GSp_2}\}$.
            \item if $\widehat\chi_1\widehat\chi_2^{-1}=\widehat\chi_2\widehat\chi_3^{-1}=\nu$ then we are in case~\ref{gsp4-1aiii} and the $L$-packet is $\{\nu^{3/2}\widehat\chi_1^{-1}1_{\GSp_4}\}$.
        \end{itemize}
    \end{enumerate}
\end{enumerate}

The mixed packets are cases~\ref{gsp4-galois-3aii}~and~\ref{gsp4-galois-4b}.

\section{Mixed packets}\label{section-mixed-packets-GSp4}

Denote the three order $2$ characters of $F^\times$ as $\eta,\eta_2,\eta_2'$, where $\eta(x):=(-1)^{v_F(x)}$ is unramified and $\eta_2$ and $\eta_2'$ are ramified quadratics.

\subsection{The \texorpdfstring{$\GSp_4$}{GSp4} case}
The mixed packet for $\GSp_4$ occurs in:
\begin{enumerate}
\item case~\ref{gsp4-galois-3aii}
\begin{proof}
    In case~\ref{gsp4-galois-3aii}, let $\varphi_v=(\chi_1,\chi_1\varphi_u)\colon W_F'\to\GL_1(\C)\times\GSp_2(\C)$ be the cuspidal support of the intermediate series, where $\varphi_v|_{\SL_2}=1$ by Remark~\ref{gl_n-cuspidal-l} and $\det(\varphi_u)=1$. By Property~\ref{infinitesimal-prop} we have $\varphi_v(w,x)=\lambda_{\varphi_v}(w)=\lambda_\varphi(w)$. Here,
    \[
    \lambda_\varphi(w)=\diag(\|w\|^{1/2}\chi_1(w),\chi_1(w)\varphi_u(w),\|w\|^{-1/2}\chi_1(w))
    \]
    so $\varphi_v(w)=\|w\|^{1/2}\chi_1(w)\otimes\chi_1(w)\varphi_u(w)$. By Remark~\ref{gsp4-levi-llc-remark} this corresponds to the representation $\nu^{1/2}\pi_u\boxtimes\nu^{-1/2}\widehat\chi_1^{-1}$ where $\pi_u$ is the self-dual supercuspidal representation of $\PGL_2(F)$ corresponding to $\varphi_u$ under the LLC for $\PGL_2(F)$. Thus the intermediate series member of the $L$-packet is an irreducible subquotient of $\nu^{1/2}\pi_u\rtimes\nu^{-1/2}\widehat\chi_1^{-1}$. By Theorem~\ref{gsp4-induced} \eqref{gsp4-20-levi} it has a unique irreducible sub-representation $\delta(\nu^{1/2}\pi_u\rtimes\nu^{-1/2}\widehat\chi_1^{-1})$, which is square-integrable. Thus $\delta(\nu^{1/2}\pi_u\rtimes\nu^{-1/2}\widehat\chi_1^{-1})\in\Pi_\varphi$.

    \begin{itemize}
        \item when the $\PGL_2(F)$-representation $\pi_u$ has depth zero, it is classified by a regular depth-zero character $\theta\colon E^\times/F^\times\to\C^\times$, where $E/F$ is the unramified quadratic extension.
        \begin{equation}
        \Pi_{\varphi(\theta)}:=\left\{\delta\left(\nu^{1/2}\pi_{(E^\times,\theta)}\rtimes\nu^{-1/2}\widehat\chi_1^{-1}\right),\quad\pi_{(S,\theta\boxtimes\theta\otimes\widehat\chi_1^{-1})}\right\},
    \end{equation}
    where the supercuspidal $\pi_{(S,\theta\boxtimes\theta\otimes\widehat\chi_1^{-1})}$ is defined in Lemma~\ref{depth-zero-sc-GSp4}.
    \item when the $\GL_2$-representation $\pi_u$ has positive depth, the $L$-packet is of the form
    \begin{equation}
        \Pi_{\varphi}:=\{\delta(\nu^{1/2}\pi_u\rtimes\nu^{-1/2}\widehat\chi_1^{-1}),\pi(\pi_u)\otimes\widehat\chi_1^{-1}\},
    \end{equation}
    where:
    \begin{itemize}
        \item $\pi_u$ is a supercuspidal representation of $\GL_2(F)$, which corresponds to a nontrivial representation $\mathrm{JL}(\pi_u)$ of $D^\times/F^\times$ under the Jacquet-Langlands correspondence, for $D/F$ the quaternion algebra. The Kim-Yu type is given by a twisted Levi sequence $(G^0\subset\cdots\subset G^d=D^\times/F^\times)$.
        \item $\pi(\pi_u)$ has Kim-Yu type given by the twisted Levi sequence $(G^0\subset\cdots\subset G^d=D^\times/F^\times\subset\GSp_4(F))$.
    \end{itemize}
    \end{itemize}
\end{proof}
    \item case~\ref{gsp4-galois-4b-iv}
    \begin{proof}
        In case~\ref{gsp4-galois-4b}, let $\varphi_v\colon W_F'\to T^\vee$ be the cuspidal support of the principal series, where since $T^\vee$ has no unipotents, we have $\varphi_v|_{\SL_2}=1$. By Property~\ref{infinitesimal-prop} we have $\varphi_v(w,x)=\lambda_{\varphi_v}(w)=\lambda_\varphi(w)$. Here,
\[
\lambda_\varphi(w)=\diag(
\|w\|^{1/2}\chi_1(w),\|w\|^{1/2}\chi_3(w),\|w\|^{-1/2}\chi_3(w),\|w\|^{-1/2}\chi_1(w)).
\]
Under the isomorphism of Remark~\ref{gsp4-self-dual}.
the $L$-parameter $\varphi$ corresponds to an irreducible subquotient of $\nu\theta\times\theta\rtimes\nu^{-1/2}\widehat\chi_3^{-1}$ where $\theta:=\widehat\chi_1\widehat\chi_3^{-1}$ is an order $2$ character of $F^\times$. By \cite[Lemma~3.6]{sally-tadic} the representation $\nu\theta\times\theta\rtimes\nu^{-1/2}\widehat\chi_1^{-1}$ has a unique essentially square integrable subquotient $\delta([\theta,\nu\theta],\nu^{-1/2}\widehat\chi_1^{-1})$. Thus by Property~\ref{property:L-packets}, we have $\delta([\theta,\nu\theta],\nu^{-1/2}\widehat\chi_1^{-1})\in\Pi_{\varphi}$. 
Here $\theta\in\{\eta,\eta_2,\eta_2'\}$.

The only singular supercuspidal from Theorem~\ref{depth-zero-sc-GSp4} that's unipotent (up to twisting) is $\pi_{\beta}(\theta_{10}\otimes1)$. Therefore it must be in the $L$-packet $\Pi_{\varphi^{(1)}}$. 

There are three $L$-packets, with notation from Proposition~\ref{depth-zero-sc-GSp4}.
\begin{align*}
    \Pi_{\varphi^{(1)}}:&=\{\delta([\eta,\nu\eta],\nu^{-1/2}\widehat\chi_1^{-1}),  \pi_{\delta}(\theta_{10}\otimes\widehat\chi_1^{-1})\}\\
    \Pi_{\varphi^{(2)}}:&=\{\delta([\eta_2,\nu\eta_2],\nu^{-1/2}\widehat\chi_1^{-1}),\pi_{\alpha}(\eta_2;\widehat\chi_1^{-1})\}\\
    \Pi_{\varphi^{(3)}}:&=\{\delta([\eta_2',\nu\eta_2'],\nu^{-1/2}\widehat\chi_1^{-1}), \pi_{\alpha}(\eta_2';\widehat\chi_1^{-1})\}.
\end{align*}
Here the $L$-packets $\Pi_{\varphi^{(2)}}$ and $\Pi_{\varphi^{(3)}}$ are assembled in Proposition~\ref{main-stability-prop-GSp4} via stability of characters.~Note that the twist $\widehat\chi_3^{-1}$ can be recovered as the central character of the representations.

We now compute the formal degree of $\delta([\eta_2,\nu\eta_2],\nu^{-1/2}\widehat\chi_1^{-1})$: By \cite{Roche-principal-series}, we have 
\begin{equation}\label{Roche-isomorphism}
\Irr(\mathcal{H}(G,\tau^\fs))\xrightarrow{\sim} \Irr(\mathcal{H}(J^{\fs},1)),
\end{equation} 
under which $\delta([\eta_2,\nu\eta_2],\nu^{-1/2}\widehat\chi_1^{-1})$ corresponds to $\St_{\GL_2\times\GL_2/\G_m}$. By \cite[Theorem~10.7]{Roche-principal-series}, up to normalization factors of volumes, we have
\begin{equation}
    \fdeg(\pi(\eta_2))=d(\St_{\GL_2\times\GL_2/\G_m}^{\mathcal{H}}).
\end{equation}
Now by~\cite[Theorem 4.1]{Ciubotaru-Kato-Kato}, we have
\begin{equation}    
d(\St_{\GL_2\times\GL_2/\G_m}^{\mathcal{H}})=\frac{1}{2}\cdot\frac{1}{q^2-1}\cdot \frac{q-1}{q^2-1}\cdot q^{3/2}
    =\frac{q^{3/2}}{2(q+1)^2}.
\end{equation}
Thus we have
\begin{equation}\label{fdeg-pi-eta2}
    \fdeg(\delta([\eta_2,\nu\eta_2],\nu^{-1/2}\widehat\chi_1^{-1}))=\frac{q^{3/2}}{2(q+1)(q^2-1)},
\end{equation}
which agrees with the formal degree for the singular supercuspidal computed in \eqref{sc-formal-degree-equation}.
    \end{proof}
\end{enumerate}

\subsection{The \texorpdfstring{$\Sp_4$}{Sp4} case}
The mixed packets for $\Sp_4$ occur in:
\begin{enumerate}
    \item case~\ref{sp4-galois-7biii}, when the packet is of size $4$, consisting of two supercuspidals and two principal series the irreducible constituents of $\nu^{1/2}\chi_1\St_{\GL_2}\rtimes1$. 
    The $L$-packets consist of principal series from case~\ref{sp4-1biv}, and depth-zero supercuspidals from Theorem \ref{depth-zero-sc-Sp4}. 
\begin{proof}
To each $\widehat\chi_1=\eta,\eta_2,\eta_2'$, we denote by $\varphi(\chi_1)$ the corresponding $L$-parameter, as in case~\ref{sp4-galois-7biii}. Concretely, $\varphi(\chi_1)\colon W_F'\to\SO_5(\C)$ corresponds to the $W_F\times\SL_2(\C)$-representation $U=M_2(\C)\oplus\C$ where $W_F$ acts on $M_2(\C)$ by $\chi_1$ and $\SL_2(\C)$ acts on $M_2(\C)$ by conjugation. In particular, the $L$-packet $\Pi_{\varphi(\eta)}$ is a unipotent $L$-packet.

The principal series members $\pi_1(\widehat\chi_1),\pi_2(\widehat\chi_1)\in\Pi_{\varphi(\chi_1)}$ have unipotent pairs $(ee,(-1,\pm1))$ on $\tO_4$, by the discussion in case~\ref{sp4-galois-7biii}. Let $\varphi_v(\chi_1)\colon W_F'\to T^\vee$ be the cuspidal support, where $\varphi_v(\chi_1)(\SL_2)=1$ since $T^\vee$ does not have unipotents. Then by Property~\ref{infinitesimal-prop} we have
\[\varphi_v(\chi_1)(w,x)=\lambda_{\varphi_v(\chi_1)}(w)=\lambda_{\varphi(\chi_1)}(w)=\varphi(\chi_1)(w,\begin{pmatrix}\|w\|^{1/2}\\&\|w\|^{-1/2}\end{pmatrix}).\]

This acts on $M_2(\C)$ as:
\begin{align*}
    \lambda_{\varphi(\chi_1)}(w)(e_{11})&=\chi_1(w)e_{11}\\
    \lambda_{\varphi(\chi_1)}(w)(e_{12})&=\|w\|\chi_1(w)e_{12}\\
    \lambda_{\varphi(\chi_1)}(w)(e_{21})&=\|w\|^{-1}\chi_1(w)e_{21}\\
    \lambda_{\varphi(\chi_1)}(w)(e_{22})&=\chi_1(w)e_{22},
\end{align*}
so $\varphi_v(\chi_1)=\|\det\|\chi_1\otimes\chi_1\otimes1$.
Now $\pi_1(\chi_1)$ and $\pi_2(\chi_1)$ are subquotients of $\nu\widehat\chi_1\times\widehat\chi_1\rtimes1=\nu^{1/2}\widehat\chi_11_{\GL_2}\rtimes1+\nu^{1/2}\widehat\chi_1\St_{\GL_2}\rtimes1$. Moreover, since $\pi_1(\chi_1)$ and $\pi_2(\chi_1)$ are square-integrable by Property~\ref{property:L-packets}, they must be subquotients of $\nu^{1/2}\widehat\chi_1\St_{\GL_2}\rtimes1$. By \cite[Lemma~3.6]{sally-tadic} over $\GSp_4$ the representation $\nu\widehat\chi_1\times\widehat\chi_1\rtimes1_{F^\times}$ contains a unique square integrable subquotient $\delta([\widehat\chi_1,\nu\widehat\chi_1],1_{F^\times})$. This splits into two irreducible representations when restricted to $\Sp_4$ by \cite[Prop~5.4]{sally-tadic}, and these are exactly the square-integrable subquotients of the $\Sp_4$-representation $\nu\widehat\chi_1\times\widehat\chi_1\rtimes1$. Thus, in the Grothendieck group
\begin{equation}\label{sp4-splits}
\delta([\widehat\chi_1,\nu\widehat\chi_1],1_{F^\times})|_{\Sp_4(F)}=\pi_1(\chi_1)+\pi_2(\chi_1).
\end{equation}

For the supercuspidals in $\Pi_{\varphi(\eta)}$, there are only two unipotent supercuspidals $\pi_{\beta}(\theta_{10})$ and $\pi_{\gamma}(\theta_{10})$ coming from Theorem \ref{depth-zero-sc-Sp4}\eqref{depth-zero-sc-Sp4(1)-theta10}. Therefore these two must be in the $L$-packet $\Pi_{\varphi(\eta)}$. Note that this agrees with the unipotent $L$-packet in \cite{lust-stevens}. 
Moreover, \cite[Example~9.4]{lust-stevens} says that $\Pi_{\varphi(\eta_2)}$ and $\Pi_{\varphi(\eta_2')}$ contains the depth-zero representations inflated from $\SL_2(\F_q)\times\SL_2(\F_q)$, i.e. the ones in Theorem \ref{depth-zero-sc-Sp4}\eqref{depth-zero-sc-Sp4(2)}. 
In summary, we have three $L$-packets
\begin{align}\label{size-4-Sp4-packet-theta-10}
    \Pi_{\varphi(\eta)}:&=\{\pi_1(\eta),\pi_2(\eta), \pi_{\beta}(\theta_{10}),\pi_{\gamma}(\theta_{10})\}\\
    \Pi_{\varphi(\eta_2)}:&=\{\pi_1(\eta_2),\pi_2(\eta_2),\pi^+_{\alpha}(\eta_2), \pi_{\alpha}^-(\eta_2)\}\\
    \Pi_{\varphi(\eta_2')}:&=\{\pi_1(\eta_2'),\pi_2(\eta_2'), \pi^+_{\alpha}(\eta_2'), \pi_{\alpha}^-(\eta_2')\}.
\end{align}
The choices between $\Pi_{\varphi(\eta_2)}$ and $\Pi_{\varphi(\eta_2')}$ are pinned down in Corollary \ref{2x4-llc} via stability of characters. Similar computations as in \eqref{fdeg-pi-eta2} shows that the formal degrees of $\pi_i(\eta_2)$ and $\pi_\alpha^{\pm}(\eta_2)$ agree.
\end{proof}

 \begin{remark} The $L$-packets $\Pi_{\varphi(\eta_2)}$ and $\Pi_{\varphi(\eta_2')}$ are those in \cite[Ex~9.4]{lust-stevens}.
    \end{remark}

    \item case~\eqref{sp4-case5b}, where the packet is of size $2$ consisting of a supercuspidal and an intermediate series.
    \begin{proof}
        Let $\pi\in\Pi_\varphi$ be the intermediate series member. By Property~\ref{infinitesimal-prop} we have $\lambda_\varphi=\iota_{\GL_2}\circ\lambda_{\varphi_v}$ up to $\SO_5$-conjugacy. For the intermediate series representation, since $\varphi_v\colon W_F'\to\GL_2(\C)$ is cuspidal, by Remark~\ref{gl_n-cuspidal-l} we have $\varphi_v(w,x)=\varphi(w,\begin{pmatrix}\|w\|^{1/2}\\&\|w\|^{-1/2}\end{pmatrix})$, which acts on $U=V^2\oplus 1$ as
\[
\begin{pmatrix}\|w\|^{1/2}\varphi(w)\\&1\\&&\|w\|^{-1/2}\varphi(w)\end{pmatrix}.
\]
Thus, the $L$-parameter of the cuspidal support is $\|\det\|^{1/2}\varphi$. Let $\varphi$ correspond to the unitary representation $\sigma$ of $\GL_2(F)$ under the LLC for $\GL_2$, so $\nu^{1/2}\sigma$ is the image of $\|\det\|^{1/2}\varphi$ under the LLC for $\GL_2$. Thus, $\pi:=\pi(\sigma)$ is an irreducible sub-representation of the induced representation $\nu^{1/2}\sigma\rtimes1$, which is the unique square-integrable subquotient by \cite[Prop~5.6(iv)]{sally-tadic}. It must be the member by Property~\ref{property:L-packets}. In summary, 
\begin{itemize}
    \item 
when $\varphi$ has depth zero, the $L$-packet is of the form
\begin{equation}\label{pi-alpha-eta-mixed-packet}
    \Pi_{\varphi}:=\{\pi(\sigma),\pi_{\alpha}(\eta)\},
\end{equation}
where $\pi_{\alpha}(\eta)$ (for $\eta\neq \tau_1,\tau_2$) is the (singular) depth-zero supercuspidal from Theorem \ref{depth-zero-sc-Sp4}\eqref{depth-zero-sc-Sp4(2)}. There are $\frac{q-1}{2}$ such depth-zero $L$-packets, which agrees with the number of depth-zero supercuspidals of $\PGL_2(F)$.
\item when $\varphi$ has positive depth, 
let $\pi(\sigma)$ be the intermediate series representation with $\sigma$ a positive-depth supercupsidal of $\PGL_2$ corresponding to the character $\psi(\sigma): E^{\times}/F^{\times}\to\C^{\times}$. The LLC for $\GL_2$ (hence $\PGL_2$) gives us a canonical identification $E^{\times}/F^{\times}\xrightarrow{\sim}R_{E/F}^{(1)}\G_m$ which identifies $\psi(\sigma):E^{\times}/F^{\times}\to\C^{\times}$ with a character $\chi(\sigma): R_{E/F}^{(1)}\G_m\to\C^{\times}$. Let $\pi_{\chi}$ be the corresponding positive-depth singular supercuspidal. 
The $L$-packet in this case is of the form 
\begin{equation}
    \Pi_{\varphi}:=\{\pi(\sigma),\pi_{\chi(\sigma)}\}.
\end{equation}
\end{itemize}

    \end{proof}
\end{enumerate}

\begin{remark}\label{gl_n-cuspidal-l}
Let $\varphi\colon W_F'\to\GL_n(\C)$ be a cuspidal $L$-parameter for $\GL_n$. Then $\varphi(\SL_2)=1$.
\end{remark}

\section{Stability of \texorpdfstring{$L$}{L}-packets}\label{stability-section}

\subsection{Parahoric invariants for the $\GSp_4(F)$ case}\label{parahoric-invariants-section}
Via twisting by the character $\nu^{1/2}\widehat\chi_3\circ\mu$ of $\GSp_4$, we may focus our attention on $\delta([\eta_2,\nu\eta_2],1)$. It is characterized as the intersection of the sub-representations $\nu^{1/2}\eta_2\St_{\GL(2)}\rtimes1$ and $\nu^{1/2}\eta_2\St_{\GL(2)}\rtimes\eta_2$ of $\nu\eta_2\times\eta_2\rtimes1$.

We calculate the invariants of $\delta([\eta_2,\nu\eta_2],1)$ with respect to $G_{x+}$, where $x$ is a vertex of the Bruhat-Tits building (i.e., $\alpha$ or $\delta$).

\subsubsection{Calculating $\delta([\eta_2,\nu\eta_2],1)^{G_{\alpha+}}$}\label{eta2-parahoric-section}

\begin{defn}\label{omega2-defn}
Let $H_\alpha$ be the parahoric subgroup of $\GSp_{2,2}(F)$ defined in \S\ref{parahoric-section}, which contains the subgroup
\begin{equation}
    H_\alpha^0:=\{(g,h)\in M_2(\cO)\times \begin{pmatrix}\cO&\p^{-1}\\\p&\cO\end{pmatrix}:\det(g)=\det(h)=1\}.
\end{equation}
For a ramified quadratic character $\eta_2$ of $F^\times$, let $\varpi\in F$ be a uniformizer such that $\eta_2(\varpi)=1$ (unique up to $(\cO_F^\times)^2$). We define the following irreducible representations of $G_\beta/G_{\beta+}\cong H_\beta/H_{\beta+}$:
\begin{align}
    \omega^{\eta_2}_{\mathrm{princ}}&:=\Ind_{G_\beta^0\Cent}^{G_\beta}(R_+(\alpha_0)\boxtimes R_+(\alpha_0)^{\mathrm{diag}(\varpi,1)})\label{pi-princ+-defn} 
    \\
    \omega^{\eta_2}_{\mathrm{cusp}}&:=\Ind_{G_\beta^0\Cent}^{G_\beta}(R_+'(\theta_0)\boxtimes R_+'(\theta_0)^{\mathrm{diag}(\varpi,1)})
\end{align}
This is independent of the choice of the uniformizer $\varpi$.
\end{defn}
By \cite[Lemma 2.0.1]{G2-stability}, we have:
\begin{lemma}\label{gsp4-hecke-iso}There are canonical support-preserving Hecke algebra isomorphisms
\begin{align}
\mcH(\GSp_4/\!/I,\epsilon\otimes\epsilon\otimes1)&\cong\mcH(\GSpin_4^\vee/\!/J,\epsilon\circ\widetilde\det_1)\\
\mcH(\GSp_4/\!/I,\epsilon\otimes\epsilon\otimes\epsilon)&\cong\mcH(\GSpin_4^\vee/\!/J,\epsilon\circ\widetilde\det_2)
\end{align}
where $\GSpin_4^\vee\cong(\GL_2\times\GL_2)/\G_m$, and $\widetilde \det_i(g_1,g_2):=\det(g_i)$ are well-defined homomorphisms $\GSpin_4^\vee(F)\to F^\times/(F^\times)^2$. Under these isomorphisms $\delta([\eta_2,\nu\eta_2],1)$ corresponds to $\eta_2\circ\widetilde\det_i\otimes\St_{\GSpin_4^\vee}$.
\end{lemma}
By the Mackey formula, we have an isomorphism of 
representations of $G_\alpha/G_{\alpha+}\cong\GSp_{2,2}(\F_q)$
\begin{equation}\label{mackey-eta2-beta}
(\nu\eta_2\times\eta_2\rtimes1)^{G_{\alpha+}}\cong\bigoplus_{w\in B\backslash G_2/G_\alpha}\Ind_{G_\beta\cap wBw^{-1}/(G_{\alpha+}\cap wBw^{-1})}^{G_\alpha/G_{\alpha+}}(\epsilon\otimes\epsilon\otimes1)^w,
\end{equation}
where
\begin{equation}
B\backslash G_2/G_\alpha\cong W(G_2)/W(\GSp_{2,2})=W/\langle s_\beta,s_{2\alpha+\beta}\rangle=\{1,s_\alpha\}.
\end{equation}

Therefore, the $G_{\alpha+}$-invariants of $(\nu\eta_2\times\eta_2\rtimes1)^{G_{\alpha+}}$ gives
\begin{equation}
    (\nu\eta_2\otimes\eta_2\rtimes1)^{G_{\alpha+}}\simeq \Ind_B^{\GSp_{2,2}}(\epsilon\otimes1\otimes\epsilon\otimes1)^2
\end{equation}

Likewise, computing the $G_{\alpha+}$-invariants gives us the following
\begin{align}
    (\nu^{1/2}\eta_2\St\rtimes1)^{G_{\alpha+}}\simeq(\nu^{1/2}\eta_2\St\rtimes\eta_2)^{G_{\alpha+}}&\simeq \Ind_B^{\GSp_{2,2}}(\epsilon\otimes1\otimes\epsilon\otimes1).
\end{align}
We pin down 
the $G_{\beta+}$-invariants of $\pi(\eta_2)$ 
in Corollary~\ref{eta2-beta-invariants}. 

\begin{prop}\label{pro-iwahori-invariants}
The $I_{+}$-invariants of $\delta([\eta_2,\nu\eta_2],1)$ is
\[
\delta([\eta_2,\nu\eta_2],1)^{I_+}\cong\epsilon\otimes\epsilon\otimes1+\epsilon\otimes\epsilon\otimes\epsilon.
\]
\end{prop}
\begin{proof}
A priori we know
\[\delta([\eta_2,\nu\eta_2],1)^{I_+}\hookrightarrow (\nu\eta_2\times\eta_2\rtimes1)^{I_+}=\bigoplus_{w\in W}(\epsilon\otimes\epsilon\otimes1)^w=(\epsilon\otimes\epsilon\otimes1)^4+(\epsilon\otimes\epsilon\otimes\epsilon)^4.\]
By Lemma~\ref{gsp4-hecke-iso}, the multiplicity of $\epsilon\otimes\epsilon\otimes1$ in $\delta([\eta_2,\nu\eta_2],1)$, which is the same as the multiplicity of $\epsilon\circ\widetilde\det_1$ in the representation $\eta_2\St_{\SO_4}$, is one. Thus the same holds for all Weyl group orbits of the character.
\end{proof}

\begin{cor}\label{eta2-beta-invariants}
There is an isomorphism of $G_\alpha/G_{\alpha+}$-representations
\[
\delta([\eta_2,\nu\eta_2],1)^{G_{\alpha+}}\cong\omega_{\princ}^{\eta_2}
\]
\end{cor}
\begin{proof}
The argument is the same as in the proof of Corollary~3.0.8 in \cite{G2-stability}
. By Proposition~\ref{pro-iwahori-invariants} we conclude $\delta([\eta_2,\nu\eta_2],1)^{G_{\beta+}}$ must be an irreducible component of $\Ind_B^{\GSp_{2,2}}(\epsilon\otimes1\otimes\epsilon\otimes1)$, i.e., $\omega_\princ^{\eta_2}$ or $\omega_\princ^{\eta_2'}$. Together with Lemma~\ref{gsp4-hecke-iso} we conclude $\delta([\eta_2,\nu\eta_2],1)^{G_{\alpha+}}\cong\omega_\princ^{\eta_2}$.
\end{proof}

\subsubsection{Calculating $\delta([\eta_2,\nu\eta_2],1)^{G_{\delta+}}$}

Again by a Mackey theory calculation, we have:
\begin{align}
(\nu\eta_2\times\eta_2\rtimes1)^{G_{\delta+}}&\cong\Ind_{B(\F_q)}^{\GSp_4(\F_q)}(\epsilon\otimes\epsilon\otimes1)\\
(\nu^{1/2}\eta_2\St_{\GL_2}\rtimes1)^{G_{\delta+}}&\cong\Ind_{P_\alpha}^{\GSp_4(\F_q)}(\epsilon\St_{\GL_2}\otimes1)\\
(\nu^{1/2}\eta_2\St_{\GL_2}\rtimes\eta_2)^{G_{\delta+}}&\cong\Ind_{P_\alpha}^{G_2(\F_q)}(\epsilon\St_{\GL_2}\otimes\epsilon),
\end{align}
where $P_\alpha$ is a parabolic subgroup of $\GSp_4(\F_q)$. Thus, $\delta([\eta_2,\nu\eta_2],1)^{G_{\delta+}}$ is the intersection of $\Ind_{P_\alpha(\F_q)}^{\GSp_4(\F_q)}(\epsilon\St_{\GL_2}\otimes1)$ and $\Ind_{P_\alpha(\F_q)}^{\GSp_4(\F_q)}(\epsilon\St_{\GL_2}\otimes\epsilon)$, denoted $\omega^\epsilon_\princ$. In terms of Lusztig's equivalence \cite[Theorem~4.23]{Lusztig-characters-Princeton-book}, if $s\in \GSpin_5(\F_q)$ is of order $2$ such that its image in $\SO_5(\F_q)$ is $\diag(-1,-1,1,-1,-1)$ then $\Cent_{\GSpin_5(\F_q)}(s)=\GSpin_4(\F_q)\cong\GSp_{2,2}(\F_q)$:
\begin{equation}\label{Lusztig-equiv-GSp22}
\mathcal E(\GSp_4(\F_q),s)\cong\mathcal E(\GSp_{2,2}(\F_q),1)=\{\St_{\GSp_{2,2}},1\boxtimes\GSp_{2},\GSp_2\boxtimes1,1_{\GSp_{2,2}}\},
\end{equation}
and $\omega^{\epsilon}_\princ$ corresponds to $\St_{\GSp_{2,2}(\F_q)}$. 
Thus, in conclusion:
\begin{prop}\label{eta2-parahoric}
The following are the 
\begin{align}
    \delta([\eta_2,\nu\eta_2],1)^{G_{\delta+}}&\cong\omega^{\epsilon}_{\princ}\\
    \delta([\eta_2,\nu\eta_2],1)^{G_{\alpha+}}&\cong\omega_\princ^{\eta_2}.
\end{align}
\end{prop}

\subsection{Parahoric invariants for the supercuspidal representations}

Recall from Proposition \ref{depth-zero-sc-GSp4}\eqref{depth-zero-sc-GSp4(2)}, we defined the supercuspidal representation 
\begin{equation}\label{sc-mixed-gsp4}
\pi_\alpha(\eta_2;1):=\cInd_{G_\alpha\Cent}^{\GSp_4}(\omega_\cusp^{\eta_2}),\end{equation}
where $\omega_\cusp^{\eta_2}:=(\overline{\rho}_{(\lambda,\lambda)}^+)^{(I_2,\diag(\varpi,1))}$ is a cuspidal representation of $G_\alpha/G_{\alpha+}$. 
We may readily calculate the $G_{x+}$-invariants of the supercuspidal representation $\pi_\alpha(\eta_2;1)$, for various vertices $x$ in the Bruhat-Tits building:
\begin{lemma}\label{eta2-sc-parahoric}
Let $\pi_{\alpha}(\eta_2;1)$ be as defined in \eqref{sc-mixed-gsp4}. We have
\begin{align}
    \pi_\alpha(\eta_2;1)^{G_{\alpha+}}&\cong\omega_\cusp^{\eta_2}\\
    \pi_\alpha(\eta_2;1)^{G_{\delta+}}&=0
\end{align}
\end{lemma}
\begin{proof}
For each vertex $x$, by Mackey theory we have
\begin{align}
   \pi_\alpha(\eta_2;1)^{G_{x+}}&\cong\bigoplus_{g\in G_\alpha\backslash G_2/G_x}\Ind^{G_x}_{G_x\cap g^{-1}G_\alpha g}((\omega_\cusp^{\eta_2})^g)^{G_{x+}\cap g^{-1}G_\alpha g}\\
    &=\bigoplus_{g\in G_\alpha\backslash G_2/G_x}\Ind^{G_x}_{G_x\cap G_{g^{-1}\alpha}}((\omega_\cusp^{\eta_2})^g)^{G_{x+}\cap G_{g^{-1}\alpha}}.
\end{align}
Here,
\[
((\omega_\cusp^{\eta_2})^g)^{G_{x+}\cap G_{g^{-1}\alpha}}\cong(\omega_\cusp^{\eta_2})^{G_\alpha\cap G_{gx+}},
\]
which is $0$ unless $\alpha=gx$ since otherwise $G_\beta\cap G_{gx+}$ will contain the unipotent radical of some parabolic subgroup of $G_\alpha$, so $(\omega_\cusp^{\eta_2})^{G_\alpha\cap G_{gx+}}=0$ since $\omega_\cusp^{\eta_2}$ is cuspidal.
\end{proof}

\subsection{Stable distributions on $\GSp_4$ and $\Sp_4$}
\textit{For this section alone, we switch notation for $k$ to denote the non-archimedean local field, as we reserve the notation $F$ for the facets.} 
First we recall from \cite{DeBacker-parametrizing-nilpotent-orbits,DeBacker-parameterizing-conjugacy-classes, DeBacker-Kazhdan-G2} the general theory of invariant distributions associated to unramified tori. 
We now recall a few more precise results for later use. 
Let $J(\mathfrak{g})$ be the space of invariant distributions on $\mathfrak{g}$. Let $J(\mathcal{N})$ be the span of the nilpotent orbital integrals. 

For each Weyl group conjugacy class $[w]$ of $G$, consider pairs $(F,\mathcal{Q}_w^F)$ consisting of a facet $F\in\mathcal{B}(G)$ and the toric Green function $\mathcal{Q}_w^F$ (see for example \cite[\S 7.6]{Carter-book}) associated to the torus $S_w$ in $G_F$ corresponding to $[w]$. Let $\mathbf{S}$ be a maximal $K$-split $k$-torus in $G$ lifting the pair $(F, S_w)$. Let $X_{S_w}\in \mathrm{Lie}(\mathbf{S})(k)\subset\mathfrak{g}_F$ be a regular semisimple element for which the centralizer in $G_F$ of the image of $X_{S_w}$ in $\Lie(G_F)$ is $S_w$. 
Since $G^{\der}$ is simply-connected, the number of rational conjugacy classes in ${}^{G(K)}X_{S_w}\cap \mathfrak{g}$ is in bijection with the group of torsion points of $X_*(T)/(1-w)X_*(T)$ for the maximal torus $T$ of $G$. The following Table \ref{table-tor} is the analogue of \cite[Table 5]{DeBacker-Kazhdan-G2} for $\GSp_4$ (note that the analogous table for $\Sp_4$ is calculated in \cite{walds}, although we do not need it), which records the number of relevant rational conjugacy classes.
\begin{table}[ht]
\begin{tabular}{ |c|c| } 
 \hline
 class of $w$& $\tor[X_*(T)/(1-w)X_*(T)]$\\ \hline
 $1$& $0$\\
 $A_1$& $0$\\
 $\widetilde{A}_1$& $0$\\
 $A_1\times A_1$& $\Z/2$\\
 $C_2$& $0$\\
 \hline
\end{tabular}
\vskip2mm
\caption{\label{table-tor}}
\end{table}

For each character $\kappa$ of $\tor[X_*(T)/(1-w)X_*(T)]$, one can associate a distribution
\begin{equation}
    T_w(\kappa):=\sum\limits_{\lambda\in \tor[X_*(T)/(1-w)X_*(T)]}\kappa(\lambda)\cdot\mu_{X_{S_w}^{\lambda}},
\end{equation}
where $X_{S_w}^{\lambda}$ belongs to the $G$-conjugacy class in ${}^{G(K)}X_{S_w}\cap \mathfrak{g}$ indexed by $\lambda$. Note that $T_w(1)$ is stable for any reductive group $G$. 
On the other hand, the rational classes in ${}^{G(K)}X$ that intersect $\mathrm{Lie}(\mathbf{S})(k)$ are parameterized by the quotient 
\begin{equation}
    N(F,S_w):=[N_{G(K)}(\mathbf{S}(K))/\mathbf{S}(K)]^{\Gal(K/k)}/[N_G(\mathbf{S}(k))/\mathbf{S}(k)]
\end{equation}
We record the cardinality of the above quotient in the following table:
\begin{center}
\begin{tabular}{ |c|c|c| } 
 \hline
 class of $w$&vertex  & $|N(F,S_w)|$\\ \hline

 $A_1\times A_1$& $C_2$ &1 \\
 $A_1\times A_1$& $A_1\times A_1$ &1 \\
 \hline
\end{tabular}
\end{center}
In general, consider the set $I^c:=\{(F,\mathcal{G})\}$ (see for example \cite[\S 4.3]{DeBacker-Kazhdan-G2}) of pairs consisting of facet $F$ and a cuspidal generalized Green function on $\Lie(G_F)(\kappa_k)$, which is endowed with an equivalence relation $\sim$ as in Definition 4.1.2 \textit{loc.cit}. 

 Let $\mathfrak{g}_0$ be the set of compact elements in $\mathfrak{g}$, and $J(\mathfrak{g}_0)\subset J(\mathfrak{g})$ the subspace of distrubitions with support in $\mathfrak{g}_0$. Let $\mathcal{D}_0$ be the invariant version of the Lie algebra analogue of the Iwahori-Hecke algebra, and let $\mathcal{D}_0^0$ be the subalgebra spanned over facets contained in (the closure of) a fixed alcove $F_{\varnothing}$. We recall the following homogeneity result due to DeBacker and Waldspurger.
\begin{thm}[Waldspurger, DeBacker] 
We have 
\begin{enumerate}
    \item $\res_{\mathcal{D}_0}J(\mathfrak{g}_0)=\res_{\mathcal{D}_0}J(\mathcal{N})$.
    \item Suppose $D\in J(\mathfrak{g}_0)$. We have 
    \[\res_{\mathcal{D}_0}D=0\quad\text{if and only if }\res_{\mathcal{D}_0^0}D=0\]
    \end{enumerate}
\end{thm}
As a corollary, one has the following. 
\begin{cor} Let $D\in J(\mathfrak{g}_0)$. We have 
\[\res_{\mcD_0}D=0\quad\text{if and only if }D(\hat{\mcG}_F)=0\quad\text{for all }(F,\mcG)\in I^c/\sim\]
\end{cor}
We have the following list of stable distributions:
\begin{align}\label{list-Bst-Bunst}
\begin{split}
    D_{C_2}^{\st}&:=D_{(F_{C_2},\mathcal{Q}_{S_{C_2}}^{F_{C_2}})}\\
    D_{A_1}^{\st}&:=D_{(F_{A_1},\mathcal{Q}_{S_{A_1}}^{F_{A_1}})}\\
    D_{\widetilde{A}_1}^{\st}&:=D(F_{\widetilde{A}_1},\mathcal{Q}_{S_{\widetilde{A}_1}}^{F_{\widetilde{A}_1}})\\
    D_e^{\st}&:=D_{(F_e,\mathcal{Q}_{S_e}^{F_e})}\\
    D_{A_1\times A_1}^{\st}&:=D_{F_{C_2},\mathcal{Q}_{S_{A_1\times A_1}}^{F_{C_2}}}+D_{(F_{A_1\times A_1},\mathcal{Q}_{S_{A_1\times A_1}}^{F_{A_1\times A_1}})}\\
    D_{A_1\times A_1}^{\unst}&:=D_{F_{C_2},\mathcal{Q}_{S_{A_1\times A_1}}^{F_{C_2}}}-D_{(F_{A_1\times A_1},\mathcal{Q}_{S_{A_1\times A_1}}^{F_{A_1\times A_1}})}\\
    D_{F_{A_1\times A_1},\mathcal{G}_{\sgn}}^{\st}&:=D_{(F_{A_1\times A_1},\mathcal{G}_{\sgn})}
    \end{split}
\end{align}
Finally, we record the following result from \cite[Lemma 6.4.1]{DeBacker-Kazhdan-G2} (see also \cite[Th\'eor\'eme IV.13]{walds}) for later use. Let $\mathcal{B}^{\st}$ be the set of stable distributions in the above list \eqref{list-Bst-Bunst}. 
\begin{lemma}[Waldspurer, DeBacker-Kazhdan]\label{Waldspurger-DK-lemma}
   The elements of the set $\{\res_{\mcD_0}D|D\in \mathcal{B}^{\st}\}$ form a basis for $\res_{\mcD_0}J(\mathfrak{g}_0)\cap\res_{\mcD_0}J^{\st}(\mathfrak{g})$. 
\end{lemma}

\subsection{Characters on a neighborhood of $1$}\label{subsec:character-nbhd-1} 
In this section, we express $\delta[\eta_2,\nu\eta_2]^{G_{x+}}$ in terms of generalized Green functions, for $x=\alpha,\delta$. 
\begin{enumerate}
    \item When $F=F_{C_2}$ corresponds to the vertex $\delta$, we have that $\delta([\eta_2,\nu\eta_2],1)^{G_{\delta+}}\cong \omega_{\princ}^{\epsilon}$ corresponds to $\St_{\GSp_{2,2}(\F_q)}$ under Lusztig's equivalence \eqref{Lusztig-equiv-GSp22}. By \cite{Deligne-Lusztig}, the character of Steinberg is \begin{equation}\Ch_{\St_{\GSp_{2,2}}}=\frac{1}{4}\left(R_{A_1\times A_1}^{A_1\times A_1}-R_{A_1\times 1}^{A_1\times A_1}-R_{1\times A_1}^{A_1\times A_1}+R_{1 \times 1}^{A_1\times A_1}\right).\end{equation}
    Since Lusztig's equivalence \eqref{Lusztig-equiv-GSp22} preserves multiplicities, we have 
    \begin{equation}
       \Ch_{\pi_{\princ}^{\epsilon}}= \frac{1}{4}\left(R_{A_1\times A_1}^{C_2}-2R_{A_1}^{C_2}+R_{1}^{C_2}\right).
    \end{equation}
    Restricting to the unipotent locus, we have 
    \begin{equation}
        \Ch_{\pi_{\princ}^{\epsilon}}(u)= \frac{1}{4}\left(\mathcal{Q}_{A_1\times A_1}^{F_{C_2}}-2\mathcal{Q}_{A_1}^{F_{C_2}}+\mathcal{Q}_{1}^{F_{C_2}}\right).
    \end{equation}
    \item When $F=F_{A_1\times A_1}$ corresponds to the vertex $\alpha$, we have that $\delta([\eta_2,\nu\eta_2],1)^{G_{\alpha+}}\cong \omega_{\princ}^{\eta_2}$. The character formula can be computed in the same way as \cite[(3.4.5)]{G2-stability} and we have 
    \begin{equation}
        \Ch_{\pi_{\princ}^{\eta_2}}=\frac{1}{2}(\mathcal{Q}_1^{F_{A_1\times \widetilde{A}_1}}\pm q^*\mathcal{G}_{\sgn}).
    \end{equation}
    \item When $F=F_{A_1}$, since $\delta([\eta_2,\nu\eta_2],1)^{G_{F+}}$ is the Jacquet restriction $r^{A_1\times A_1}_{A_1}(\delta([\eta_2,\nu\eta_2],1))$, thus by \eqref{pi-princ+-defn} on the unipotent locus we have $\Ch(\delta([\eta_2,\nu\eta_2],1)^{G_{F+}})=\mathcal{Q}_1^{F_{A_1}}$;
    \item When $F=F_{\widetilde{A}_1}$, we have $\Ch(\delta([\eta_2,\nu\eta_2],1)^{G_{F+}})=\mathcal{Q}_1^{F_{\widetilde{A}_1}}$;
    \item When $F=F_{e}$, we have $\Ch(\delta([\eta_2,\nu\eta_2],1)^{G_{F+}})=2$.
\end{enumerate}
Similarly, we have for $F=F_{A_1\times A_1}$, 
\begin{equation}
    \Ch(\pi_{\alpha}(\eta_2;1)^{G_{F+}})=\frac{1}{2}(\mathcal{Q}_{A_1\times A_1}^{F_{A_1\times A_1}}\pm q^*\mathcal{G}_{\sgn}). 
\end{equation}
Therefore, we have the following 
\begin{prop}
    For any (possibly equal) ramified quadratic characters $\eta_2,\eta_2'$, the sum $\delta([\eta_2,\nu\eta_2],\varrho)+\pi_{\alpha}(\eta_2';\rho)$ has a stable character on the topologically unipotent elements. 
\end{prop}
\begin{proof}
As remarked at the beginning of \S\ref{parahoric-invariants-section}, it suffices to work with the case $\varrho=1$ in the notation $\delta([\eta_2,\nu\eta_2],\varrho)$. 
   From the discussions above, we see that for some explicitly computable constants $c_i$, 
   \begin{align*}
\Ch_{\delta([\eta_2,\nu\eta_2],1)}&=c_1\cdot \frac{1}{2}(D_{A_1\times A_1}^{\st}-D_{A_1\times A_1}^{\unst})\pm c_2\cdot D^{\st}_{(F_{A_1\times A_1},\mathcal{G}_{\sgn})}+cD_e^{\st}\\
       \Ch_{\pi_{\alpha}(\eta_2;1)}&=c_1\cdot\frac{1}{2}(D_{A_1\times A_1}^{\st}+D_{A_1\times A_1}^{\unst})\pm c_2\cdot D_{(F_{A_1\times A_1},\mathcal{G}_{\sgn})}^{\st}
   \end{align*}
   Thus by Lemma \ref{Waldspurger-DK-lemma}, the sum is always stable. 
\end{proof}

\subsection{Characters on a neighborhood of $s$} 
Let 
\[s=\begin{pmatrix}1&&&\\ &-1&&\\ &&-1&\\ &&&-1\end{pmatrix}\in\GSp_4(F)\]
be order $2$ such that $\Cent_{\GSp_4}(s)=\GSp_{2,2}$. 
By the construction in \cite[\S 7]{Adler-Korman-LCE}, the distributions $\Ch_{\delta([\eta_2,\nu\eta_2],\varrho)}$ and $\Ch_{\pi_{\alpha}(\eta_2;\varrho)}$ on $\GSp_4$ induce distributions $\Theta_{\delta([\eta_2,\nu\eta_2],\varrho)}$ and $\Theta_{\pi_{\alpha}(\eta_2;\varrho)}$ on $(\GSp_{2,2})_{0+}$, the topologically unipotent elements in $\GSp_{2,2}$, such that the attached locally constant functions are compatible (see \cite[Lemma 7.5]{Adler-Korman-LCE}). We shall see when the sum $\Theta_{\delta([\eta_2,\nu\eta_2],\varrho)}+\Theta_{\pi_{\alpha}(\eta_2';\varrho)}$ is a stable distribution on $(\GSp_{2,2})_{0+}$. 

We now look at the characters on an element of the form $su$ for $u$ topologically unipotent. They follow from computations in the previous section \S\ref{subsec:character-nbhd-1}. 
\begin{enumerate}
    \item When $F=F_{C_2}$, by \cite[Theorem 4.2]{Deligne-Lusztig}, we have 
    \begin{align}\label{FC2-omega-princ-epsilon}
    \begin{split}
        \Ch_{\omega_{\princ}^{\epsilon}}(su)&=\frac{1}{4}\left(R_{S_{A_1\times A_1}}^{\epsilon}(su)-2R_{S_{A_1}}^{\epsilon}(su)+R_{S_1}^{\epsilon}(su)\right)\\
        &=(-1)^{\frac{q-1}{2}}\frac{1}{2}\left(\mathcal{Q}_{S_{A_1\times A_1}}^{A_1\times A_1}(u)-\mathcal{Q}_{S_{A_1\times 1}}^{A_1\times A_1}(u)-\mathcal{Q}_{S_{1\times A_1}}^{A_1\times A_1}(u)+\mathcal{Q}_{S_1}^{A_1\times A_1}(u)\right).
        \end{split}
    \end{align}
    \item When $F=F_{A_1\times A_1}$, we have 
    \begin{align}\label{FA1xA1-nonsc}
        \Ch_{\delta([\eta_2,\nu\eta_2],1)^{F_+}}(su)&=(-1)^{\frac{q-1}{2}}\cdot \frac{1}{2}\left(\mathcal{Q}_1^{F_{A_1\times A_1}}(u)\pm q^*\mathcal{G}_{\sign}(u)\right)\\
        \Ch_{\pi_{\alpha}(\eta_2;1)^{F_+}}(su)&=(-1)^{\frac{q+1}{2}}\cdot \frac{1}{2}\left(\mathcal{Q}_{A_1\times A_1}^{F_{A_1\times A_1}}(u)\pm q^*\mathcal{G}_{\sign}(u)\right)\label{FA1xA1-sc}
    \end{align}
\end{enumerate}
The following lemma is an analogue of \cite[Lemma 3.5.1]{G2-stability}.
\begin{lemma}\label{unstable-D-A1xA1-G-sgn}
    The distribution $D_{(F_{A_1\times A_1},\mathcal{G}_{\sgn})}$ on $\GSp_{2,2}$ is not stable. 
\end{lemma}
\begin{proof}
    A distribution on $\GSp_{2,2}(F)$ is stable if and only if it is stable under conjugation by $\GL_2(F)\times \GL_2(F)$. Thus all stable distributions on $\GSp_{2,2}$ must be restricted from invariant distributions on $\GL_2(F)\times \GL_2(F)$. But the only invariant distributions on $\GL_2(F)\times \GL_2(F)$ are spanned by semisimple orbital integrals, and $D_{(F_{A_1\times A_1},\mathcal{G}_{\sgn})}$ is linearly independent from them (as can be seen by evaluating against $\mathcal{G}_{\sgn}$). 
\end{proof}

\begin{prop}\label{main-stability-prop-GSp4}
Let $G=\GSp_4(F)$. For ramified quadratic characters $\eta_2$ and $\eta_2'$, the character $\Ch_{\delta([\eta_2,\nu\eta_2],\varrho)}+\Ch_{\pi_{\alpha}(\eta_2';\varrho)}$ is stable in a neighborhood of $s$ if and only if $\eta_2=\eta_2'$. \\
Thus, $\{\delta([\eta_2,\nu\eta_2],\varrho),\pi_{\alpha}(\eta_2;\varrho)\}$ is an $L$-packet, as dictated by Property \ref{property:atomic-stability}, for each ramified quadratic character $\eta_2$. 
\end{prop}
\begin{proof}
    This follows from the above computations \eqref{FC2-omega-princ-epsilon}, \eqref{FA1xA1-nonsc} and \eqref{FA1xA1-sc}, as well as Lemma \ref{unstable-D-A1xA1-G-sgn} that $D_{(F_{A_1\times A_1},\mathcal{G}_{\sgn})}$ is a non-stable distribution on $\GSp_{2,2}$. 
\end{proof}

Now, by \Cref{sp-gsp-functoriality}, functoriality for $\Sp_4\to\GSp_4$, we obtain the following corollary of Proposition \ref{main-stability-prop-GSp4}. Let $\pi_i(\eta_2)$ be as defined in \eqref{sp4-splits}. Let $\pi_{\alpha}^{\pm}(\eta_2)$ be as defined in Proposition~\ref{depth-zero-sc-Sp4}\eqref{depth-zero-sc-Sp4(2)}. 
\begin{cor}\label{2x4-llc}
Let $G=\Sp_4(F)$.~The character $\Ch_{\pi_1(\eta_2)}+\Ch_{\pi_2(\eta_2)}+\Ch_{\pi_{\alpha}^+(\eta_2)}+\Ch_{\pi_{\alpha
}^-(\eta_2)}$ is stable in a neighborhood of $s$, for each ramified quadratic character $\eta_2$. 
Thus we have the following explicit $L$-packets, as dictated by Property~\ref{property:atomic-stability}:
\[
    \Pi_{\varphi(\eta_2)}:=\{\pi_1(\eta_2),\pi_2(\eta_2),\pi^+_{\alpha}(\eta_2), \pi_{\alpha}^-(\eta_2)\},\]
    for each ramified quadratic character $\eta_2$.
\end{cor}
\begin{proof}
Indeed, by definition we have
\[
\delta([\widehat\chi_1,\nu\widehat\chi_1],1)|_{\Sp_4(F)}=\pi_1(\chi_1)+\pi_2(\chi_1)
\]
and
\begin{align}
\pi_\alpha(\eta_2;1)|_{\Sp_4(F)}&=\cInd_{G_\alpha}^{\Sp_4}(\omega_\cusp^{\eta_2})\\
&=\cInd_{G_\alpha}^{\Sp_4}(R_+'(\theta_0)\boxtimes (R_+'(\theta_0))^{\diag(\varpi,1)}+R_-'(\theta_0)\boxtimes (R_-'(\theta_0))^{\diag(\varpi,1)})\\
&=\pi_\alpha^+(\eta_2)+\pi_\alpha^-(\eta_2). 
\end{align}
The claim now follows from Proposition \ref{main-stability-prop-GSp4}.
\end{proof}

\section{Explicit \texorpdfstring{$L$}{L}-parameters}\label{sec:non-sc}

We construct $L$-parameters for each reducible induced representation in Theorem~\ref{gsp4-induced}. For representations that are not essentially tempered, we give explicit Langlands classifications, so by Property~\ref{langlands-class} we have explicit $L$-parameters (since LLC is known for Levis of $\GSp_4$). We only give the $L$-parameters for $\GSp_4$, but those for $\Sp_4$ follows by functoriality, Property~\ref{sp-gsp-functoriality}.

\subsection{Principal series for \texorpdfstring{$\GSp_4$}{GSp4}}\label{principal-series-section}

We proceed by considering Bernstein blocks: let $\mathfrak s=[T,\chi_1\otimes\chi_2\otimes\theta]$. Then by Remark~\ref{gsp4-levi-llc-remark} the dual of $\chi_1\otimes\chi_2\otimes\theta$ is the homomorphism $F^\times\to T^\vee(\C)$ given by $\widehat\theta^{-1}\diag(1,\widehat\chi_2^{-1},\widehat\chi_1^{-1},\widehat\chi_1^{-1}\widehat\chi_2^{-1})$, whose restriction $c^{\mathfrak s}$ to $\cO_F^\times$ is well-defined. Let $\mathcal J^{\mathfrak s}=\Cent_{G^\vee}(\mathrm{Im}(c^{\mathfrak s}))$ and let $J^{\mathfrak s}$ be the Langlands dual group. Then \cite{Roche-principal-series} gives a (non-canonical) isomorphism between $\mcH(G/\!/J_\chi,\chi_1\otimes\chi_2\otimes\theta)$ and $\mcH(J^{\mathfrak s}/\!/I^{\mathfrak s},1_{I^{\mathfrak s}})$, where $I^{\mathfrak s}$ is an Iwahori subgroup of $J^{\mathfrak s}$. There are the following cases (up to Weyl group conjugates):
\begin{enumerate}[label=({J\arabic*})]
    \item\label{J1} If $\chi_1=\chi_2=1$ then $\mathcal J^\mathfrak s=G^\vee$. Representations of the Iwahori-Hecke algebra are classified in \cite[Table~5.1]{Ram}.
    \item\label{J2} If $\chi_1\ne1$ and $\chi_2=1$ then $\mathcal J^\mathfrak s=\GL_2\times\GSp_0$ so $J^\mathfrak s=\GL_1\times\GSp_2$.  
    \item\label{J3} If $\chi_1=\chi_2^{-1}\ne1$ and $\chi_1^2=1$ then $\mathcal J^\mathfrak s=\{(g,h)\in\GL_2(\C):\det(g)=\det(h)\}$. Here $J^\mathfrak s=\GL_2(F)\times\GL_2(F)/F^\times$. Representations of the Iwahori-Hecke algebra are classified in \cite[Table~2.1]{Ram}.

    \item\label{J4} If $\chi_1=\chi_2^{-1}$ and $\chi_1^2\ne1$ on $\cO_F^\times$ then $\mathcal J^\mathfrak s=\GL_1\times\GSp_2$ so $J^\mathfrak s=\GL_2\times\GSp_0$. Representations of the Iwahori-Hecke algebra are classified in \cite[Table~2.1]{Ram}.
\end{enumerate}

We have the following cases:
\begin{itemize}
    \item In case~\ref{gsp4-1ai} the only essentially tempered representation is $\nu^{1/2}\chi_2\St_{\GL_2}\rtimes\theta$ where $e(\chi_2)=-\frac12$.
    \begin{itemize}
        \item if $\chi_2$ is unramified, we are in case~\ref{J1}. This is case $t_e$ in Table~5.1 of \cite{Ram} so the enhanced $L$-parameter is: $(\varphi_{\sigma,[1^4]},1),(\varphi_{\sigma,[2^2]},1)$.
        \item In case~\ref{J3}, when $\chi_2^2$ is unramified but $\chi_2$ is not, we have $J^\mathfrak s$ of type $A_1\times A_1$. This is case $t_a\times t_o$ in the notation of Table~2.1 of \cite{Ram} since the induced representation is of length $2$ with a tempered subquotient. Thus the enhanced $L$-parameter is $(\varphi_{\sigma,[1^4]},1),(\varphi_{\sigma,[2^2]},1)$.
    \item In case~\ref{J4}, when $\chi_2^2$ is ramified, we have $J^\mathfrak s=\GL_2\times\GSp_0$, of type $A_1$, which is case $t_a$ in \cite[Table~2.1]{Ram} so the $L$-parameter is $(\varphi_{\sigma,[1^4]},1),(\varphi_{\sigma,[2^2]},1)$
    \end{itemize}
    \item In case~\ref{gsp4-1aii} the only essentially tempered representation is $\chi_1\rtimes\nu^{1/2}\theta\St_{\GSp_2}$ for $e(\chi_1)=0$. Here $\mathfrak s=[\chi_1,1,\theta]$.
    \begin{itemize}
    \item In case~\ref{J1}, when $\chi_1$ is unramified, we have $J^\mathcal s=G^\vee$. This is case $t_e$ in Table~5.1 of \cite{Ram} so the enhanced $L$-parameters are: $(\varphi_{\sigma,[1^4]},1),(\varphi_{\sigma,[2^2]},1)$.
    \item In case~\ref{J2}, when $\chi_1$ is ramified, we have $J^\mathcal s=\GL_1\times\GSp_2$. This is case $t_a$ in \cite[Table~2.1]{Ram} so the $L$-parameters are $(\varphi_{\sigma,[1^4]},1),(\varphi_{\sigma,[2^2]},1)$
\end{itemize}
\item In case~\ref{gsp4-1aiii} the Steinberg representation corresponds to $(\varphi_{\sigma,[4]},1)$, with the regular unipotent.
\item In case~\ref{gsp4-1aiv} the representation $\delta([\chi_2,\nu\chi_2],\theta)$ is essentially square-integrable, living in the .
\begin{itemize}
    \item In case~\ref{J1}, when $\chi_2$ is the unramified quadratic character, we have $J^\mathfrak s=G^\vee$. This is case $t_a$ or $t_c$ in \cite[Table~5.1]{Ram}. To see which case we're in, note that $\delta([\eta_2,\nu\eta_2],\theta)^{G_{\delta+}}$  corresponds to $\St_{\GSpin_4}$ under Lusztig's equivalence $\cE(\GSp_4,\epsilon\otimes\epsilon\otimes\overline \theta)\cong\cE(\Cent_{\GSpin_5}(s),1)=\cE(\GSpin_4,1)$. Thus,
    \begin{align*}
    \dim\delta([\eta_2,\nu\eta_2],\theta)^I&=\langle\delta([\eta_2,\nu\eta_2],\theta)^{G_{\delta+}},R_T^1\rangle\\
    &=\langle\St_{\GSpin_4},R_T^1\rangle=1,
    \end{align*}
    and we are in case $t_a$ of \cite[Table~5.1]{Ram}. Thus the $L$-parameter of $\delta([\chi_2,\nu\chi_2],\theta)$ is $(\varphi_{\sigma,1},1)$, with trivial unipotent.
    \item In case~\ref{J4}, when $\chi_2$ is ramified, we have $J^\mathfrak s$ of type $A_1\times A_1$. This is case $t_a\times t_a$ in the notation of \cite[Table~2.1]{Ram}. Thus the $L$-parameters are: \[(\varphi_{\sigma,[1^4]},1),(\varphi_{\sigma,[2,1^2]},1),(\varphi_{\sigma,[2,1^2]},1),(\varphi_{\sigma,[2^2]},1).\]
    Here there is a slight abuse of notation; the two unipotents $[2,1^2]$ are embedded into $\cG_\varphi$ in different ways.
\end{itemize}
\item In case~\ref{gsp4-1bi}, where $\mathfrak s=[T,1\otimes1\otimes\theta]$, we have $J^\mathfrak s=G^\vee$. Here, there are two essentially tempered subquotients so we are in case $t_b$ of \cite[Table~5.1]{Ram}:
{\begin{center}
\begin{tabular}{ |c|c|c| } 
 \hline
 Indexing triple& nilpotent orbit &representation\\ \hline
 $(t_b,0,1)$&$[1^4]$&$J(\nu;1_{F^\times}\rtimes\theta)$\\
 $(t_b,e_{\beta},1)$&$[2^2]$&$J(\nu^{1/2}\St_{\GL_2};\theta)$\\
 $(t_b,e_{\alpha_1+\beta},-1)$&$[2,1^2]$&$\tau$\\
 $(t_b,e_{\alpha_1+\beta},1)$&$[2,1^2]$&$\tau'$\\
 \hline
\end{tabular}
\end{center}}
We again used that $\St_{\GL_2}$ corresponds to the regular unipotent under LLC for $\GL_2$.
\item In case~\ref{gsp4-1biii} the representation $\nu^{1/2}\chi_2\St_{\GL_2}\rtimes\theta$ is essentially tempered.

where $\mathfrak s=[T,\chi_1\otimes\chi_1\otimes\theta]$, with $\chi_1^2=1$, either:
\begin{itemize}
    \item In case~\ref{J1}, when $\chi_1=1$, we have $J^\mathfrak s=G^\vee$. Then we are in case $t_e$ of \cite[Table~5.1]{Ram} so the $L$-parameters are $(\varphi_{[1^4]},1)$ and $(\varphi_{[2^2]},1)$.
    \item In case~\ref{J4}, when $\chi_1\ne1$, we have $J^\mathfrak s$ of type $A_1\times A_1$. This is of type $t_a\times t_o$ in the notation of \cite[Table~2.1]{Ram} so the $L$-parameters are $(\varphi_{[1^4]},1)$ and $(\varphi_{[2^2]},1)$.
\end{itemize}
\end{itemize}

\subsection{Intermediate series for \texorpdfstring{$\GSp_4$}{GSp4}}

\begin{lemma}\label{weil-group-repn-lemma}
Let $\varphi$ be a $2$-dimensional irreducible semisimple representation of $W_F$. Then $\varphi|_{I_F}$ remains irreducible.
\end{lemma}
\begin{proof}
Suppose otherwise, that $\varphi|_{I_F}=\widehat\zeta_1\oplus\widehat\zeta_2$ for some characters $\widehat\zeta_i$ of $I_F$. Since $W_F$ acts trivially on $I_F^{\mathrm{ab}}\cong\cO_F^\times$, the group $W_F$ intertwines $\widehat\zeta_1\oplus\widehat\zeta_2$. Thus if $\widehat\zeta_1\ne\widehat\zeta_2$ then $\varphi$ also splits into two distinct characters, a contradiction, and if $\widehat\zeta_1=\widehat\zeta_2$ then $\varphi(w)$ for $w\in W_F$ such that $|w|=1$ can be diagonalized, which provides a splitting of $\varphi$. 
\end{proof}

\subsubsection{When $L=\GL_2\times\GSp_0$, i.e., case~\ref{gsp4-20-levi}} Let $\mathfrak s=[L,\pi\otimes\chi]$, where we assume $\omega_\pi=1$. By Remark~\ref{gsp4-levi-llc-remark}, local Langlands for the Levi gives an $L$-parameter $\widehat\chi^{-1}\otimes\widehat\chi^{-1}\varphi_\pi^\vee=\widehat\chi^{-1}(1\otimes\varphi_\pi^\vee)\colon W_F\to\GL_1(\C)\times\GSp_2(\C)$, whose restriction $c^\mathfrak s$ to $I_F$ is well-defined. The centralizer $\mathcal J^{\mathfrak s}:=\Cent_{G^\vee}(\mathrm{Im}(c^{\mathfrak s}))$ is independent of $\chi$. When $J^\mathfrak s$ is connected we have the bijection
\(
\mathrm{Irr}^{\mathfrak s}(G)\cong\mathrm{Irr}(\mcH(J^{\mathfrak s}//I^{\mathfrak s})),
\)
where the group of $F$-rational points on the Langlands dual of $\mathcal J^\mathfrak s$ and $I^\mathfrak s$ is an Iwahori subgroup.

By Lemma~\ref{weil-group-repn-lemma}, the restriction $\varphi|_{I_F}$ remainds irreducible, so $\mathcal J^\mathfrak s=\{(z,g)\in\C^\times\times\GSp_2(\C):\det(g)=z^2\}\cong\C^\times\times\SL_2(\C)$ so $J^\mathfrak s=F^\times\times\PGL_2(F)$. Since the induced representation is of length $2$, we are in case $t_a$ of \cite[Table~2.1]{Ram}, and the $L$-parameter for the tempered sub-representation is $(\varphi_{\sigma,[2,1^2]},1)$.

\subsubsection{When $L=\GL_1\times\GSp_2$, i.e., case~\ref{gsp4-12-levi}} Let $\mathfrak s=[L,\chi\otimes\pi]$. By Remark~\ref{gsp4-levi-llc-remark}, local Langlands for the Levi gives an $L$-parameter
\[
\varphi_\pi^\vee\otimes\det(\varphi_\pi^\vee)\widehat\chi^{-1}\colon W_F\to\GL_2(\C)\times\GSp_0(\C),
\]
whose restriction $c^\mathfrak s$ to $I_F$ is well-defined. The centralizer $\mathcal J^{\mathfrak s}:=\Cent_{G^\vee}(\mathrm{Im}(c^{\mathfrak s}))$ is independent of $\chi$. That is,
\[
\varphi_\pi^\vee\otimes\det(\varphi_\pi^\vee)\widehat\chi^{-1}(w)=\begin{pmatrix}\varphi_\pi^\vee(w)\\&\widehat\chi^{-1}(w)\varphi_\pi^\vee(w)\end{pmatrix}.
\]

The induced representation $\chi\rtimes\pi$ is irreducible only when a) $\chi=1_{F^\times}$ or b) $\chi=\nu^{\pm1}\xi_o$ where $\xi_o$ is of order two and $\xi_o\pi\cong\pi$. In either case $\widehat\chi\varphi_\pi=\varphi_\pi$, so the $I_F$-representation $c^\mathfrak s$ is simply $\diag(\varphi_\pi^\vee(w),\varphi_\pi^\vee(w))$.

Here, in the notation of \cite[\S2.1]{aubert-xu-Hecke-algebra},
    \[
    \mathfrak X_{\mathrm{nr}}(M,\pi):=\{\xi\in\mathfrak X_{\mathrm{nr}}(M):\xi\otimes\pi\cong\pi\}
    \]
    has order $1$ or $2$, since $\xi\otimes\pi\cong\pi$ implies $\xi^2\omega_\pi=\omega_\pi$. Moreover, $W(M,\mathcal O)$ is order $2$, since the Weyl group acts by $\chi\otimes\pi\mapsto\chi^{-1}\otimes\chi\pi$. Thus, $W(M,\pi,\mathfrak X_{\mathrm{nr}}(M))$ is of order $2$ or $4$, and by \cite{solleveld-hecke}, there is a bijection
    \[
    \mathrm{Irr}^{\mathfrak s}(G)\simeq\mathrm{Irr}(\C[\mathfrak X_{\mathrm{nr}}(M)]\rtimes\C[W(M,\pi,\mathfrak X_{\mathrm{nr}}(M))]).
    \]
    The Kazhdan-Lusztig triples can be computed by following the commutative diagram in Property~\ref{property:AMS-conjecture-7.8}.

\section{Main Theorem}
\subsection{Properties of LLC}\label{section-properties}
We assume for the rest of this paper that $p$ does not divide the order of the Weyl group. 

We now state a compatibility property of the LLC with supercuspidal supports. 
\begin{defn}\label{infinitesimal-defn} \cite{Vogan-LLC-1993}
The \emph{infinitesimal parameter} of an $L$-parameter $\varphi$ for $G$ is $\lambda_\varphi\colon W_F\to G^\vee$ defined by, for $w\in W_F$,
\begin{equation} \label{eqn:infinitesimal parameter}
\lambda_\varphi(w):=\varphi\left(w,\left(\begin{smallmatrix}|\!|w|\!|^{1/2}&0\\ 0&|\!|w|\!|^{-1/2}\end{smallmatrix}\right)\right)\quad\text{for any $w\in W_F$.}
\end{equation}
\end{defn}

\begin{property}\label{langlands-class}
Let $(P,\pi,\nu)$ be a standard triple for $G$. We have
\[
\varphi_{J(P,\pi,\nu)}=\iota_{L^\vee}\circ\varphi_{\pi\otimes\chi_\nu}.
\]
\end{property}

\begin{property} \label{property:size of L-packets} {\rm (\cite[\S2]{Arthur-Note}, and \cite[Conjecture~B]{Kaletha-LLC})} 
The elements of $\Pi_\varphi(G)$ are in bijection with $\Irr(S_\varphi)$. 
\end{property}

The following property is \cite[Conjecture~7.18]{Vogan-LLC-1993}, or equivalently \cite[Conjecture~5.2.2]{Haines-Bst}.
\begin{property}\label{infinitesimal-prop}
Let $P\subset G$ be a parabolic subgroup with Levi subgroup $L$, and $\sigma$ a supercuspidal representation of $L$. For any irreducible constituent $\pi$ of $\Ind_P^G\sigma$, the infinitesimal $L$-parameters $\lambda_{\varphi_\pi}$ and $\iota_{L^\vee}\circ\lambda_\sigma$ are $G^\vee$-conjugate.
\end{property}

\begin{numberedparagraph}
The following Property \ref{property:AMS-conjecture-7.8} generalizes Property \ref{infinitesimal-prop}.
Let $\mathcal{L}(G)$ be a set of representatives for the conjugacy classes of Levi subgroups of $G$. 
By \cite[Proposition~3.1]{ABPS-CM}, for any $L\in\cL(G)$ there is a canonical isomorphism 
\begin{equation}
    W_G(L)\xrightarrow{\sim}W_{G^\vee}(L^\vee).
\end{equation}
We set the following notations 
\begin{equation} \label{eqn:Gphis}
\rZ_{G^\vee}(\varphi):=\Cent_{G^\vee}(\varphi(W'_F))\quad\text{and}\quad\cG_\varphi:=\Cent_{G^\vee}(\varphi(W_F)).
\end{equation} 
We also consider the following component groups
\begin{equation} \label{eqn:A-phi}
A_\varphi:=\rZ_{G^\vee}(\varphi)/\rZ_{G^\vee}(\varphi)^\circ\quad\text{and}\quad
S_\varphi:=\rZ_{G^\vee}(\varphi)/\rZ_{G^\vee}\cdot\rZ_{G^\vee}(\varphi)^\circ.
\end{equation}
Recall that $A_{\cG_\varphi}(u_\varphi)$ denotes the component group of $\rZ_{\cG_\varphi}(u_\varphi)$. By \cite[\S~3.1]{Moussaoui-Bernstein-center},  
\begin{equation} \label{eqn:Au-iso}
A_\varphi\simeq A_{\cG_\varphi}(u_\varphi), \text{where $u_\varphi:=\varphi\left(1,\left(\begin{smallmatrix}
1&1\cr
0&1
\end{smallmatrix}\right)\right)$.}
\end{equation}

Let $(\varphi,\rho)$ be an enhanced $L$-parameter for $G$. Recall that $u_\varphi:=\varphi\left(1,\left(\begin{smallmatrix} 1&1\cr 0&1
\end{smallmatrix}\right)\right)$. Then $u_\varphi$ is a unipotent element of the (possibly disconnected) complex reductive group $\cG_\varphi$ defined in \eqref{eqn:Gphis}, and $\rho\in\Irr(A_{\cG_\varphi}(u_\varphi))$ by \eqref{eqn:Au-iso}. Let $\ft_\varphi:=(\cL^\varphi,(v^\varphi,\epsilon^\varphi))$ denote the cuspidal support of $(u_\varphi,\rho)$, i.e.~
\begin{equation} \label{eqn:Sc u,rho}
(\cL^\varphi,(v^\varphi,\epsilon^\varphi)):=\Sc_{\cG_\varphi}(u_\varphi,\rho).
\end{equation}
In particular, $(v^\varphi,\epsilon^\varphi)$ is a cuspidal unipotent pair in $\cL^{\varphi}$. 

Upon conjugating $\varphi$ with a suitable element of $\rZ_{\cG_\varphi^\circ}(u_\varphi)$, we may assume that the identity component of $\cL^\varphi$ contains $\varphi\left(\left(1, \left(\begin{smallmatrix} z&0\cr 0&z^{-1}
\end{smallmatrix}\right)\right)\right)$ for all $z\in\C^\times$. 
Recall that by the Jacobson–Morozov theorem (see for example \cite[\S~5.3]{Carter-book}), any unipotent element $v$ of $\cL^\varphi$ can be extended to a homomorphism of algebraic groups
\begin{equation} \label{eqn:JM}
j_v\colon \SL_2(\C)\to \cL^\varphi \text{ satisfying }j_v\left(\begin{smallmatrix}1&1\\0&1\end{smallmatrix}\right)=v.
\end{equation}
Moreover, by \cite[Theorem~3.6]{Kostant}, this extension is unique up to conjugation in $\Cent_{\cL^\varphi}(v)^\circ$. We shall call a homomorphism $j_v$ satisfying these conditions to be \textit{adapted to $\varphi$}.

By \cite[Lemma~7.6]{AMS18}, up to $G^\vee$-conjugacy, there exists a unique homomorphism $j_{v}\colon \SL_2(\C)\to \cL^{\varphi}$ which is adapted to $\varphi$, and moreover, the cocharacter
\begin{equation} \label{eqn:cocharacter}
\chi_{\varphi,v}\colon z\mapsto \varphi\left(1, \left(\begin{smallmatrix} z&0\cr 0&z^{-1}
\end{smallmatrix}\right)\right)\cdot j_v\left(\begin{smallmatrix} z^{-1}&0\cr 0&z
\end{smallmatrix}\right)
\end{equation}
has image in $\rZ_{\cL^\varphi}^\circ$. 
We define an $L$-parameter $\varphi_{v} \colon W_F \times\SL_2(\C) \to\rZ_{G^\vee}(\rZ_{\cL^\varphi}^\circ)$ by
\begin{equation} \label{eqn:varphi_v}
\varphi_v(w,x):= \varphi(w,1)\cdot\chi_{\varphi,v}(|\!|w|\!|^{1/2})\cdot j_v(x)\quad\text{for any $w\in W_F$ and any $x\in \SL_2(\C)$.}
\end{equation}
\begin{remark} \label{remark:infinitesimal parameter}
{\rm Let $w\in W_F$ and $x_w:=\left(\begin{smallmatrix}|\!|w|\!|^{1/2}&0\\ 0&|\!|w|\!|^{-1/2}\end{smallmatrix}\right)$. By \eqref{eqn:infinitesimal parameter}, we have
\begin{equation}
\begin{matrix}
&\lambda_{\varphi_v}(w)=\varphi_v(w,x_w)=
\varphi(w,1)\cdot\chi_{\varphi,v}(|\!|w|\!|^{1/2})\cdot j_v(x_w)\cr
&=\varphi(w,1)\cdot\varphi(1,x_w)\cdot j_v(x_w^{-1})\cdot j_v(x_w)=\varphi(w,x_w)=\lambda_\varphi(w).
\end{matrix}
\end{equation}
}\end{remark}
\begin{defn} \label{defn:cuspidal support} {\rm \cite[Definition~7.7]{AMS18}}
The \textit{cuspidal support} of $(\varphi,\rho)$ is 
\begin{equation} \label{eqn:Sc Galois side}
\Sc(\varphi,\rho):=(\rZ_{G^\vee}(\rZ_{\cL^\varphi}^\circ),(\varphi_{v^\varphi},\epsilon^\varphi)).
\end{equation}
\end{defn}

\begin{property}\label{property:AMS-conjecture-7.8}\cite[Conjecture 7.8]{AMS18}
The following diagram is commutative:
\[ \begin{tikzcd}
\Irr(G) \arrow{r}{\mathrm{LLC}_G} \arrow[swap]{d}{\Sc} & \Phi_e(G) \arrow{d}{\Sc} \\
\bigsqcup_{L\in\mathcal L(G)}\Irr_{\mathrm{scusp}}(L)/W_G(L) \arrow{r}{\bigsqcup\mathrm{LLC}_L}& \bigsqcup_{L\in\mathcal L(G)}\Phi_{e,\mathrm{cusp}}(L)/W_G(L)
\end{tikzcd}
\]
\end{property}

\begin{property} \label{property:L-packets} {\rm \cite[\S10.3]{Borel-Corvallis}}
Let $\varphi$ be an $L$-parameter for $G$.
\begin{enumerate}
\item $\varphi$ is bounded if and only if one element (equivalently any element) of $\Pi_\varphi(G)$ is tempered;
\item $\varphi$ is discrete  if and only if one element (equivalently any element) of $\Pi_\varphi(G)$ is square-integrable modulo center;
\item $\varphi$ is supercuspidal if and only if all the elements of $\Pi_\varphi(G)$ are supercuspidal.
\end{enumerate}
\end{property}

\begin{property}\label{property: Shahidi-fdeg}\cite{ShahidiAnnalsAproofof} The quantity $\frac{\fdeg(\pi)}{\dim(\rho)}$ is constant in an $L$-packet. 
\end{property}

\begin{property}\label{property:generic tempered L-packets}{\rm \cite[Conjecture~9.4]{ShahidiAnnalsAproofof}}
If $\varphi$ is bounded, then the $L$-packet $\Pi_\varphi(G)$ is ${\mathfrak w}$-generic for some Whittaker datum ${\mathfrak w}$. Moreover, the conjectural bijection $\iota_{\mathfrak w}\colon \Pi_\varphi(G)\to \Irr(S_\varphi)$ maps the $\mathfrak w$-generic representation to the trivial representation of $S_\varphi$.
\end{property}

\begin{conj} \label{conj:matching} {\rm \cite[Conjecture~2]{AMS18}} For any $\fs=[L,\sigma]_G\in\fB(G)$, the LLC for $L$ given by $\sigma\mapsto(\varphi_\sigma,\rho_\sigma)$ induces a bijection
\begin{equation} \label{eqn:matching}
    \Irr^{\fs}(G)\xrightarrow{\sim} \Phi_\enh^{\fs^{\vee}}(G),
\end{equation}
where $\fs^\vee=[L^\vee,(\varphi_\sigma,\rho_\sigma)]_{G^\vee}$.
\end{conj}

Conjecture~\ref{conj:matching} is proved  for split classical groups  
\cite[\S5.3]{Moussaoui-Bernstein-center}, for $\GL_n(F)$ and $\SL_n(F)$  \cite[Theorems 5.3 and 5.6]{ABPS-SL}, for principal series representations of split groups \cite[\S16]{ABPS-LMS}. 
For the group $\rG_2$, a bijection between $\Irr^{\fs}(G)$ and $\Phi_\enh^{\fs^{\vee}}(G)$ has been constructed in \cite[Theorem 3.1.19]{aubert-xu-Hecke-algebra}. 
For $\GSp_4(F)$ and $\Sp_4(F)$, one can easily verify 
the axioms in the Main Theorem of \cite{aubert-xu-Hecke-algebra}, and thus we have an isomorphism 
\begin{equation}\label{Bernstein-block-isom}
    \mathrm{Irr}^{\fs}(G)\xrightarrow{\sim}\Phi_\enh^{\fs^{\vee}}(G)
\end{equation}
for each Bernstein series $\mathrm{Irr}^{\fs}(G)$ of \textit{intermediate series}. On the other hand, the bijection \eqref{Bernstein-block-isom} holds for \textit{principal series} blocks thanks to \cite{Roche-principal-series,Reeder-isogeny,ABPS-KTheory,AMS18}.

\begin{property}[Functoriality]\label{sp-gsp-functoriality}
There is a commutative diagram
\[ \begin{tikzcd}
\Pi(\GSp_{2n})\arrow{r}{\mathrm{LLC}} \arrow{d} & \Phi(\GSp_{2n}) \arrow{d}{std} \\
\Pi(\Sp_{2n})\arrow{r}{\mathrm{LLC}}&\Phi(\Sp_{2n})
\end{tikzcd}
\]
Here, the left vertical arrow is a correspondence defined by the subset of $\Pi(\GSp_{2n})\times\Pi(\Sp_{2n})$ consisting of pairs $(\pi,\varpi)$ such that $\varpi$ is a constituent of the restriction of $\pi$ to $\Sp_{2n}$.
\end{property}

\begin{property}[Stability]\label{property:atomic-stability}
Let $\varphi$ be a discrete $L$-parameter. There exists a non-zero $\C$-linear combination 
\begin{equation}\label{stable-distribution-Lpacket}
S\Theta_{\varphi}:=\sum\limits_{\pi\in\Pi_{\varphi}}z_\pi\Theta_{\pi},\quad\text{for }z_{\pi}\in\C,
\end{equation}
which is stable. In fact, one can take $z_\pi=\dim(\rho_\pi)$ where $\rho_\pi$ is the enhancement of the $L$-parameter. Moreover, no proper subset of $\Pi_{\varphi}$ has this property. 
\end{property}
\end{numberedparagraph}

\subsection{Main Result}
Construction of the Local Langlands Correspondence
\begin{equation} \label{eqn:LLC}
\begin{split}
    \LLC\colon\mathrm{Irr}(G)&\xrightarrow{1\text{-}1}\Phi_\enh(G)\\
    \pi &\mapsto (\varphi_{\pi},\rho_{\pi}).
\end{split}
\end{equation}
Recall from \cite[(3.3.2)]{LLC-G2}
and \cite[(2.4.3)]{LLC-G2} 
that we have 
\begin{equation}
\Irr^{\mathfrak{s}}(G)=\bigsqcup\limits_{\fs\in \mathcal{B}(G)}\Irr^{\fs}(G)\;\;\text{and}\;\;
\Phi_\enh(G)=\bigsqcup\limits_{\mathfrak{s}^{\vee}\in \mathcal{B}^{\vee}(G)}\Phi_\enh^{\fs^{\vee}}(G).
\end{equation}
When $\pi\in\Irr(G)$ is not supercuspidal, we have $\fs=[L,\sigma]_G$ where $L$ is a proper Levi subgroup of $G$. Recall from \S\ref{Levi-section}, $L$ is conjugate to $\GL_1\times\GL_1\times\GSp_0$ (resp.~$\GL_1\times\GL_1\times\Sp_0$), $\GL_2\times\GSp_0$ (resp.~$\GL_2\times\Sp_0$) and $\GL_1\times \GSp_2$ (resp.~$\GL_1\times\Sp_2$). Let $\varphi_\sigma\colon W'_F\to L^\vee$ be the $L$-parameter attached to $\sigma$ by the Local Langlands Correspondence for $L$ (see \cite{Bushnell-Henniart-GL2, Labesse-Langlands}). 
The $L^\vee$-conjugacy class of $\varphi_\sigma$ is uniquely determined by $\sigma$, and one can easily check that $\varphi_{(\chi\circ\det) \otimes\sigma}=\varphi_\sigma\otimes \varphi_\chi$ (see for example \cite[Proposition 3.4.6]{Kaletha-nonsingular}), i.e.~\cite[Property~3.12(1)]{aubert-xu-Hecke-algebra} holds. This allows us to define 
\begin{equation}
\fs^\vee:=[L^\vee,(\varphi_\sigma,1)]_{G^\vee}.
\end{equation}
Let $\pi\mapsto (\varphi_\pi,\rho_\pi$) be the bijection 
\begin{equation}
\Irr^{\fs}(G)\xrightarrow{\sim}\Phi_\enh^{\fs^{\vee}}(G),
\end{equation}
established in \cite[Main Theorem]{aubert-xu-Hecke-algebra} (for intermediate series) and in \cite{ABPS-KTheory} (for principal series). We have given explicit Kazhdan-Lusztig triples and $L$-packets in $\mathsection$\ref{sec:non-sc}. 

We consider now the case where $\pi$ is supercuspidal. Hence we have $\fs=[G,\pi]_G$ for $\pi$ an irreducible supercuspidal representation of $G$. 

\begin{enumerate}
\item[(a)]
When $\pi$ is non-singular supercuspidal, we define $(\varphi_\pi,\rho_\pi)$
to be the enhanced $L$-parameter constructed in \cite{Kal-reg, Kaletha-nonsingular}. 

\item[(b)]
When $\pi$ is a unipotent supercuspidal representation of $G$, we define $(\varphi_\pi,\rho_\pi)$ to be the enhanced $L$-parameter constructed in \cite{Lu-padicI}, \cite[\S~5.6]{Morris-ENS} and \cite{Solleveld-LLC-unipotent} (see also \cite{solleveld-principal-series-2023}). 

$\bullet$ $x=\delta$: From~\S\ref{d0s} Proposition~\ref{depth-zero-sc-GSp4}\eqref{depth-zero-sc-GSp4(1)}, the reductive quotient $\bbG_{\delta}\cong\GSp_4(\F_q)$ has a unique unipotent cuspidal representation $\theta_{10}$, giving unipotent supercuspidals $\pi_\delta(\theta_{10}\otimes\chi)$ for each character $\chi$. Define the following $L$-parameter $\varphi(\eta;\chi)$ with unipotent $[2^2]$:
\[
\varphi(\eta;\chi):=\diag(\widehat\eta\widehat\chi,\widehat\chi,\widehat\chi,\widehat\eta\widehat\chi).
\]
By case~\ref{gsp4-galois-4b-iv} we have $\cG_\varphi\simeq\GSp_{2,2}(\C)$ and $S_\varphi\simeq\mu_2$. By the discussion in \S\ref{section-mixed-packets-GSp4}, we have $\varphi(\eta_2;\chi)=\varphi_{\delta([\eta_2,\nu\eta_2],\nu^{-1/2}\chi)}$.

\item[(c)]\label{summary-section-d0-mixed-packets}
Let $\pi$ be a non-unipotent depth-zero \textit{singular} supercuspidal representation of $G$. As recalled in~\eqref{eqn:depth zero supercuspidal}, we have
$\pi=\cInd_{G_{[x]}}^G\tau$, where $x$ is a vertex of the Bruhat-Tits building of $G$ and $\tau$ is inflated from a representation in the Lusztig series $\cE(\bbG_x,s)$ with $s\ne 1$. By Proposition~\ref{depth-zero-sc-GSp4}, We have two cases, where $x=\alpha$:

$\bullet$ From \S\ref{d0s} Proposition~\ref{depth-zero-sc-GSp4}\eqref{depth-zero-sc-GSp4(2)}, the reductive quotient $\bbG_{\delta}\cong\GSp_{2,2}(\F_q):=\{(g,h)\in\GL_2(\F_q)\times\GL_2(\F_q):\det(g)=\det(h)\}$ has a rational Lusztig series $\cE(\bbG_{x_1},s)$, where $s=(\lambda,\lambda)$ for some $\lambda\in\F_{q^2}$ such that $\lambda^{q-1}=-1$, with singular cuspidal representations $\omega_\cusp^{\eta_2}$. Let $\pi(\eta_2;\chi)$ denote the compact induction $\cInd_{G_\alpha\Cent}^{\GSp_4}(\omega_\cusp^{\eta_2}\otimes\chi)$, for each unramified character $\chi$ of $F^\times$. There are two (depth-zero) ramified cubic characters $\eta_2$ and $\eta_2'$ of $F^\times$. Define the following $L$-parameter with unipotent $[2^2]$:
\begin{equation} \label{eqn:Lpar_2_i}
    \varphi(\eta_2;\chi)|_{W_F}:=\diag(\widehat\eta_2\widehat\chi,\widehat\chi,\widehat\chi,\widehat\eta_2\widehat\chi).
\end{equation} 
By case~\ref{gsp4-galois-4b-iv} we have $\cG_{\varphi}\simeq\GSp_{2,2}(\C)$, the unipotent element $u$ is regular in $\cG_{\varphi}$, and $S_\varphi\simeq\mu_2$. By the discussion in \S~\ref{section-mixed-packets-GSp4}, we have $\varphi(\eta_2;\chi)=\varphi_{\delta([\eta_2,\nu\eta_2],\nu^{-1/2}\chi)}$, where $\delta([\eta_2,\nu\eta_2],\nu^{-1/2}\chi)$ is the unique discrete series subquotient of $\nu\eta_2\times\eta_2\rtimes\nu^{-1/2}\chi$.

By Proposition~\ref{main-stability-prop-GSp4}, we obtain two $L$-packets of size $2$, for each $i=1,2,3$,
\begin{equation}
    \Pi_{\varphi(\eta_2;\chi)}(G):=\{\pi(\eta_2';\chi),\delta([\eta_2,\nu\eta_2],\nu^{-1/2}\chi)\}.
\end{equation}

$\bullet$ From \S\ref{d0s} Proposition~\ref{depth-zero-sc-GSp4}\eqref{depth-zero-sc-GSp4(3)}, the reductive quotient $\bbG_{\alpha}\cong\GSp_{2,2}(\F_q):=\{(g,h)\in\GL_2(\F_q)\times\GL_2(\F_q):\det(g)=\det(h)\}$ has a cuspidal representation $R_T^\theta\boxtimes R_T^\theta$, where $T\subset\GL_2(\F_q)$ is an anisotropic maximal torus and $\theta$ is a character of $T$ such that $\theta^2$ is regular. This gives rise to the singular supercuspidal $\pi_{(S,\theta\boxtimes\theta)}$, where $\theta$ is a regular character of $E^\times$, for an unramified quadratic extension $E/F$ (see Definition~\ref{regular-supercuspidal-definition}). Let $\varphi_\theta$ be the $L$-parameter which is $\chi^2\oplus\Ind_{W_E}^{W_F}(\theta)$ as a $W_F$-representation, with unipotent $\SL_2(\C)$ acting on $\chi^2$.

Then by the discussion in \S\ref{section-mixed-packets-GSp4}, the $L$-packet is
\[
\Pi_{\varphi(\theta)}=\{\delta(\nu^{1/2}\pi_{(E^\times,\theta)}\rtimes\nu^{-1/2}\chi_1^{-1}),\pi_{(S,\theta\boxtimes\theta\otimes\widehat\chi_1^{-1})}\}.
\]

\item[(d)] Let $\pi$ be a positive-depth singular supercuspidal representation of $G$. As in \S\ref{section-mixed-packets-GSp4}, such a singular supercuspidal representation necessarily arises from a self-dual supercuspidal representation $\pi_u$ of $\PGL_2(F)$, via the following recipe:
\begin{itemize}
    \item $\pi_u$ is a supercuspidal representation of $\GL_2(F)$, which corresponds to a nontrivial representation $\mathrm{JL}(\pi_u)$ of $D^\times/F^\times$ under the Jacquet-Langlands correspondence, for $D/F$ the quaternion algebra. The Kim-Yu type is given by a twisted Levi sequence $(G^0\subset\cdots\subset G^d=D^\times/F^\times)$.
        \item $\pi$ has Kim-Yu type given by the twisted Levi sequence $(G^0\subset\cdots\subset G^d=D^\times/F^\times\subset\GSp_4(F))$.
\end{itemize}
It lives in a mixed $L$-packet together with $\delta(\nu^{1/2}\pi_u\rtimes\nu^{-1/2}\widehat\chi^{-1})$, the essentially tempered sub-representation of $\nu^{1/2}\pi_u\rtimes\nu^{-1/2}\widehat\chi^{-1}$. Letting $\varphi$ be the $L$-parameter $\chi^2\oplus V$ where $V$ is the $W_F$-representation corresponding to $\varphi_u$ under the LLC for $\PGL_2(F)$, with unipotent $[2,1^2]$. Then
\begin{equation}
\Pi_\varphi(G)=\{\pi,\delta(\nu^{1/2}\pi_u\rtimes\nu^{-1/2}\widehat\chi^{-1})\}
\end{equation}

\end{enumerate}
Let $G$ be the group of $F$-rational points of the groups $\Sp_4$ and $\GSp_4$. We suppose that the residual characteristic of $F$ is different from $2$. 
\begin{thm}\label{main-thm} 
The explicit Local Langlands Correspondence defined in \eqref{eqn:LLC} satisfies \eqref{Bernstein-block-isom-intro} for any $\fs\in\fB(G)$,
where $\fs^\vee=[L^\vee,(\varphi_\sigma,\rho_\sigma)]_{G^\vee}$, and also satisfies Properties~{\rm\ref{langlands-class}, \ref{property:size of L-packets},  \ref{property:AMS-conjecture-7.8}, \ref{property:L-packets}, 
\ref{property:generic tempered L-packets}}. Moreover, we have Property~\ref{property: Shahidi-fdeg} for depth-zero $L$-packets.\footnote{we certainly expect this property to hold for positive-depth $L$-packets as well.}

Moreover, Properties~{\rm\ref{langlands-class},
\ref{property:size of L-packets},  \ref{infinitesimal-prop}, \ref{property:AMS-conjecture-7.8},} and \ref{property:L-packets} (and Property~\ref{sp-gsp-functoriality} for $\Sp_4$) uniquely characterize our correspondence. 
\end{thm}
\begin{proof}
By Property~\ref{langlands-class}, the $L$-parameter $\varphi_\pi$ of each irreducible non-tempered representation $\pi$ of $G$ is uniquely determined. For $\GSp_4$, since the $L$-packets of the representations of the proper Levi subgroups of $G$ are all singletons, the $L$-packet $\Pi_{\varphi_\pi}(G)$ is a singleton.
Hence, by Property~\ref{property:size of L-packets}, we have $\rho_\pi=1$.
Thus the map 
\eqref{eqn:LLC} is uniquely characterized for non-tempered representations. 
This finishes the case of non-discrete series tempered representations. 

Property \ref{property:L-packets} holds for supercuspidal $L$-packets by \cite[Lemma~10.1.7]{LLC-G2}.~For the mixed $L$-packets, this can be seen directly from \S\ref{summary-section-d0-mixed-packets} and the lists \textit{loc.cit.}, where we specify which member in a given $L$-packet is generic. 

Since we have already treated the discrete series in \ref{summary-section-d0-mixed-packets}, we are done. For $\Sp_4(F)$, this follows from Property~\ref{sp-gsp-functoriality}. Finally, Property~\ref{property: Shahidi-fdeg} follows from the calculations in Sections~\ref{sec:group-side} and \ref{section-mixed-packets-GSp4}, as in \cite{LLC-G2}. Note that we fix a Whittaker datum for $\Sp_4(F)$ as in \cite{AMS-2023} (see also \cite{solleveld-principal-series-2023}).
\end{proof}

\appendix 
\section{Applications to the Taylor-Wiles method}\label{TW-appendix}
\textit{In this appendix, we adopt notations consistent with standard literature on this topic, though these notations may differ slightly from our main text.} 

We apply the theory developed in \cite{whitmore2022taylorwiles}, which gives a generalized Taylor-Wiles method (see for example \cite{Thorne-BKvanishing}) using input from (explicit) Local Langlands Correspondences (e.g.~\cite{Roberts-Schmidt-book}), except that we are now equipped with our explicit Local Langlands Correspondence \eqref{main-thm-bijection}
\begin{align}
    \LLC_{\mathrm{SX}}:
    \pi \mapsto (V_{\pi},N_{\pi}).
\end{align}

Here we switch to the notation $(V_{\pi},N_{\pi})$ \textit{loc.cit.}~instead of our original notations in \eqref{main-thm-bijection}. We work with $\overline{\Q}_p$-coefficients by fixing an isomorphism $\iota:\C\xrightarrow{\sim}\overline{\Q}_p$ compatible with the choice of $q_v^{1/2}$ as \textit{loc.cit.} As in \cite{BCGP}, we view $\LLC$ as sending an equivalence class of a smooth irreducible $\overline{\Q}_p$-valued representation of $\GSp_4(F_v)$ to a Weil-Deligne representation of $W_{F_v}$ valued in $\GSp(\overline{\Q}_p)$. 

Let $\overline{g}\in \hat{T}(k)$ for a split maximal torus $\hat{T}$ contained in a Borel subgroup $\hat{B}$ of $\hat{G}$. Let $M_{\overline{g}}:=\Cent_{\hat{G}_k}(\overline{g})$ be the scheme-theoretic centralizer of $\overline{g}$. 

Suppose that $q_v\equiv 1\mod p$. 
Our explicit LLC gives the following ``local lemmas'' \cite[Propositions 5.18, 5.19]{whitmore2022taylorwiles}, which are analogues for $\GSp_4$ of \cite[Proposition 3.13]{Thorne-BKvanishing}. 
\begin{prop}[Whitmore]\label{localized-p1-invariant}
    Let $\pi$ be an admissible irreducible $\overline{\Q}_p[G(F_v)]$-module such that $(\pi^{\p_1})_{\mathfrak{n}_1}\neq 0$. Then (1) $\pi$ is a subquotient of a parabolically induced representation $i_B^G\chi$ for some tamely ramified smooth character $\chi: T(F_v)\to \overline{\Z}_p^{\times}$. (2) The characters through which $\Oo[T/T\cap \p_1]^{W_L}$ acts on $\pi^{\p_1}$ are $W_G$-conjugates of $\chi$ and there exists $w\in W_G$ such that $w\chi$ lifts $\overline{\chi}$. (3) The localized invariants $(\pi^{\p_1})_{\mathfrak{n}_1}$ are $1$-dimensional and the action of $\Oo[T/(T\cap \p_1)]^{W_F}$ is through $w\chi$. (4) Finally, if $\LLC_p(\pi)=(V_{\pi},N_{\pi})$ is the Weil-Deligne representation associated to $\pi$ under the Local Langlands Correspondence \eqref{main-thm-bijection}, then $N_{\pi}=0$.
\end{prop}
\begin{proof}
Statements (1)--(3) follow from \cite[Lemma 5.16]{whitmore2022taylorwiles}. To verify (4), one works case by case according to $M_{\overline{g}}$ up to conjuacy. 
\begin{itemize}
    \item Suppose that $\overline{g}$ is regular semisimple. In this case, $L$ is a maximal torus and $\pi$ is an irreducible principal series $\chi_1\times\chi_2\rtimes\sigma$. Then by \S\ref{sec:Galois-side} Case~\eqref{gsp4-galois-4e}, we have $N_{\pi}=0$.
    \item Suppose that $M_{\overline{g}}$ is conjugate to a Levi subgroup of the Klingen parabolic subgroup $\GL_2(\C)\times \GSp_0(\C)$. In this case, we claim that $\pi$ cannot be conjugate to a representation of the form $\chi\St_{\GL_2}\rtimes\chi'$ for some smooth characters $\chi$ and $\chi'$, otherwise $(\pi^{\p_1})_{\mathfrak{n}_1}=0$. This can be seen by first applying the geometric lemma in \cite{Bernstein-Zelevinsky} along with \cite[Lemma 5.15]{whitmore2022taylorwiles}. Then by our classification \S\ref{sec:Galois-side} Case~\eqref{gsp4-galois-4c}, we have $N_{\pi}=0$. 
    \item Suppose that $M_{\overline{g}}$ is conjugate to a Levi subgroup of the Siegel parabolic $\GL_1(\C)\times\GSp_2(\C)$. In this case, $L$ is conjugate to a Levi subgroup of the Klingen parabolic $\GL_1(F)\times\GSp_0(F)$. We claim that $\pi$ cannot be conjugate to a representation $\chi\rtimes\chi'\St_{\GSp_2}$; otherwise, similar to the previous bullet point, we get $(\pi^{\p_1})_{\mathfrak{n}_1}=0$ which is a contradiction. Then by \S\ref{sec:Galois-side} Case~\eqref{gsp4-galois-4d}, we have $N_{\pi}=0$.
    \item The remaining case is when $L=G$. By \S\ref{sec:Galois-side} Case~\eqref{gsp4-galois-4a}, we have $N_{\pi}=0$. 
\end{itemize}
\end{proof}

The following proposition is an analogue of Proposition \ref{localized-p1-invariant} for representations with nonzero localized $\p$-invariants (instead of $\pi_1$-invariants). 
\begin{prop}[Whitmore]\label{localized-p-invariant}
    Let $\pi$ be an admissible irreducible $\overline{\Q}_p[G(F_v)]$-module such that $(\pi^{\p})_{\mathfrak{n}_0}\neq 0$. Then (1) $\pi$ is a subquotient of a parabolically induced representation $i_B^G\chi$ for some tamely ramified smooth character $\chi: T(F_v)\to \overline{\Z}_p^{\times}$. (2) The characters through which $\Oo[T/T(\Oo_{F_v})]^{W_L}$ acts on $\pi^{\p}$ are $W_G$-conjugates of $\chi$ and there exists $w\in W_G$ such that $w\chi$ lifts $\overline{\chi}$. (3) The localized invariants $(\pi^{\p})_{\mathfrak{n}_0}$ are $1$-dimensional and the action of $\Oo[T/(T(\Oo_{F_v}))]^{W_L}$ is through $w\chi$. (4) Finally, if $\LLC_p(\pi)=(V_{\pi},N_{\pi})$ is the Weil-Deligne representation associated to $\pi$ under the Local Langlands Correspondence \eqref{main-thm-bijection}, then $N_{\pi}=0$ and (5) there is an isomorphism of $\Oo[T/T(\Oo_{F_v})]^{W_G}$-modules $(\pi^{\p})_{\mathfrak{n}_0}\xrightarrow{\sim} \pi^{\mathfrak{g}}$.
\end{prop}
\begin{proof}
  Representations with Iwahori-fixed vectors are classified in \S\ref{principal-series-section}, and we attach explicit $L$-parameters.
\end{proof}
Proposition \ref{localized-p1-invariant} is then applied in \cite[Theorem 7.7]{whitmore2022taylorwiles} to a certain $\pi_v$ for some cuspidal automorphic representation $\pi$ of $\GSp_4(\A_f)$ and $v\in Q$ a Taylor-Wiles place, where $Q$ is part of a Taylor-Wiles datum $(Q,\{(\hat{T}_v,\hat{B}_v)\}_{v\in Q})$ as in 
\cite[Definition 3.9]{whitmore2022taylorwiles}, thus giving the existence of Galois representations associated to a classical weight cuspidal automorphic representation $\pi$. Combined with the patching criterion of \cite[Proposition 7.10.1]{BCGP}, one can then construct the patched modules as in \cite{BCGP} and \cite[7.11]{whitmore2022taylorwiles} to deduce modularity lifting theorems for abelian surfaces.

\bibliographystyle{amsalpha}
\bibliography{bibfile}

\end{document}